\documentclass[a4paper,10pt]{article}

\usepackage[english]{babel}

\usepackage{amssymb,amsmath,amsthm}
\usepackage{amsfonts}
\usepackage{color}
\usepackage{hyperref}

\numberwithin{equation}{section}

\theoremstyle{plain}
\newtheorem{theorem}{Theorem}[section]
\newtheorem{proposition}[theorem]{Proposition}
\newtheorem{lemma}[theorem]{Lemma}
\newtheorem{corollary}[theorem]{Corollary}

\newtheorem{definition}{Definition}[section]

\theoremstyle{definition}
\newtheorem{remarkbase}[theorem]{Remark}
\newenvironment{remark}{\pushQED{\qed} \remarkbase}{\popQED\endremarkbase}

\newcommand{\mA}{\mathcal{A}}
\newcommand{\mB}{\mathcal{B}}
\newcommand{\mC}{\mathcal{C}}

\newcommand{\mE}{\mathcal{E}}
\newcommand{\mF}{\mathcal{F}}
\newcommand{\mG}{\mathcal{G}}
\newcommand{\mH}{\mathcal{H}}

\newcommand{\mL}{\mathcal{L}}
\newcommand{\mM}{\mathcal{M}}
\newcommand{\mN}{\mathcal{N}}

\newcommand{\mS}{\mathcal{S}}

\newcommand{\mR}{\mathcal{R}}
\newcommand{\mT}{\mathcal{T}}

\renewcommand{\a}{\alpha}
\renewcommand{\b}{\beta}
\newcommand{\g}{\gamma}
\renewcommand{\d}{\delta}

\newcommand{\e}{\varepsilon}
\newcommand{\ph}{\varphi}
\newcommand{\vphi}{\varphi}
\newcommand{\lm}{\lambda}
\newcommand{\Om}{\Omega}
\newcommand{\om}{\omega}
\newcommand{\p}{\pi}
\newcommand{\s}{\sigma}
\renewcommand{\t}{\tau}
\renewcommand{\th}{\vartheta}

\renewcommand{\t}{\tau }

\newcommand{\teta}{\theta}

\newcommand{\be}{\begin{equation}}
\newcommand{\ee}{\end{equation}}

\newcommand{\ii}{{\mathrm i} }

\newcommand{\gr}{\nabla}

\newcommand{\la}{\langle}
\newcommand{\ra}{\rangle}

\newcommand{\wtilde}{\widetilde}

\newcommand{\R}{\mathbb R}
\newcommand{\C}{\mathbb C}
\newcommand{\Z}{\mathbb Z}
\newcommand{\N}{\mathbb N}
\newcommand{\T}{\mathbb T}

\newcommand{\inv}{^{-1}}

\newcommand{\pa}{\partial}

\newcommand{\fracchi}{{\mathfrak{I}}}

\newcommand{\Lipg}{{\mathrm{Lip}(\g)}}
\newcommand{\lip}{\mathrm{lip}}

\begin{document}

\title{\textbf{KAM for autonomous quasi-linear \\
perturbations of KdV}} 

\date{}
\author{Pietro Baldi, Massimiliano Berti, Riccardo Montalto}

\maketitle

\noindent
{\bf Abstract.}
We prove the existence and stability of Cantor families of quasi-periodic, small amplitude
solutions of {\it quasi-linear} (i.e.  {\it strongly nonlinear}) autonomous Hamiltonian perturbations of KdV.

\smallskip


\smallskip


\tableofcontents

\section{Introduction and main results}

In this paper we prove the existence and  stability of Cantor families of  quasi-periodic solutions of  Hamiltonian 
{\it quasi-linear} (also called  ``{\it strongly nonlinear}'', e.g. in \cite{k1}) perturbations of the KdV equation 
\begin{equation}\label{kdv quadratica}
u_t + u_{xxx} - 6 u u_x + {\mathcal N}_4 (x, u, u_x, u_{xx}, u_{xxx}) = 0\,, 
\end{equation}
under periodic boundary conditions $ x \in \T := \R / 2 \p \Z$, 
where 
\begin{equation}\label{qlpert}
{\mathcal N}_4 (x, u, u_x, u_{xx}, u_{xxx}) := - \partial_x \big[ (\partial_u f)(x, u,u_x)  - \partial_{x} ((\partial_{u_x} f)(x, u,u_x)) \big]   
\end{equation}
is the most general quasi-linear Hamiltonian (local) nonlinearity. 
Note that $ {\cal N}_4 $ contains as many derivatives as the linear part $ \pa_{xxx} $.
The equation \eqref{kdv quadratica}  is the Hamiltonian PDE
$ u_t = \partial_x  \nabla H(u)  $
where $ \nabla H $ denotes the $L^2(\T_x)$  gradient of the  Hamiltonian 
\begin{equation} \label{Ham in intro}
H (u) = \int_\T  \frac{u_x^2}{2} + u^3 + f(x, u,u_x) \, dx 
\end{equation}
on the real phase space   
\begin{equation}\label{def phase space}
H^1_0 (\T_x) := \Big\{ u(x ) \in H^1 (\T, \R) \ : \ \int_{\T} u(x) dx = 0  \Big\} \, . 
\end{equation}
We assume  that  the ``Hamiltonian density" 
$ f \in C^q (\T \times \R \times \R; \R ) $  
for some $ q $ large enough, and that 
\begin{equation}\label{order5}
f = f_5(u,u_x) + f_{\geq 6}(x,u,u_x) \,,
\end{equation}
where $ f_5 (u,  u_x)  $ denotes the homogeneous component of $ f $ of degree 5  and 
$ f_{\geq 6} $ collects all the higher order terms. 
By \eqref{order5} the nonlinearity $ {\mathcal N}_4 $ vanishes of order $ 4 $ at $ u = 0 $ and 
 \eqref{kdv quadratica}
may be seen, close to the origin, as  a ``small" perturbation of the KdV equation
\begin{equation}\label{KdVmKdV} 
u_t + u_{xxx} - 6 u u_x = 0 \, , 
\end{equation}
which is completely integrable. 
Actually, the KdV equation \eqref{KdVmKdV} may be described by global analytic action-angle variables, see \cite{KaP} and the references therein.

A natural question  is  to know whether the  periodic, quasi-periodic or almost periodic solutions of \eqref{KdVmKdV}
persist under small perturbations.  This is the content of KAM theory. 

\smallskip

The first KAM results for PDEs 
have been obtained 
for   $1$-d  semilinear Schr\"odinger and wave equations 
by Kuksin \cite{Ku}, Wayne \cite{W1}, 
Craig-Wayne \cite{CW}, P\"oschel \cite{Po2}, see \cite{C}, \cite{k1} and references therein. 
For PDEs in higher space dimension the theory has been more recently extended 
by Bourgain \cite{B5}, Eliasson-Kuksin \cite{EK}, and Berti-Bolle
\cite{BB13JEMS}, Geng-Xu-You \cite{GXY}, Procesi-Procesi \cite{PP}-\cite{PP1}, Wang \cite{Wang}. 

\smallskip
For {\it unbounded} perturbations the first KAM results 
have been proved by Kuksin \cite{K2} and Kappeler-P\"oschel \cite{KaP}
for KdV (see also Bourgain \cite{B96}), and more recently by  Liu-Yuan \cite{LY}, Zhang-Gao-Yuan \cite{ZGY} for derivative NLS, and by Berti-Biasco-Procesi \cite{BBiP1}-\cite{BBiP2} for derivative NLW.
For a recent survey of known results for 
KdV, we refer to \cite{K13}. 

The KAM theorems in \cite{K2}, \cite{KaP} prove the persistence 
of the finite-gap solutions of the integrable KdV 
\eqref{KdVmKdV} under semilinear 
Hamiltonian perturbations $ \e \partial_{x} (\partial_u f) (x, u) $, 
namely when the density  $ f $ is independent of
$ u_x $, so that \eqref{qlpert} is a differential operator of order $ 1 $  (note that in \cite{k1} such 
nonlinearities are called ``quasi-linear" and \eqref{qlpert} ``strongly nonlinear"). 
The key point is that the frequencies of KdV grow as $ \sim j^3 $ and
the difference $ |j^3  - i^3| \geq (j^2 + i^2)/2 $, $i \neq j $, so that  KdV gains (outside the diagonal) two 
derivatives.
This approach also works  for Hamiltonian pseudo-differential perturbations of order $ 2 $ (in space),
using the improved Kuksin's lemma in \cite{LY}. 
However it does {\it not} work for a general quasi-linear 
 perturbation as in \eqref{qlpert},
which is a nonlinear differential operator of the {\it same} order (i.e.\ 3) as the constant coefficient linear operator $ \partial_{xxx}$.
Such a strongly nonlinear perturbation term makes the KAM question quite delicate 
because of the possible phenomenon of formation of singularities in finite time, see 
Lax \cite{Lax}, Klainerman-Majda \cite{KM} for quasi-linear wave equations,
see also section 1.4 of \cite{k1}. 
For example, Kappeler-P\"oschel \cite{KaP} (Remark 3, page 19) wrote: 
``{\it It would be interesting to obtain perturbation results which also include terms of higher order, at least in the region where the KdV approximation is valid. However, results of this type are still out of reach, if true at all}''.

\smallskip

In this paper we give the first positive answer to this problem, proving the existence of 
small amplitude, linearly stable, quasi-periodic solutions of \eqref{kdv quadratica}, see Theorem \ref{thm:KdV}.  
Note that  \eqref{kdv quadratica} does not depend on external parameters.
Moreover the KdV equation \eqref{kdv quadratica} is a {\it completely resonant} PDE, 
namely the linearized equation at the origin is the linear Airy equation $ u_t + u_{xxx} = 0 $, 
which possesses only the $ 2 \pi $-periodic in time solutions
\begin{equation}\label{Airyper}
u(t,x) = {\mathop \sum}_{j \in \Z \setminus \{0\} } u_j {\rm e}^{\ii j^3 t} e^{\ii jx } \, .
\end{equation}
Thus the existence of quasi-periodic solutions of \eqref{kdv quadratica} is a purely nonlinear phenomenon 
(the diophantine frequencies in \eqref{solution u} are $ O(|\xi|) $-close to integers with $ \xi \to 0 $)
and a perturbation theory is more difficult. 

The solutions that we find are localized in Fourier space close to  finitely many {``tangential sites''}  
\begin{equation}  \label{tang sites}
S^+ := \{ \bar\jmath_1, \ldots, \bar\jmath_\nu \}\,, \quad 
S := S^+ \cup (- S^+) = \{ \pm j : j \in S^+ \}\,, 
\quad {\bar \jmath}_i \in \N \setminus \{0\}  \, , \quad \forall i =1, \ldots, \nu \,.
\end{equation}
The set $ S $ is required to be even because the solutions $ u $ of  \eqref{kdv quadratica} have to be real valued. 
Moreover, we also assume the following explicit hypotheses on $ S $: 
\begin{itemize}
\item {\sc  $ ({\mathtt S}1) $} 
$j_1 + j_2 + j_3 \neq 0$ for all $j_1, j_2, j_3 \in S$. 
\item {\sc  $ ({\mathtt S}2) $}  $ \nexists j_1, \ldots, j_4 \in S $ such that   
$ j_1 + j_2 + j_3 + j_4 \neq 0 $, $ j_1^3 + j_2^3 + j_3^3 + j_4^3 - (j_1 + j_2 + j_3 + j_4)^3 = 0 $. 
\end{itemize}
\begin{theorem}  \label{thm:KdV}
Given $ \nu \in \N $,  
let $ f  \in C^q $ (with $ q := q(\nu) $ large enough) satisfy \eqref{order5}.
Then, for all the  tangential sites $ S $ as in \eqref{tang sites} satisfying
$ ({\mathtt S}1) $-$ ({\mathtt S}2) $,   
the KdV equation \eqref{kdv quadratica} possesses small amplitude quasi-periodic solutions 
with diophantine frequency vector $\om := \om(\xi) = (\om_j)_{j \in S^+} \in \R^\nu$, 
of the form 
\begin{equation}\label{solution u}
u(t,x) = {\mathop\sum}_{j \in S^+} 2 \sqrt{\xi_j} \, \cos( \om_j t + j x) + o( \sqrt{|\xi|} ), \quad 
\om_j := j^3 - 6 \xi_j j^{-1} \,,
\end{equation}
for a ``Cantor-like" set of small amplitudes $ \xi \in \R^\nu_+ $ with density $ 1 $ at $ \xi = 0 $.
The term $ o( \sqrt{|\xi|} ) $ is small in some $ H^s $-Sobolev norm, $ s < q $. 
These quasi-periodic solutions are {\sl linearly stable}.  
\end{theorem}

This result is deduced from Theorem \ref{main theorem}. Let us make some comments. 
\begin{enumerate}
\item 
The set of tangential sites $ S $  satisfying  $({\mathtt S}1)$-$({\mathtt S}2) $ 
can be  iteratively constructed in an explicit way, see the end of section \ref{sec:NM}. 
After fixing  $  \{\bar\jmath_1, \ldots, \bar\jmath_n \} $, 
in the choice of $\bar\jmath_{n+1}$ there are only finitely many forbidden values,
while all the other infinitely many values are good choices for $\bar\jmath_{n+1}$. 
In this precise sense the set $ S $ is ``generic''.  
\item 
The linear stability of the quasi-periodic solutions is discussed after \eqref{sistema lineare dopo}.
In a suitable set of symplectic coordinates $ (\psi, \eta, w) $, $ \psi \in \T^\nu $,  near the invariant torus,
 the linearized equations at the quasi-periodic solutions assume the form 
\eqref{sistema lineare dopo}, \eqref{vjmuj}. Actually there is a complete KAM normal form near the 
invariant torus (remark \ref{rem:KAM normal form}), see also \cite{BB13}. 
\item 
A similar result holds for perturbed (focusing/defocusing) mKdV equations
\be\label{mKdV}
u_t + u_{xxx} \pm \partial_x u^3 + {\mathcal N}_4 (x, u, u_x, u_{xx}, u_{xxx})  = 0 
\ee
for  tangential sites $ S $  which satisfy 
$ \frac{2}{2\nu - 1} \sum_{i = 1}^\nu \bar \jmath_i^{\,2} \notin \Z $.
If the density $  f(u, u_x) $ is independent on $ x $, the result holds for {\it all} the choices of the tangential sites.  
The KdV equation \eqref{kdv quadratica} is more difficult  than 
 \eqref{mKdV} because the nonlinearity is quadratic and not cubic.

An important point is that the fourth order Birkhoff normal form of KdV and mKdV is completely integrable. 
The present strategy of proof --- that we describe in detail below --- is a rather general 
approach for constructing small amplitude quasi-periodic solutions of quasi-linear perturbed KdV equations. For example it could be applied to generalized KdV equations with 
leading nonlinearity $ u^p $, $ p \geq 4 $, by using 
the normal form techniques of Procesi-Procesi \cite{PP1}-\cite{PP}. 
A further interesting open question concerns perturbations of the finite gap solutions of KdV.
\end{enumerate}

Let us describe the strategy of proof of Theorem \ref{thm:KdV}, which involves many different arguments. 
\\[2mm]
{\it Weak Birkhoff normal form.} 
Once the finite set of tangential sites $ S $ has been fixed, 
the first step is to perform  
a ``weak" Birkhoff normal form (weak BNF), 
whose goal is to find an invariant manifold of solutions of the third order approximate KdV equation \eqref{kdv quadratica}, 
on which the dynamics is completely integrable,  see section \ref{sec:WBNF}.  Since the  KdV  nonlinearity is quadratic, two steps of weak BNF are required. 
The present Birkhoff map is close to the  identity up to  {\it finite dimensional} operators, see Proposition \ref{prop:weak BNF}. 
The key advantage is that it modifies  $ {\mathcal N}_4 $ very mildly, only up to finite dimensional operators 
(see for example Lemma \ref{lemma astratto potente}), and thus the spectral analysis of the linearized equations 
(that we shall perform in section \ref{operatore linearizzato sui siti normali}) is essentially the
same as if we were in the original coordinates.

The weak normal form \eqref{widetilde cal H} does not remove (or normalize) the monomials  $ O(z^2) $. 
This could be done. However,
we do not perform such stronger normal form (called  ``partial BNF" in P\"oschel \cite{Po3})
because the corresponding 
Birkhoff map is close to the identity only up to an operator of order $ O(\partial_x^{-1}) $, and so it would 
produce, in the transformed vector field 
$ {\mathcal N}_4 $,  terms of order $ \partial_{xx} $ and  $ \partial_x $.
A fortiori, we cannot either use the full Birkhoff normal form computed
in \cite{KaP} for KdV, which completely diagonalizes the fourth order terms, because 
such Birkhoff map is only close to the identity up to a bounded operator. 
For the same reason, we do not use the global nonlinear Fourier transform in \cite{KaP} (Birkhoff coordinates), which is close to the Fourier transform up to smoothing operators of order $ O(\partial_x^{-1}) $.

The weak BNF procedure of section \ref{sec:WBNF}  is sufficient 
to find the first nonlinear (integrable) approximation of the solutions and 
to extract the  ``frequency-to-amplitude'' modulation  \eqref{mappa freq amp}.  

In Proposition \ref{prop:weak BNF}
we also remove  the terms   $O(v^5)$, $O(v^4 z)$ in order to have  sufficiently good approximate solutions  
so that  the Nash-Moser iteration of section \ref{sec:NM} will converge. This is necessary for KdV
whose nonlinearity is quadratic at the origin. 
These further steps of Birkhoff normal form are not required if the nonlinearity  
is yet cubic as for mKdV, see Remark \ref{remark:cubic}. 
To this aim, we choose the tangential sites $S$ such that $({\mathtt S}2)$ holds. 
We also note that we assume  \eqref{order5} because 
we use the conservation of momentum up to the homogeneity order 5, see \eqref{cons mom}.
 \\[1mm]  
 {\it Action-angle and rescaling.}
At this point we introduce action-angle variables on the tangential sites (section \ref{sec:4})
and, after the rescaling \eqref{rescaling kdv quadratica}, 
we  look for quasi-periodic solutions of the 
Hamiltonian \eqref{formaHep}.
Note that the coefficients of the  normal form $ {\cal N } $ in \eqref{Hamiltoniana Heps KdV}
depend on the angles $ \theta $, unlike the usual KAM theorems \cite{Po3}, \cite{Ku},
where the whole normal form is reduced to constant coefficients.
This is because the weak BNF of section \ref{sec:WBNF} did {\it not} normalize the quadratic terms $ O(z^2) $. 
These terms are dealt with the ``linear Birkhoff normal form'' (linear BNF)
in sections \ref{BNF:step1}, \ref{BNF:step2}.
In some sense here the ``partial" Birkhoff normal form of \cite{Po3} is split into the weak BNF of section \ref{sec:WBNF} and the linear BNF of sections \ref{BNF:step1}, \ref{BNF:step2}.
 
The action-angle variables are convenient for 
proving the stability of the solutions. 
\\[1mm]  
 {\it The nonlinear functional setting.}
We look for a zero of the nonlinear operator \eqref{operatorF}, whose unknown is the
embedded torus and the frequency $ \om $ is seen as an ``external" parameter.
The solution is obtained  by a Nash-Moser iterative scheme in Sobolev scales. 
The key step  is to  construct  (for $ \omega  $ restricted to a suitable Cantor-like set) 
 an approximate inverse ({\it \`a la} Zehnder \cite{Z1}) of the linearized operator
at any approximate solution. Roughly, this means to find a linear operator which 
is an inverse  at an exact solution. A major difficulty is that the tangential and the normal dynamics near an invariant torus are 
strongly coupled. 

This difficulty is overcome by implementing the abstract procedure in Berti-Bolle \cite{BB13}-\cite{BB14}
 developed in order to
prove existence of quasi-periodic solutions for autonomous NLW (and NLS) with a multiplicative
potential. This approach reduces the search of an approximate inverse for \eqref{operatorF}
to 
the invertibility of a quasi-periodically forced PDE 
restricted on the normal directions. 
This  method  approximately decouples the ``tangential" and the ``normal" dynamics around  
an approximate invariant torus, introducing a suitable set of symplectic  variables
$ (\psi, \eta, w) $ near the torus, see \eqref{trasformazione modificata simplettica}.
Note that, in the first line of \eqref{trasformazione modificata simplettica}, $ \psi $ is the ``natural" angle variable which coordinates  the torus, and, in the third line, the normal variable $ z $ is 
only translated by the component $ z_0 (\psi )$ of the torus.
The second line 
completes this transformation
to a symplectic one. The canonicity of this map  is proved in  \cite{BB13} using the isotropy of
the approximate invariant torus $ i_\d $, see Lemma \ref{toro isotropico modificato}.
The change of variable 
\eqref{trasformazione modificata simplettica} brings the  torus $ i_\d $ ``at the origin". The advantage is that
 the second equation in \eqref{operatore inverso approssimato} (which corresponds to the action variables of the torus) 
 can be immediately solved, see 
\eqref{soleta}. Then it remains to solve the third equation \eqref{cal L omega}, i.e. to invert
the linear operator  $ {\cal L}_\om $. 
This is, up to finite dimensional remainders,  
a quasi-periodic Hamiltonian linear  Airy equation perturbed by a 
variable coefficients differential operator of order $ O(\pa_{xxx} ) $.
The exact form of $ {\cal L}_\om $ is obtained in Proposition \ref{prop:lin}. 
\\[1mm]
{\it Reduction of the linearized operator in the normal directions.} In section 
\ref{operatore linearizzato sui siti normali}
we conjugate the variable coefficients operator $ {\cal L}_\om $ in \eqref{Lom KdVnew} 
to a diagonal operator with constant coefficients
which describes infinitely many harmonic oscillators
\begin{equation}\label{linearosc}
{\dot v}_j + \mu_j^\infty v_j  = 0 \,  , \quad \mu_j^\infty := \ii (-m_3 j^3 + m_1 j) + r_j^\infty \in \ii \R \, , \quad j \notin S \, , 
\end{equation}
where the constants $  m_3 -1 $, $ m_1 \in \R $ and  $ \sup_j |r_j^\infty | $ are small, see Theorem \ref{teoremadiriducibilita}.
The main perturbative effect to the spectrum (and the eigenfunctions) of $ {\cal L}_\om $
is clearly due to  the term $ a_1 (\omega t, x ) \partial_{xxx} $ (see \eqref{Lom KdVnew}), 
and it is too strong for the usual 
 reducibility KAM techniques to work directly. 
The conjugacy of $ {\cal L}_\om $ with \eqref{linearosc} is obtained in several steps. 
The first task (obtained in sections \ref{step1}-\ref{step5}) is to conjugate 
$ {\cal L}_\om $ to another 
Hamiltonian operator of $ H_S^\bot $ with constant coefficients 
\be\label{L6 qualitativo}
{\cal L}_6 := \om \cdot \partial_\vphi + m_3 \partial_{xxx}  + m_1 \partial_x   + R_6  \, , \quad m_1, m_3 \in \R \, , 
\ee
up to a  small bounded remainder $ R_6 = O(\partial_x^0 ) $, see \eqref{def L6}. 
This expansion of $ {\cal L}_\om $ in ``decreasing symbols" with constant coefficients 
is similar to \cite{BBM}, and  it is somehow in the spirit of the works of Iooss, Plotnikov and Toland \cite{Ioo-Plo-Tol}-\cite{IP11} 
in water waves theory, and Baldi \cite{Baldi Benj-Ono} for Benjamin-Ono.
It is obtained by transformations which are very different from the usual KAM changes of variables. 
There are several differences  with respect to \cite{BBM}: 

\begin {enumerate}
\item 
The first step is to eliminate the $ x $-dependence from the  coefficient $ a_1 (\omega t, x ) \partial_{xxx} $ of the Hamiltonian operator $ {\cal L}_\om $. 
We cannot use the symplectic transformation $ {\cal A } $ defined in \eqref{primo cambio di variabile modi normali}, used in \cite{BBM},  
because  $ {\cal L}_\om $ acts on the normal subspace $ H_S^\bot $ only,
 and not on the whole Sobolev space as in \cite{BBM}.
We can not use the restricted map $ {\cal A}_\bot :=  \Pi_S^\bot  {\cal A} \Pi_S^\bot $
which is {\it not} symplectic. 
In order to find a symplectic diffeomorphism of $ H_S^\bot $ near  $ {\cal A}_\bot $, 
the first observation is to realize $ {\cal A } $  as the  flow map of the time dependent 
Hamiltonian transport linear PDE \eqref{transport-free}. 
Thus we conjugate $ {\cal L}_\om $ with the 
flow map of the projected Hamiltonian equation \eqref{problemi di cauchy}. In 
Lemma \ref{modifica simplettica cambio di variabile} we prove that it differs from $ {\cal A}_\bot $
up to finite dimensional operators.  
A technical, but important, fact is that the remainders produced after this conjugation of ${\cal L}_\om $ remain of the finite dimensional form \eqref{forma buona con gli integrali}, see  Lemma  \ref{cal R3}.

This step may be seen as a quantitative application of the Egorov theorem, see \cite{Taylor},
which describes how the principal symbol of a pseudo-differential operator (here $ a_1 (\om t, x) \pa_{xxx} $)
transforms 
under the flow of a  linear hyperbolic PDE (here \eqref{problemi di cauchy}). 
 \item 
Since the weak BNF procedure of section \ref{sec:WBNF} did not touch the quadratic terms $ O(z^2 ) $, 
the operator $ {\cal L}_\om $  has variable coefficients
also at the orders $ O(\e)$ and $ O(\e^2 )$, see \eqref{Lom KdVnew}-\eqref{a1p1p2}. 
These terms cannot be reduced to constants by the 
perturbative scheme in \cite{BBM},  which applies to terms $ R $  such that 
$ R \g^{ -1} \ll 1 $ where $ \g $ is the diophantine constant of the frequency vector $ \om $.
Here, since KdV is completely resonant,  
such 
$ \gamma = o(\e^2 ) $, see \eqref{omdio}.
These terms are reduced to constant coefficients  in sections 
\ref{BNF:step1}-\ref{BNF:step2} 
by means of purely algebraic arguments (linear BNF), which, ultimately, 
stem from the complete integrability of the fourth order BNF of the 
KdV equation \eqref{KdVmKdV}, see \cite{KaP}. 
\end{enumerate}

The order of the  transformations of sections \ref{step1}-\ref{subsec:mL0 mL5}
used to reduce $ {\cal L}_\om $ is not accidental. The first two steps in sections \ref{step1}, \ref{step2}
reduce to constant coefficients the quasi-linear  
term  $ O(\partial_{xxx}) $ and eliminate the term $ O(\partial_{xx})$, see \eqref{L2 Kdv} (the second transformation is a
time quasi-periodic reparametrization of time).  
Then, 
in section \ref{step3}, we  apply the transformation ${\cal T}$ \eqref{gran tau} in such a way 
that the space average of the coefficient $ d_1 (\vphi, \cdot ) $  in 
\eqref{L3 KdV} is constant. This is done in view
of the applicability of the descent method in section \ref{step5}. 
All these transformations are composition operators induced by diffeomorphisms of the torus. 
Therefore they are well-defined operators of a Sobolev space into itself, but their decay norm is infinite! 
We perform the transformation $ {\cal T } $  \emph{before} 
the linear Birkhoff normal form steps of sections \ref{BNF:step1}-\ref{BNF:step2}, because $\mathcal{T}$ is a change of variable that preserves the form \eqref{forma buona con gli integrali} 
of the remainders (it is not evident after the Birkhoff normal form). 
The Birkhoff transformations are symplectic maps of the form  $ I + \e O( \partial_x^{-1}) $.  
Thanks to this property the coefficient $ d_1 ( \varphi,x) $ 
obtained in step \ref{step3} is {\it not} changed by these Birkhoff maps. 
The transformation in section  \ref{step5} is one step of ``descent method'' which transforms 
$ d_1 ( \varphi,x) \partial_x $ into a constant  $ m_1 \partial_x  $. 
It is at this point of the regularization procedure that the assumption $({\mathtt S}1)$ on the tangential sites is used,
so that the space average of the function $ q_{>2} $ is zero, see Lemma \ref{p3 zero average}. 
Actually we only need that the average of 
the function  in \eqref{unico pezzo} is zero. If 
 $ f_5 = 0 $ 
 (see \eqref{order5})  then $({\mathtt S}1)$  is not required. 
This completes the task of conjugating ${\cal L}_\om $ to $ {\cal L}_6 $ in \eqref{L6 qualitativo}.

Finally, in section \ref{subsec:mL0 mL5} we apply the abstract reducibility Theorem 4.2 in \cite{BBM}, 
based on a quadratic KAM scheme, which 
completely diagonalizes the linearized operator, obtaining  \eqref{linearosc}. 
The required smallness condition  \eqref{R6resto} for $ R_6 $  holds.  Indeed
the biggest term in $ R_6 $ comes from the conjugation of  
$  \e \partial_x v_\e (\theta_0 (\vphi), y_\d (\vphi)) $ in   \eqref{a1p1p2}.
The linear BNF procedure of section \ref{BNF:step1} had eliminated 
its main contribution $  \e \partial_x v_\e ( \vphi,0) $. 
It remains  $ \e \partial_x \big( v_\e (\theta_0 (\vphi), y_\d (\vphi) ) - v_\e ( \vphi,0)  \big) $ 
which has size $ O( \e^{7-2b} \g^{-1} ) $ due to the estimate \eqref{ansatz 0} of the approximate solution.
This term enters in the variable coefficients of $ d_1 (\vphi,x)  \partial_x $
and $ d_0 (\vphi, x) \partial_x^0 $. The first one had been reduced to the constant 
operator $ m_1 \pa_x $ by the descent method of section \ref{step5}.
The latter term is an operator of order $O(\pa_x^0 )$ which satisfies \eqref{R6resto}.
Thus $ {\cal L}_6 $ may be diagonalized 
by the iterative scheme of Theorem 4.2 in \cite{BBM} 
which requires the smallness condition $ O( \e^{7-2b} \g^{-2}) \ll 1  $. This is the content of section \ref{subsec:mL0 mL5}.  
\\[1mm]
{\it The Nash-Moser iteration.}
In section \ref{sec:NM} we perform the nonlinear Nash-Moser iteration which 
finally proves Theorem \ref{main theorem} and, therefore, Theorem \ref{thm:KdV}.
The optimal smallness condition required for the convergence of the scheme is 
$ \e \| {\cal F}(\vphi, 0, 0 ) \|_{s_0+ \mu} \g^{-2} \ll 1 $, see  \eqref{nash moser smallness condition}.
It is verified because $ \| X_P(\vphi, 0 , 0 ) \|_s \leq_s \e^{6 - 2b} $ 
(see \eqref{stima XP}), which, in turn, is a consequence of having eliminated  the terms
$ O(v^5), O( v^4 z)$ from the original Hamiltonian \eqref{H iniziale KdV}, 
see \eqref{widetilde cal H}. This requires the condition ($ {\mathtt S}2$). 

\smallskip

\noindent
{\it Acknowledgements}. We  thank M. Procesi, P. Bolle and T. Kappeler for many useful discussions. 
This research was supported by the European Research Council under FP7, 
and partially by the grants STAR 2013 and 
PRIN 2012 ``Variational and perturbative aspects of nonlinear differential problems".

\section{Preliminaries}

\subsection{Hamiltonian formalism of KdV}\label{sec:Ham For}

The Hamiltonian vector field $X_H$ generated by a Hamiltonian 
$ H : H^1_0(\T_x) \to \R$ is $ X_H (u) := \partial_x \nabla H (u) $, because
$$
d H (u) [ h] 
= ( \nabla H(u), h )_{L^2(\T_x)} = \Om ( X_H (u), h ) \, , \quad \forall u, h \in H^1_0(\T_x) \, , 
$$
where $\Om$ is the non-degenerate symplectic form
\be\label{2form KdV}
\Omega (u, v) := \int_{\T} (\partial_x^{-1 } u) \, v \, dx \, ,   \quad 
\forall u, v \in H^1_0 (\T_x) \, , 
\ee
and $ \partial_x^{-1} u $ is the periodic primitive of $ u $ with zero average. 
Note  that 
\begin{equation} \label{def pi 0} 
\partial_x \partial_x^{-1} =  \partial_x^{-1} \partial_x = \pi_0 \, , \quad 
\pi_0 (u) :=  u - \frac{1}{2\p} \int_{\T} u(x) \, dx \,. 
\end{equation}
A map 
is symplectic if  it preserves 
the 2-form $ \Omega $.

We also remind that the Poisson bracket between two functions $ F $, $ G : H^1_0(\T_x) \to \R$ is
\begin{equation}\label{Poisson bracket}
\{ F (u), G(u) \} := \Om (X_F, X_G ) = \int_{\T} \gr F(u) \pa_x \gr G (u) dx \, . 
\end{equation}
The linearized KdV equation at $ u $ is  
$$
h_t = \pa_x \, (\pa_u \nabla H)(u) [h ] = X_{K}(h) \, ,  
$$
where $ X_K $ is the KdV Hamiltonian vector field with quadratic Hamiltonian 
$ K = \frac12  ((\pa_u \nabla H)(u)[h], h)_{L^2(\T_x)} $ 
$= \frac12 (\pa_{uu} H)(u)[h, h]  $. By the Schwartz theorem, the Hessian operator $ A := (\pa_u \nabla H)(u) $ is symmetric, namely
$ A^T = A $, with respect to the $ L^2$-scalar product. 

\smallskip

\noindent
{\bf Dynamical systems formulation.} 
It is convenient to regard the KdV equation also in the Fourier representation 
\begin{equation}\label{Fourier}
u(x) = {\mathop \sum}_{j \in \Z \setminus \{0\} } u_j e^{\ii j x} \, , \qquad 
u(x) \longleftrightarrow u := (u_j)_{j \in \Z \setminus \{0\} } \, , \quad u_{-j} = \overline{u}_j  \, ,
\end{equation}
where the Fourier indices $ j \in \Z \setminus \{ 0 \}$ by the definition \eqref{def phase space} of the phase space
and $u_{-j} = \overline{u}_j$ because $u(x)$ is real-valued.
The symplectic structure writes
\begin{equation}\label{2form0}
\Omega = \frac12 \sum_{j \neq 0}  \frac{1}{\ii j} du_j \wedge d u_{-j} = \sum_{j \geq 1} 
\frac{1}{\ii j} du_j \wedge d u_{-j}
\, , \qquad 
\Omega ( u, v ) = \sum_{j \neq 0} \frac{1}{\ii j} u_j v_{-j} = \sum_{j \neq 0} \frac{1}{\ii j} u_j { \overline v}_{j} \, , 
\end{equation}
the Hamiltonian vector field $X_H$ and the Poisson bracket $\{ F, G \}$ are
\begin{equation}\label{PoissonBr}
[X_H (u)]_j = \ii j (\partial_{u_{-j}} H) (u) \, , \  \forall j \neq 0 \, ,
\quad 
\{ F (u), G(u) \} = - {\mathop \sum}_{j \neq 0} \ii j   (\partial_{u_{-j}} F) (u)  (\partial_{u_j} G) (u) \, .  \end{equation}
\noindent
\noindent
{\bf Conservation of momentum.} 
A Hamiltonian 
\begin{equation}   \label{cons mom}
H(u) = \sum_{j_1, \ldots, j_n \in \Z \setminus \{ 0 \} } H_{j_1, \ldots, j_n} u_{j_1} \ldots u_{j_n}, 
\quad 
u(x) = \sum_{j \in \Z \setminus \{ 0 \} } u_j e^{\ii jx},
\end{equation}
homogeneous of degree $n$, \emph{preserves the momentum} if the coefficients 
$H_{j_1, \ldots, j_n}$ are zero for $j_1 + \ldots + j_n \neq 0$, 
so that the sum in \eqref{cons mom} is restricted to integers 
 such that $j_1 + \ldots + j_n = 0$. 
Equivalently, $H$ preserves the momentum if $\{ H, M \} = 0$, where $M$ is the  momentum 
$ M(u) := \int_{\T} u^2 dx = $ $  \sum_{j \in \Z \setminus \{0\}} u_j u_{-j} $. 
The homogeneous components of degree $ \leq 5 $ 
of the KdV Hamiltonian $ H $ in \eqref{Ham in intro} 
preserve the momentum because, by \eqref{order5}, 
the homogeneous component $f_5$ of degree 5 does not depend on the space variable $x$. 

\smallskip

\noindent
{\bf Tangential and normal variables.} 
Let $\bar\jmath_1, \ldots, \bar\jmath_\nu \geq 1$ be $\nu$ distinct integers, 
and $S^+ := \{ \bar\jmath_1, \ldots, \bar\jmath_\nu \}$. 
Let $S$ be the symmetric set in \eqref{tang sites}, 
and $S^c := \{ j \in \Z \setminus \{ 0 \} : j \notin S \}$ its complementary set in $\Z \setminus \{ 0 \}$. 
We decompose the phase space as 
\begin{equation}\label{splitting S-S-bot}
H^1_0 (\T_x) := H_S \oplus H_S^\bot \,, 
\quad 
H_S := \mathrm{span}\{ e^{\ii jx} : \, j \in S \} , \quad 
H_S^\bot := \big\{ u = \sum_{j \in S^c} u_j e^{\ii jx} \in H^1_0(\T_x)  \big\} ,
\end{equation}
and we denote by $\Pi_S $, 
$\Pi_S^\bot $ the corresponding orthogonal projectors. 
Accordingly we decompose
\begin{equation}  \label{u = v + z}
u = v + z, \qquad  
v = \Pi_S u := {\mathop \sum}_{j \in S} u_j \, e^{\ii jx}, \quad  
z = \Pi_S^\bot u := {\mathop \sum}_{j \in S^c} u_j \, e^{\ii jx} \, ,  
\end{equation}
where $  v $ is called the {\it tangential} variable and $ z $ the {\it normal} one.
We shall sometimes identify $ v \equiv (v_j)_{j \in S } $ and  $ z \equiv (z_j)_{j \in S^c } $. 
The subspaces  $ H_S $ and $ H_S^\bot $ are {\it symplectic}. 
The dynamics of these two components is quite different. 
On $ H_S $ we shall introduce the action-angle variables, see \eqref{coordinate azione angolo}.
The linear frequencies of oscillations on the tangential  sites are 
\begin{equation}\label{bar omega}
\bar\om := (\bar\jmath_1^3, \ldots, \bar\jmath_\nu^3) \in \N^\nu.
\end{equation}

\subsection{Functional setting}

{\bf Norms.} Along the paper we shall use the notation 
\begin{equation}\label{Sobolev coppia}
\| u \|_s := \| u \|_{H^s( \T^{\nu + 1})} := \| u \|_{H^s_{\vphi,x} }
\end{equation}
to denote the Sobolev norm of functions $ u = u(\vphi,x) $ in the Sobolev space $ H^{s} (\T^{\nu + 1} ) $. 
We shall denote by $ \| \ \|_{H^s_x} $ the Sobolev norm in the phase space of functions $ u :=  u(x) \in H^{s} (\T ) $.
Moreover $ \| \ \|_{H^s_\vphi} $ will denote the Sobolev norm of scalar functions, like 
the Fourier components $ u_j (\vphi)  $.  

We fix $ s_0 := (\nu+2) \slash 2 $ so that   $ H^{s_0} (\T^{\nu + 1} ) \hookrightarrow L^{\infty} (\T^{\nu + 1} ) $ and the spaces 
$ H^s (\T^{\nu + 1} ) $, $ s > s_0  $, are an algebra.  
At the end of this section 
we report interpolation properties of the Sobolev norm that will be currently used along the paper.
We shall also denote 
\begin{align} \label{HSbot nu}
H^s_{S^\bot} (\T^{\nu+1}) 
& := \big\{  u \in H^s(\T^{\nu + 1} )  \, : \, u (\vphi, \cdot ) \in H_S^\bot \  
\forall \vphi \in \T^\nu \big\} \,,
\\
\label{HS nu}
H^s_{S} (\T^{\nu+1}) 
& := \big\{  u \in H^s(\T^{\nu + 1} )  \, : \, u (\vphi, \cdot ) \in H_{S} \   
\forall \vphi \in \T^\nu \big\} \,. 
\end{align}
For a function $u : \Om_o \to E$, $\om \mapsto u(\om)$, where $(E, \| \ \|_E)$ is a Banach space and 
$ \Om_o $ is a subset of $\R^\nu $, we define the sup-norm and the Lipschitz semi-norm
\begin{equation} \label{def norma sup lip}
\| u \|^{\sup}_E 
:= \| u \|^{\sup}_{E,\Om_o} 
:= \sup_{ \om \in \Om_o } \| u(\om ) \|_E \, , 
\quad
\| u \|^{\lip}_E 
:= \| u \|^{\lip}_{E,\Om_o}  
:= \sup_{\om_1 \neq \om_2 } 
\frac{ \| u(\om_1) - u(\om_2) \|_E }{ | \om_1 - \om_2 | }\,,
\end{equation}
and, for $ \g > 0 $, the Lipschitz norm
\begin{equation} \label{def norma Lipg}
\| u \|^{\Lipg}_E  
:= \| u \|^{\Lipg}_{E,\Om_o}
:= \| u \|^{\sup}_E + \g \| u \|^{\lip}_E  \, . 
\end{equation}
If $ E = H^s $ we simply denote $ \| u \|^{\Lipg}_{H^s} := \| u \|^{\Lipg}_s $.  We shall use the notation 
$$
a \leq_s b \quad \ \Longleftrightarrow \quad a \leq C(s) b \quad
\text{for  some  constant  } C(s) > 0  \, . 
$$ 
\noindent
{\bf Matrices with off-diagonal decay.}
A linear operator can be identified, as usual, with its matrix representation. 
We recall the definition  
of the $ s $-decay norm (introduced in \cite{BB13JEMS})  of an infinite dimensional matrix. 
This norm is  used in \cite{BBM} for the KAM reducibility scheme of the linearized operators. 
\begin{definition}\label{def:norms}
The $s$-decay norm of an infinite dimensional matrix $ A := (A_{i_1}^{i_2} )_{i_1, i_2 \in \Z^b } $, $b \geq 1$, is  
\begin{equation} \label{matrix decay norm}
\left| A \right|_{s}^2 := 
\sum_{i \in \Z^b} \left\langle i \right\rangle^{2s} 
\Big( \sup_{ \begin{subarray}{c} i_{1} - i_{2} = i 
\end{subarray}}
| A^{i_2}_{i_1}| \Big)^{2} \, .
\end{equation}
For parameter dependent matrices $ A := A(\omega) $, $\omega  \in \Omega_o \subseteq \R^\nu $, 
the definitions \eqref{def norma sup lip} and \eqref{def norma Lipg} become 
\begin{equation} \label{matrix decay norm Lip}
| A |^{\sup}_s  := \sup_{ \omega \in \Omega_o } | A(\om ) |_s \, , 
\quad
| A |^{\lip}_s := \sup_{\om_1 \neq \om_2} 
\frac{ | A(\om_1) - A(\om_2) |_s }{ | \om_1 - \om_2 | }\,,
\quad
| A |^{\Lipg}_s := | A |^{\sup}_s + \g | A |^{\lip}_s  \,. 
\end{equation}
\end{definition}

Such a norm is modeled on the behavior of matrices representing the multiplication
operator by a function.
Actually, given a function $ p \in H^s(\T^b) $, the multiplication operator $ h \mapsto p h $ is represented by the T\"oplitz matrix 
$ T_i^{i'} = p_{i - i'} $ and  $ |T|_s = \| p \|_s $. 
If $p = p(\om )$ is a Lipschitz family of functions, then 
\be\label{multiplication Lip}
|T|_s^\Lipg = \| p \|_s^\Lipg\,.
\ee
The $s$-norm  satisfies classical algebra and interpolation inequalities, see \cite{BBM}.

\begin{lemma} \label{prodest}
Let  $A = A(\om)$ and $B = B(\om)$ be matrices depending in a Lipschitz way on the parameter $\om \in 
\Omega_o \subset \R^\nu $. 
Then for all $s \geq s_0 > b/2 $ there are $ C(s) \geq C(s_0) \geq 1 $ such that
\begin{align} \label{algebra Lip}
|A B |_s^{\Lipg} 
& \leq  C(s) |A|_s^{\Lipg} |B|_s^{\Lipg} \, , 
\\
\label{interpm Lip}
|A B|_{s}^{\Lipg} 
& \leq C(s) |A|_{s}^{\Lipg} |B|_{s_0}^{\Lipg} 
+ C(s_0) |A|_{s_0}^{\Lipg} |B|_{s}^{\Lipg} .
\end{align}
\end{lemma}
The $ s $-decay norm controls the Sobolev norm, namely 
\be\label{interpolazione norme miste}
\| A h \|_s^\Lipg 
\leq C(s) \big(|A|_{s_0}^\Lipg \| h \|_s^\Lipg + |A|_{s}^\Lipg \| h \|_{s_0}^\Lipg \big).
\ee
Let now $ b := \nu + 1 $. 
An important sub-algebra  is formed by the {\it T\"oplitz in time matrices} defined by
\be\label{Topliz matrix}
 A^{(l_2, j_2)}_{(l_1, j_1)}  := A^{j_2}_{j_1}(l_1 - l_2 )\,  ,
\ee
whose  decay norm \eqref{matrix decay norm} is
\be\label{decayTop}
|A|_s^2 =  \sum_{j \in \Z, l \in \Z^\nu} \big( \sup_{j_1 - j_2 = j} |A_{j_1}^{j_2}(l)| \big)^2  \langle l,j \rangle^{2 s} \, . 
\ee
These matrices are identified with the $ \vphi $-dependent family
of operators
\be\label{Aphi}
A(\vphi) := \big( A_{j_1}^{j_2} (\vphi)\big)_{j_1, j_2 \in \Z} \, , \quad 
A_{j_1}^{j_2} (\vphi) := {\mathop\sum}_{l \in \Z^\nu} A_{j_1}^{j_2}(l) e^{\ii l \cdot \vphi}
\ee
which act on functions of the $x$-variable as
\be\label{notationA}
A(\vphi) : h(x) = \sum_{j \in \Z} h_j e^{\ii jx} \mapsto  
A(\vphi) h(x) = \sum_{j_1, j_2 \in \Z}  A_{j_1}^{j_2} (\vphi) h_{j_2} e^{\ii j_1 x} \, .  
\ee
We still denote by $ | A(\vphi) |_s $ the $ s $-decay norm of the matrix in \eqref{Aphi}.
As in \cite{BBM}, all the transformations that  we shall construct in this paper are of this type (with $ j, j_1, j_2 \neq 0 $ because
they act on the phase space $ H^1_0 (\T_x) $).
This observation allows to interpret the conjugacy procedure from a dynamical point of view, see \cite{BBM}-section 2.2. 
Let us fix some terminology.

\begin{definition}\label{operatore Hamiltoniano}
We say that:

the operator 
$(A h)(\vphi, x) := A(\vphi) h(\vphi, x)$ is  {\sc symplectic} if each 
$ A (\vphi ) $, $ \vphi \in \T^\nu $,  is  a symplectic map of the phase space
(or of a symplectic subspace like $ H_S^\bot $);

the operator
$
\om \cdot \partial_{\vphi} - \partial_x G( \vphi )  $ is  {\sc Hamiltonian} 
if each $ G (\vphi) $, $  \vphi \in \T^\nu $, is  symmetric; 

an operator is {\sc real} if it maps real-valued functions into real-valued functions.
\end{definition}

As well known, a Hamiltonian operator $  \om \cdot \partial_{\vphi} - \partial_x G( \vphi ) $ is transformed,
under a symplectic map $ {\cal A} $, 
into another  Hamiltonian operator $ \om \cdot \partial_{\vphi} - \partial_x E( \vphi ) $, 
see e.g. \cite{BBM}-section 2.3.

\smallskip

We conclude this preliminary section recalling the following well known lemmata, see Appendix of \cite{BBM}.

\begin{lemma} {\bf (Composition)}
\label{lemma:composition of functions, Moser}
Assume $ f \in C^s (\T^d \times B_1)$, $B_1 := \{ y \in \R^m  : 
|y| \leq 1 \}$. Then 
$ \forall u \in H^{s}(\T^d, \R^m) $ such that $ \| u \|_{L^\infty} < 1  $, 
the composition operator $\tilde{f}(u)(x) := f(x, u(x))$ satisfies
$ \| \tilde f(u) \|_s \leq C \| f \|_{C^s} (\|u\|_{s} + 1)  $
where the constant $C $ depends on $ s ,d $.  If $ f \in C^{s+2} $  
and $ \| u + h \|_{L^\infty} < 1$, then    
\begin{align*}
\big\| \tilde f(u+h) - {\mathop\sum}_{i = 0}^k \frac{\tilde{f}^{(i)}(u)}{i !} [h^i] \big\|_s
& \leq C \| f \|_{C^{s+ 2}} \, \| h \|_{L^\infty}^k ( \| h \|_{s} + \| h \|_{L^\infty} \| u \|_{s}) \, , \quad k = 0, 1 \, . 
\end{align*}
The previous statement also holds replacing 
$\| \ \|_s$  
with the norms $| \ |_{s, \infty} $.  
\end{lemma}

\begin{lemma} \label{lemma:standard Sobolev norms properties}
{\bf (Tame product).} 
For $s \geq s_0 > d/2  $, 
$$
\| uv \|_s \leq C(s_0) \|u\|_s \|v\|_{s_0} + C(s) \|u\|_{s_0} \| v \|_s \, , 
\quad \forall u,v \in H^s(\T^d) \, .
$$
For $s \geq 0$, $s \in \N$,
$$
\| uv \|_s \leq \tfrac32 \,  \| u \|_ {L^\infty}  \| v \|_s + C(s) \| u \|_ {W^{s, \infty}} \| v \|_0  \, , 
\quad \forall u \in W^{s,\infty}(\T^d) \, , \ v \in H^s(\T^d)   \, .
$$
The above inequalities also  hold for 
the norms $\Vert \  \Vert_s^{{\rm Lip}(\gamma)}$.
\end{lemma}

\begin{lemma} {\bf (Change of variable)}  \label{lemma:utile} 
Let $p \in W^{s,\infty} (\T^d,\R^d) $,  $ s \geq 1$, with 
$ \| p \|_{W^{1, \infty}} \leq 1/2 $.  Then the function  $f(x) = x + p(x)$
is invertible, with inverse  $ f\inv(y)  = y + q(y)$ where
 $q \in W^{s,\infty}(\T^d,\R^d)$, and 
 $ \| q \|_{W^{s, \infty}}  \leq C \| p \|_{ W^{s, \infty}} $.
If, moreover,  $p = p_\om $ depends in a Lipschitz way on 
a parameter $\om \in \Omega \subset \R^\nu $, 
and 
$ \| D_x p_\om \|_ {L^\infty}  \leq 1/2 $, $ \forall \om $, 
then 
$ \| q \|_{W^{s, \infty}}^{{\rm Lip}(\gamma)} \leq  C  \| p \|_{W^{s+1, \infty}}^{{\rm Lip}(\gamma)} $.
The constant $C := C (d, s) $ is independent of $\g$.

If $u \in H^s (\T^d,\C)$, then $ (u\circ f)(x) := u(x+p(x))$ satisfies 
\begin{align*}
\| u \circ f \|_s 
& \leq  C (\|u\|_s + \| p \|_{W^{s, \infty}} \|u\|_1),
\quad
\| u \circ f - u \|_s 
\leq C ( \| p \|_{L^\infty} \| u \|_{s + 1}  + \| p \|_{W^{s, \infty}} \| u \|_{2} ) ,
\\
\| u \circ f \|_{s}^{{{\rm Lip}(\gamma)}}
& \leq C  \, 
\big( \| u \|_{s+1}^{{{\rm Lip}(\gamma)}} + \| p \|_{W^{s, \infty}}^{{\rm Lip}(\gamma)}\| u \|_2^{{\rm Lip}(\gamma)} \big). \nonumber
\end{align*}
The function  $u \circ f^{-1} $  satisfies the same bounds.
\end{lemma}

\section{Weak Birkhoff normal form}\label{sec:WBNF}

The Hamiltonian 
of the perturbed KdV equation \eqref{kdv quadratica} is  $ H = H_2 + H_3 + H_{\geq 5} $ (see \eqref{Ham in intro})  where 
\begin{equation} \label{H iniziale KdV}
H_2 (u):=\frac{1}{2} \int_{\T} u_x^{2} \, dx \, , \quad 
H_3(u) := \int_\T u^3 dx \, , \quad 
H_{\geq 5}(u) := \int_\T f(x, u,u_x) dx \,,
\end{equation}
and $f$ satisfies \eqref{order5}. 
According to the splitting \eqref{u = v + z} $ u = v + z $, $ v \in H_S $, $ z \in H_S^\bot $, 
we have
\begin{equation}\label{prima v z}
H_2(u) = \int_\T \frac{v_x^2}{2}\, dx + \int_\T \frac{z_x^2}{2} \, dx, 
\quad 
H_3 (u) =  \int_{\T}  v^3 dx + 3 \int_{\T}  v^2 z dx 
+ 3\int_{\T}  v z^2  dx  + \int_{\T}  z^3 dx  \, . 
\end{equation}
For a finite-dimensional space
\begin{equation} \label{def E finito}
E := E_{C} :=  \mathrm{span} \{ e^{\ii jx} :  0 < |j| \leq C \}, \quad C > 0,
\end{equation}
let $\Pi_E $ denote the corresponding $ L^2 $-projector on $E$.

The notation $R(v^{k-q} z^q)$ indicates a homogeneous polynomial of degree $k$ in $(v,z)$ of the form 
$$
R(v^{k-q} z^q) = M[\underbrace{v, \ldots, v}_{(k-q)\, \text{times}}, 
\underbrace{z, \ldots, z}_{q \, \text{times}} \,], \qquad 
M = k\text{-linear} \, .
$$
\begin{proposition} \label{prop:weak BNF}  
{\bf (Weak Birkhoff normal form)} Assume Hypothesis $ ({\mathtt S}2) $. 
Then there exists an  analytic invertible symplectic transformation 
of the phase space $ \Phi_B : H^1_0 (\T_x) \to H^1_0 (\T_x) $ 
of the form 
\begin{equation} \label{finito finito}
\Phi_B(u) = u + \Psi(u), 
\quad
\Psi(u) = \Pi_E \Psi(\Pi_E u),
\end{equation}
where $ E $ is a finite-dimensional space as in \eqref{def E finito}, such that the transformed Hamiltonian is
\begin{equation} \label{widetilde cal H}
{\cal H} := H \circ \Phi_B  
= H_2 + \mH_3 + \mH_4 + {\cal H}_5 + {\cal H}_{\geq 6} \,,
\end{equation}
where $H_2$ is defined in \eqref{H iniziale KdV}, 
\begin{equation} \label{H3tilde} 
\mH_3 := \int_{\T} z^3\,dx + 3 \int_{\T} v z^2 \,dx \,, \quad 
\mH_4 := - \frac32 \sum_{j \in S} \frac{|u_j|^4}{j^2} + \mH_{4,2} + \mH_{4,3} \,, \quad 
{\cal H}_5 := \sum_{q=2}^5 R(v^{5-q} z^q)\,,
\end{equation} 
\begin{equation} \label{mH3 mH4}
\mH_{4,2} := 6 \int_\T  v z \Pi_S \big((\partial_x^{-1} v)(\partial_x^{-1} z) \big)\,dx + 3 \int_\T z^2 \pi_0 (\partial_x^{-1} v)^2 \,dx\,,  \quad 
\mH_{4, 3} := R(v z^3) \,,  
\end{equation}
and ${\cal H}_{\geq 6}$ collects all the terms of order at least six in $(v,z)$. 
\end{proposition}

The rest of this section is devoted to the proof of Proposition \ref{prop:weak BNF}.

First, we remove the cubic terms $ \int_{\T} v^3 + 3 \int_{\T} v^2 z $ from the Hamiltonian $ H_3 $ defined in \eqref{prima v z}.
In the Fourier coordinates \eqref{Fourier}, we have
\begin{equation}\label{H3 Fourier}
H_2  = \frac12 \sum_{j \neq 0} j^2 |u_j|^2, 
\quad H_3  = \sum_{j_1 + j_2 + j_3 = 0} u_{j_1} u_{j_2} u_{j_3} \, .
\end{equation}
We look for  a symplectic transformation $ \Phi^{(3)}  $ of the phase space 
which eliminates the monomials $ u_{j_1} u_{j_2} u_{j_3} $ of  $ H_3 $ with at most {\it one} index
outside $ S $. Note that, by the relation $ j_1 + j_2 + j_3 = 0 $, they are {\it finitely} many.
We look for $\Phi^{(3)} := (\Phi^t_{F^{(3)}})_{|t=1}$ as the time-1 flow map generated by the Hamiltonian vector field $X_{F^{(3)}}$, with an auxiliary Hamiltonian of the form
$$
F^{(3)}(u) := \sum_{j_1 + j_2 + j_3 = 0} F^{(3)}_{j_1 j_2 j_3} u_{j_1} u_{j_2} u_{j_3} \,.
$$
The transformed Hamiltonian is 
\begin{align}
H^{(3)} & := H \circ \Phi^{(3)} 
=  H_2 + H_3^{(3)}  +  H_4^{(3)}  +  H_{\geq 5}^{(3)}  \,,
\notag \\ 
\label{H tilde 234}
H_3^{(3)} & = H_3 + \{ H_2, F^{(3)} \}, \quad 
 H_4^{(3)}  = \frac12 \{ \{ H_2, F^{(3)}\}, F^{(3)}\} + \{H_3, F^{(3)} \} , 
\end{align}
where $ H_{\geq 5}^{(3)} $ collects all the terms of order at least five in $(u,u_x)$.
By \eqref{H3 Fourier} and \eqref{PoissonBr}  we calculate 
$$
H_3^{(3)}  = 
\sum_{j_1 + j_2 + j_3 = 0} 
\big\{ 1 - \ii (j_1^3 + j_2^3 + j_3^3) F^{(3)}_{j_1 j_2 j_3} \big\} \, u_{j_1} u_{j_2} u_{j_3} \,.
$$
Hence, in order to eliminate the monomials with at most one index outside $ S $, 
we choose
\begin{equation}\label{F3q}
F^{(3)}_{j_1 j_2 j_3} := \begin{cases}
 \dfrac{1}{\ii (j_1^3 + j_2^3 + j_3^3)} & \text{if} \,\,(j_1,j_2,j_3) \in {\cal A}\,, 
\\
0 & \text{otherwise},
\end{cases}
\end{equation}
where ${\cal A} := \big\{ (j_1 , j_2 , j_3) \in (\Z \setminus \{ 0 \})^3$ : $j_1 + j_2 + j_3 = 0$, 
$j_1^{3} + j_2^{3} + j_3^{3} \neq 0$, and at least 2 among $j_1 , j_2 , j_3$ belong to $S \big\}$.  
Note that 
\begin{equation} \label{def calA quadratica}
{\cal A} = \big\{ (j_1 , j_2 , j_3) \in (\Z \setminus \{ 0 \})^3 : 
j_1 + j_2 + j_3 = 0, \, \text{and at least 2 among}\,\, j_1 , j_2 , j_3 \,\, 
\text{belong to} \, S  \big\} 
\end{equation}
because of the elementary relation
\begin{equation}\label{prodottino}
j_1 + j_2 + j_3 = 0 \quad \Rightarrow \quad j_1^3 + j_2^3 + j_3^3 = 3 j_1 j_2 j_3 \neq 0 
\end{equation}
being $ j_1, j_2, j_3 \in \Z \setminus \{ 0 \}$. 
Also note that $ \mA $ is a finite set, actually $ \mA \subseteq [- 2 C_S, 2 C_S]^{3} $ where 
the tangential sites $ S \subseteq [- C_S, C_S ]$.
As a consequence, the Hamiltonian vector field $ X_{F^{(3)}} $  has finite rank and vanishes outside  
the finite dimensional subspace $ E := E_{2 C_S}$ (see \eqref{def E finito}), namely
$$
 X_{ F^{(3)}}(u)  = \Pi_E X_{ F^{(3)}} ( \Pi_E u ) \, . 
$$
Hence its flow $ \Phi^{(3)} : H^1_0 (\T_x) \to H^1_0 (\T_x) $ has the form  \eqref{finito finito} and it is analytic. 

By construction, all the monomials of $ H_3 $   
with at least two indices outside $ S $ 
are not modified by the transformation $ \Phi^{(3)}  $. Hence  (see \eqref{prima v z}) we have  
\begin{equation}\label{lem:H3tilde}
H_3^{(3)}  = \int_{\T} z^3\,dx + 3 \int_{\T} v z^2 \,dx \, .
\end{equation}
We now compute the fourth order term
$ H_4^{(3)}  =  \sum_{i=0}^4 H_{4,i}^{(3)} $   in \eqref{H tilde 234}, 
where $ H_{4,i}^{(3)} $ is of type $ R( v^{4-i} z^i )$.

\begin{lemma}
One has (recall the definition \eqref{def pi 0} of $ \pi_0 $) 
\begin{equation}\label{Htilde41} 
{H}_{4,0} ^{(3)}   := \frac32 \int_\T v^2 \pi_0[(\partial_x^{-1} v)^2] dx \, ,  \quad
H_{4,2}^{(3)} := 6 \int_\T v z \Pi_S \big((\partial_x^{-1} v)(\partial_x^{-1} z) \big)\,dx 
+ 3 \int_\T z^2 \pi_0 [(\partial_x^{-1} v)^2] dx \, . 
\end{equation}
\end{lemma}

\begin{proof}
We write $ H_3 = H_{3, \leq 1} +  H_3^{(3)}  $ where $ H_{3, \leq 1}(u)  := \int_\T v^3 dx  + 3 \int_\T  v^2 z \,  dx  $. 
Then, by \eqref{H tilde 234}, we get
\begin{equation} \label{grado 4 *}
H_4^{(3)} 
=  \frac12 	\big\{ H_{3, \leq 1} \, , F^{(3)} \big\} + \{ H_3^{(3)} , F^{(3)} \} \,.
\end{equation} 
By \eqref{F3q},  \eqref{prodottino}, the auxiliary Hamiltonian may be written as
$$
F^{(3)} (u) 
= - \frac{1}{3}\sum_{(j_1, j_2, j_3) \in {\cal A}}  
\frac{u_{j_1} u_{j_2} u_{j_3}}{ (\ii j_1) ( \ii j_2) ( \ii j_3)}  
= - \frac{1}{3}\int_\T (\partial_x^{-1} v)^3 dx  - \int_\T (\partial_x^{-1} v)^2 (\partial_x^{-1} z) dx \, . 
$$
Hence, using that the projectors $\Pi_S$, $\Pi_S^\bot $ are self-adjoint and $\partial_x^{-1}$ is skew-selfadjoint, 
\begin{equation} \label{grad F3 formula}
\nabla F^{(3)}(u) 
= \partial_x^{-1}\big\{ (\partial_x^{-1} v)^2 + 2 \Pi_S \big[ (\partial_x^{-1} v)(\partial_x^{-1} z)\big] \big\}
\end{equation}
(we have used that $\pa_x^{-1} \pi_0 = \pa_x^{-1}$ be the definition of $\pa_x^{-1}$). 
Recalling the Poisson bracket definition \eqref{Poisson bracket}, 
using  that $ \nabla H_{3, \leq 1}(u) = 3 v^2 + 6 \Pi_S(v z) $ and   \eqref{grad F3 formula},
we get 
\begin{align}
\{ H_{3, \leq1}, F^{(3)} \} 
& = \int_\T \big\{ 3 v^2 + 6 \Pi_S(v z) \big\} \pi_0 \big\{ (\partial_x^{-1} v)^2 + 
2 \Pi_S \big[ (\partial_x^{-1} v)(\partial_x^{-1} z)\big]  \big\}\,dx  
\nonumber
\\
& = 3 \int_\T v^2 \pi_0 (\partial_x^{-1} v)^2\,dx  +
 12 \int_\T \Pi_S(v z) \Pi_S [ (\partial_x^{-1} v)(\partial_x^{-1} z) ]\,dx + R(v^3 z) \, .  \label{H31F}
\end{align}
Similarly, since $ \nabla  H_3^{(3)}(u) = 3 z^2 + 6 \Pi_S^\bot (v z) $, 
\begin{equation}
\{ H_3^{(3)}, F^{(3)} \}  
=  
3 \int_\T z^2 \pi_0 (\partial_x^{-1} v)^2\,dx  
+ R(v^3 z) + R(v z^3) \,.  \label{H3tildeF}
\end{equation}
The lemma follows by \eqref{grado 4 *}, \eqref{H31F}, \eqref{H3tildeF}.
\end{proof}

We now construct a symplectic map $ \Phi^{(4)}  $ 
such that  the Hamiltonian system obtained transforming  $ H_2 + H_3^{(3)} +  H_4^{(3)} $ 
possesses the invariant subspace $ H_S $  (see \eqref{splitting S-S-bot}) 
and its dynamics on $ H_S $ is integrable and non-isocronous.  
Hence we have to eliminate the term $ H_{4,1}^{(3)} $ 
(which is linear in $ z $), 
and to normalize  $ H_{4,0}^{(3)} $ (which is independent of $ z $). 
We need the following elementary lemma (Lemma 13.4 in \cite{KaP}).

\begin{lemma}  \label{lemma:interi} 
Let $j_1, j_2, j_3, j_4 \in \Z $ such that $ j_1 + j_2 + j_3 + j_4 = 0 $. 
Then 
$$
j_1^3 + j_2^3 + j_3^3 + j_4^3 = -3 (j_1 + j_2) (j_1 + j_3) (j_2 + j_3).
$$
\end{lemma}

\begin{lemma}
There exists a symplectic transformation $ \Phi^{(4)}  $ of the form \eqref{finito finito}
such that
\begin{equation}\label{Ham4quadratica}
H^{(4)} := H^{(3)} \circ \Phi^{(4)} = 
H_2 + H_3^{(3)} + H_4^{(4)} + H^{(4)}_{\geq 5} \,, 
\qquad 
H^{(4)}_4 := - \frac32 \sum_{j \in S} \frac{|u_j|^4}{j^2}   +  H_{4,2}^{(3)} + H_{4,3}^{(3)} \,,
\end{equation}
where  $ H_3^{(3)} $ is defined in \eqref{lem:H3tilde},
$ H_{4,2}^{(3)} $ in \eqref{Htilde41}, $ H_{4,3}^{(3)} = R( v z^3) $ 
and  $ H_{\geq 5}^{(4)} $ collects all the terms of degree at least five in $(u,u_x)$.
\end{lemma}

\begin{proof}
We look for a map $\Phi^{(4)} := (\Phi_{F^{(4)}}^t)_{|t=1}$ which is the time $ 1$-flow map 
of an auxiliary Hamiltonian
$$
F^{(4)}(u) := \sum_{\begin{subarray}{c}
j_1 + j_2 + j_3 + j_4 = 0 \\
\text{at least}\,\,3\,\,\text{indices are in}\,\,S
\end{subarray}} F^{(4)}_{j_1 j_2 j_3 j_4} u_{j_1} u_{j_2} u_{j_3} u_{j_4} 
$$
with the same form of the Hamiltonian $ H_{4,0}^{(3)} + H_{4,1}^{(3)} $.
The transformed Hamiltonian is
\begin{equation}\label{callH}
H^{(4)} :=  H^{(3)} \circ \Phi^{(4)} 
= H_2 + H_3^{(3)} + H_4^{(4)} +  H_{\geq 5}^{(4)}, \quad 
H_4^{(4)} = \{ H_2, F^{(4)} \} + H_4^{(3)} , 
\end{equation}
where $ H_{\geq 5}^{(4)} $ collects all the terms of order at least five. 
We write $ H_4^{(4)} = \sum_{i=0}^4 H_{4,i}^{(4)} $ where each $ H_{4,i}^{(4)} $ if of type $ R( v^{4-i} z^i ) $. 
We choose the coefficients 
\begin{equation}\label{def F4}
F^{(4)}_{j_1 j_2 j_3 j_4} := \begin{cases}
 \dfrac{ H^{(3)}_{j_1j_2j_3j_4}   }{\ii (j_1^3 + j_2^3 + j_3^3+ j_4^3)} & \text{if} \,\,(j_1,j_2,j_3, j_4) \in {\cal A}_4\,, 
\\
0 & \text{otherwise},
\end{cases}
\end{equation}
where 
\begin{align*} 
{\cal A}_4  := \big\{ (j_1 , j_2 , j_3, j_4) \in (\Z \setminus \{ 0 \})^4 & : 
j_1 + j_2 + j_3 + j_4 = 0, \,  j_1^{3} + j_2^{3} + j_3^{3} + j_4^{3} \neq 0, \\
& \quad  \text{and at most one among}\,\, j_1 , j_2 , j_3, j_4 \,\, \text{outside} \, S \big\}  \, .  
\end{align*}
By this definition  $ H_{4,1}^{(4)}=  0 $
because there exist no integers $ j_1, j_2 , j_3 \in S$, $ j_4 \in S^c $ satisfying 
$ j_1 + j_2 + j_3 + j_4 = 0 $, $ j_1^3 + j_2^3 + j_3^3 + j_4^3 = 0 $, by Lemma \ref{lemma:interi} and the
fact that $ S $ is symmetric. 
By construction, the terms $ H_{4,i}^{(4)}=  H_{4,i}^{(3)} $, $ i = 2, 3, 4$, are not changed by $ \Phi^{(4)} $.
Finally, by \eqref{Htilde41} 
\begin{equation} \label{H2F1}
H_{4,0}^{(4)}=  \frac32 
\sum_{\begin{subarray}{c}
j_1, j_2, j_3, j_4  \in S \\
j_1 + j_2 + j_3 + j_4 = 0 \\
j_1^3 + j_2^3 + j_3^3 + j_4^3 = 0 \\
j_1 + j_2 \,,\,j_3 + j_4 \neq 0
\end{subarray}} \frac{1}{(\ii j_3) (\ii j_4)}u_{j_1} u_{j_2} u_{j_3} u_{j_4} \, . 
\end{equation}
If $ j_1 + j_2 + j_3 + j_4 = 0 $ and $ j_1^3 + j_2^3 + j_3^3 + j_4^3 = 0$, then $(j_1 + j_2)(j_1 + j_3)(j_2 + j_3) = 0$ by Lemma \ref{lemma:interi}. 
We develop the  sum in \eqref{H2F1} with respect to the first index $j_1$.  Since $ j_1 + j_2 \neq 0 $
the possible cases are:
\[
(i) \ \big\{ 
j_2 \neq - j_1, \ 
j_3 = - j_1, \
j_4 = - j_2 \big\}   
\qquad \text{or} \qquad
(ii) \ \big\{
j_2 \neq - j_1, \ 
j_3 \neq - j_1, \ 
j_3 = - j_2, \ 
j_4 = - j_1 \big\} . 
\]
Hence, using $ u_{-j} = \bar{u}_j $ (recall \eqref{Fourier}), and since $S$ is symmetric, we have
\begin{equation} \label{case ii}
\sum_{(i)}  \frac{1}{j_3 j_4} u_{j_1} u_{j_2} u_{j_3} u_{j_4} 
= \sum_{j_1, j_2 \in S, j_2 \neq - j_1}  \frac{|u_{j_1}|^2 |u_{j_2}|^2 }{j_1 j_2} 
= \sum_{j,j' \in S}  \frac{|u_j|^2 |u_{j'}|^2}{j j'}  + \sum_{j \in S}  \frac{|u_j|^4}{j^2}   = \sum_{j \in S}  \frac{|u_j|^4}{j^2}\,,    
\end{equation}
and in the second case ($ii$) 
\begin{equation} \label{case iii}
\sum_{(ii)} \frac{1}{j_3 j_4}  u_{j_1} u_{j_2} u_{j_3} u_{j_4} 
 = \sum_{j_1, j_2, j_2 \neq \pm j_1} \frac{1}{j_1 j_2}  u_{j_1} u_{j_2} u_{-j_2} u_{-j_1}  =
  \sum_{j \in S} \frac{1}{j}  |u_{j}|^2 \Big( \sum_{j_2 \neq \pm j} \frac{1}{j_2}  |u_{j_2}|^2 \Big) = 0\,.   
\end{equation}
Then \eqref{Ham4quadratica} follows by \eqref{H2F1}, \eqref{case ii}, \eqref{case iii}. 
\end{proof}

Note that the Hamiltonian $ H_2 + H_3^{(3)} + H_4^{(4)} $ (see \eqref{Ham4quadratica}) possesses
the invariant subspace $ \{ z =  0 \} $ and the system restricted to $ \{ z =  0 \} $ is
completely integrable and non-isochronous (actually it is formed by $ \nu $ decoupled rotators). 
We shall construct  quasi-periodic solutions which bifurcate from this invariant manifold. 

In order to enter in a perturbative regime, 
we have to eliminate further monomials of $ H^{(4)} $ in  \eqref{Ham4quadratica}.   
The minimal requirement for  the convergence of the nonlinear  Nash-Moser iteration is to 
eliminate the monomials $ R(v^5) $ and $ R(v^4 z)$.
Here we need the choice of the sites of Hypothesis $ ({\mathtt S}2) $. 

\begin{remark}\label{remark:cubic}
In the KAM theorems  \cite{k1}, \cite{Po3} (and \cite{PP}, \cite{Wang}),
as well as for the perturbed mKdV equations \eqref{mKdV}, 
these further steps of Birkhoff normal form are not required because 
the nonlinearity of the original PDE is yet cubic. A difficulty of KdV 
is that the nonlinearity is quadratic.  
\end{remark}

We spell out Hypothesis $( {\mathtt S}2)$ as follows:
\begin{itemize}
\item {\sc $( {\mathtt S}2_0)$.} 
There is no choice of $ 5 $ integers $ j_1, \ldots, j_5 \in  S $ such that 
\begin{equation}\label{scelta dei siti grado 7}
j_1 + \ldots + j_5 = 0\,,\quad 
j_1^3 + \ldots + j_5^3 = 0 \,.
\end{equation} 
\item {\sc  $({\mathtt S}2_1)$.} 
There is no choice of $4 $ integers $j_1, \ldots, j_{4} $ in  $ S $ 
and an integer in the complementary set $ j_5 \in S^c :=  (\Z \setminus \{0\}) \setminus S   $ such that  \eqref{scelta dei siti grado 7} holds. 
\end{itemize}

The homogeneous component of degree $ 5 $ of $H^{(4)}$ is
$$
H^{(4)}_5 (u) = \sum_{j_1+ \ldots + j_5 = 0} H^{(4)}_{j_1, \ldots, j_5}  u_{j_1} \ldots u_{j_5} \,.
$$
We want to remove from $H^{(4)}_5$ the terms with at most one index among $ j_1, \ldots , j_5 $ outside $ S $.
We consider the auxiliary Hamiltonian 
\begin{equation} \label{Fn odd}
F^{(5)} = \sum_{\begin{subarray}{c} j_1 + \ldots + j_5 = 0 \\ 
\text{at most one index outside $ S $}
\end{subarray} } 
F_{j_1 \ldots j_5}^{(5)} u_{j_1} \ldots u_{j_5} \, , \quad 
 F_{j_1 \ldots j_5}^{(5)} := \frac{H_{j_1 \ldots j_5}^{(5)}}{ \ii (j_1^3 + \ldots + j_5^3)} \,.
\end{equation}
By Hypotheses $ ({\mathtt S}2_0), ({\mathtt S}2_1) $, if $ j_1 + \ldots + j_5 = 0 $ with at most one index outside $ S $ then 
$ j_1^3 + \ldots + j_5^3 \neq 0 $ and  
$ F^{(5)}$ is well defined. 
Let $ \Phi^{(5)} $ be the time $ 1 $-flow generated by 
$ X_{F^{(5)}} $.  The new Hamiltonian is
\begin{equation}\label{callHn}
H^{(5)} :=  H^{(4)} \circ \Phi^{(5)} 
= H_2 + H_3^{(3)} + H_4^{(4)} +  \{ H_2, F^{(5)} \} + H_5^{(4)} + H^{(5)}_{\geq 6}  
\end{equation}
where, by  \eqref{Fn odd}, 
$$
H_5^{(5)} := \{ H_2, F^{(5)} \} + H_5^{(4)} = {\mathop\sum}_{q = 2}^5 R(v^{5 - q} z^q) \, .
$$

Renaming $ {\cal H} := H^{(5)} $, namely $ {\cal H}_n := H^{(n)}_n $, $ n =3, 4, 5 $, and 
setting $ \Phi_B := \Phi^{(3)} \circ \Phi^{(4)}  \circ \Phi^{(5)} $,  formula \eqref{widetilde cal H} follows. 

The homogeneous component $ H^{(4)}_{5} $ preserves the momentum, see section \ref{sec:Ham For}.  
Hence $F^{(5)} $ also preserves the momentum. 
As a consequence, also $ H^{(5)}_k $, $ k \leq 5 $, preserve the momentum.

Finally, since $F^{(5)} $  is Fourier-supported on a finite set, 
the transformation $\Phi^{(5)}$ is of type \eqref{finito finito} (and analytic), 
and therefore also the composition $ \Phi_B $ is of type \eqref{finito finito} (and analytic). 

\section{Action-angle variables}\label{sec:4}

We now introduce action-angle variables on the tangential directions by the change of coordinates
\begin{equation}\label{coordinate azione angolo}
\begin{cases}
u_j  := \sqrt{\xi_j + |j| y_j} \, e^{\ii \theta_j}, \qquad & \text{if} \  j \in S\,,\\
u_j  := z_j, \qquad  & \text{if} \  j \in S^c \, , 
\end{cases}
\end{equation}
where (recall $ u_{-j} = {\overline u}_j $)
\begin{equation}\label{simmeS}
\xi_{-j} = \xi_j  \, , \quad 
\xi_j > 0 \, , \quad 
y_{-j} = y_j \, , \quad 
\theta_{-j} = - \theta_j \, , \quad 
\teta_j, \, y_j \in \R \, , \quad  
\forall j \in S \,.
\end{equation}
For the tangential sites $ S^+ := \{ {\bar \jmath_1}, \ldots, {\bar \jmath_\nu} \} $ we shall also denote 
$ \teta_{\bar \jmath_i} := \teta_i $, $ y_{\bar \jmath_i} := y_i $, 
$ \xi_{\bar \jmath_i} := \xi_i $,  $ i =1, \ldots \, \nu $.

The symplectic 2-form $ \Omega $ in \eqref{2form0} (i.e.  \eqref{2form KdV}) becomes 
\begin{equation}\label{2form}
{\cal W} := \sum_{i=1}^\nu  d \theta_i \wedge d y_i  
+ \frac12 \sum_{j \in S^c \setminus \{ 0 \} } \frac{1}{\ii j} \, d z_j \wedge d z_{-j} =  
\big( \sum_{i=1}^\nu  d \theta_i \wedge d y_i \big)  \oplus \Om_{S^\bot} = d \Lambda 
\end{equation}
where $ \Om_{S^\bot} $ denotes the restriction of $ \Om $ to $ H_S^\bot $ (see \eqref{splitting S-S-bot}) and 
$ \Lambda $ is the contact $ 1 $-form on $ \T^\nu \times \R^\nu \times H_S^\bot $ defined by
$ \Lambda_{(\theta, y, z)} : \R^\nu \times \R^\nu \times H_S^\bot \to \R $, 
\begin{equation}\label{Lambda 1 form}
\Lambda_{(\theta, y, z)}[\widehat \theta, \widehat y, \widehat z] := 
-  y \cdot \widehat \theta + \frac12 ( \partial_x^{-1} z, \widehat z )_{L^2 (\T)} \, .    
\end{equation}
Instead of working in a shrinking neighborhood of the origin, it is a convenient devise to rescale the ``unperturbed actions" $ \xi $
and the action-angle variables as
\begin{equation}\label{rescaling kdv quadratica}
 \xi \mapsto \e^2 \xi \, , \quad  y \mapsto \e^{2b} y \, , \quad z \mapsto \e^b z  \, . 
 \end{equation}
Then the symplectic $ 2 $-form in \eqref{2form} transforms into $  \e^{2b} {\cal W } $. Hence
the Hamiltonian system generated by $ {\cal H} $ in \eqref{widetilde cal H} 
transforms into  the new Hamiltonian system 
\begin{equation}  \label{def H eps}
 \dot \theta = \partial_y H_\e (\theta, y, z)   \, , \ 
 \dot y   = -  \partial_\teta H_\e  (\theta, y, z) \, , \  
 z_t  =  \partial_x \nabla_z H_\e  (\theta, y, z)  \, , \quad H_\e := \e^{-2b} \mH \circ A_\e
\end{equation}
where 
\begin{equation}  \label{def A eps}
A_\e (\theta, y, z) := \e v_\e(\theta, y) + \e^b z 
:= \e {\mathop \sum}_{j \in S} \sqrt{\xi_j + \e^{2(b-1)} |j| y_j} \, e^{\ii \theta_j} e^{\ii j x} + \e^b z \, .
\end{equation}
We shall still denote by $ X_{H_\e} = (\partial_y H_\e, - \partial_\teta H_\e, \partial_x \nabla_z H_\e) $ the Hamiltonian vector field in the variables $ (\teta, y, z ) \in \T^\nu \times \R^\nu \times H_S^\bot $.

We now write explicitly the Hamiltonian $ H_\e (\theta, y, z) $ in  \eqref{def H eps}. 
The quadratic Hamiltonian $ H_2 $ in \eqref{H iniziale KdV} transforms into
\be\label{shape H2}
\e^{-2b}H_2 \circ A_\e = const + 
{\mathop \sum}_{j \in S^+} j^3 y_j + \frac{1}{2} \int_{\T} z_x^2 \, dx \, , 
\ee
and, recalling \eqref{H3tilde}, \eqref{mH3 mH4}, 
the Hamiltonian $ {\cal H }$ in \eqref{widetilde cal H} 
transforms into
(shortly writing $ v_\e := v_\e (\theta, y) $)  
\begin{align}
H_\e (\theta, y, z) &  
=  e(\xi) + \a (\xi) \cdot y + \frac12 \int_\T z_x^2 dx + \e^b\int_\T z^3 dx  + 3 \e \int_\T v_\e  z^2 dx \label{formaHep}
\\ & \quad 
+ \e^2 \Big\{6 \int_\T  v_\e z
 \Pi_S \big((\partial_x^{-1} v_\e)(\partial_x^{-1} z) \big)\,dx + 3 \int_\T z^2 \pi_0 (\partial_x^{-1} v_\e)^2\,dx \Big\} 
 - \frac32 \e^{2b} {\mathop \sum}_{j \in S} y_j^2   \nonumber 
\\ & \quad
+ \e^{b + 1} R(v_\e z^3) 
+ \e^3 R(v_\e^3 z^2) 
+ \e^{2+b} \sum_{q=3}^5 \e^{(q-3)(b-1)} R(v_\e^{5-q} z^q) 
+ \e^{-2b} {\cal H}_{\geq 6} (\e v_\e + \e^b z )   \nonumber 
\end{align}
where $e(\xi)$ is a constant, and the frequency-amplitude map is
\begin{equation} \label{mappa freq amp}
\a(\xi) := \bar\om + \e^2 {\mathbb A} \xi  \, , \quad 
{\mathbb A} := - 6 \, \text{diag} \{ 1/j \}_{j \in S^+} \, .
\end{equation} 
We write the Hamiltonian in \eqref{formaHep} as 
\begin{equation}  \label{Hamiltoniana Heps KdV}
H_\e = {\cal N} +  P \,,   \quad 
{\cal N}(\theta, y, z) 
= \a (\xi) \cdot y + \frac12 \big(N(\theta) z , z \big)_{L^2(\T)} \,,
\end{equation}
where  
\begin{align}
\frac12 \big(N(\theta) z, z \big)_{L^2(\T)}  & 
:= \frac12 \big( (\pa_{z} \gr H_\e)(\theta, 0, 0)[z], z \big)_{L^2(\T)}  \label{Nshape}  
=  \frac12 \int_\T z_x^2 dx  + 3 \e \int_\T v_\e(\theta, 0)  z^2 dx  \\
& + \e^2 \Big\{6 \int_\T  v_\e(\theta, 0) z \Pi_S 
\big((\partial_x^{-1} v_\e(\theta,0))(\partial_x^{-1} z) \big) dx 
+ 3 \int_\T z^2 \pi_0 (\partial_x^{-1} v_\e(\theta, 0))^2 dx \Big\} + \ldots \nonumber
\end{align}
and $P := H_\e - {\cal N} $.

\section{The nonlinear functional setting}\label{sec:functional}

We look for an embedded invariant torus 
\begin{equation} \label{embedded torus i}
i : \T^\nu \to \T^\nu \times \R^\nu \times H_S^\bot, \quad  
\vphi \mapsto i (\vphi) := ( \theta (\vphi), y (\vphi), z (\vphi)) 
\end{equation}
of the Hamiltonian vector field $ X_{H_\e}  $ 
filled by quasi-periodic solutions with diophantine frequency $ \om $. We require that $ \om $ belongs to the set 
\begin{equation}\label{Omega epsilon}
\Omega_\e := \a ( [1,2]^\nu ) = 
\{ \a(\xi) : \xi \in [1,2]^\nu \}  
\end{equation}
where $ \a $ is the diffeomorphism 
\eqref{mappa freq amp}, and, in the Hamiltonian $ H_\e $ in \eqref{Hamiltoniana Heps KdV}, we choose 
\be\label{linkxiomega}
\xi = \a^{-1}(\om) = \e^{-2} {\mathbb A}^{-1} (\om - \bar\om) \, . 
\ee
Since any $ \omega \in \Omega_\e$ is $ \e^2 $-close to the integer vector
$ \bar \om $ (see \eqref{bar omega}), we require that the constant $\g$ in the diophantine inequality 
\be\label{omdio}
 |\omega \cdot l | \geq \gamma \langle l \rangle^{-\tau} \, , \ \  \forall l \in \Z^\nu \setminus \{0\}   \, , 
\quad \text{satisfies} \  \  \gamma = \e^{2+a} \quad \text{for some} \ a > 0 \,.  
\ee
We remark that the definition of $\g$ in \eqref{omdio} is slightly stronger than the minimal condition, which is $ \g \leq c \e^2 $ with $ c $ small enough. 
In addition to \eqref{omdio} we shall also require that $ \om $ satisfies the first and second order Melnikov-non-resonance conditions \eqref{Omegainfty}. 

We 
look for an embedded invariant torus 
of the modified Hamiltonian vector field $ X_{H_{\e, \zeta}} = X_{H_\e} + (0, \zeta, 0)$ which is generated by 
the Hamiltonian 
\begin{equation}\label{hamiltoniana modificata}
H_{\e, \zeta} (\teta, y, z) :=  H_\e (\teta, y, z) + \zeta \cdot \theta\,,\quad \zeta \in \R^\nu\,.
\end{equation}
Note that $ X_{H_{\e, \zeta}} $ is  periodic in $\theta $
(unlike $ H_{\e, \zeta}  $). 
It turns out that an invariant torus for $ X_{H_{\e, \zeta}} $ is actually 
invariant  for $ X_{H_\e} $, see Lemma \ref{zeta = 0}. 
We introduce the parameter $\zeta \in \R^\nu $ in order to control the average in the 
$ y$-component 
of the linearized equations. 
Thus we look for zeros of the nonlinear operator
\begin{align}
 {\cal F} (i, \zeta ) 
& :=   {\cal F} (i, \zeta, \om, \e )  := {\cal D}_\om i (\vphi) - X_{H_{\e, \zeta}}  (i(\vphi)) =
{\cal D}_\om i (\vphi) - X_{\cal N}  (i(\vphi))  - X_P  (i(\vphi)) + (0, \zeta, 0 )  \label{operatorF}  \\
& \nonumber  :=  \left(
\begin{array}{c}
{\cal D}_\om \theta (\vphi) - \partial_y H_\e ( i(\vphi)  )   \\
{\cal D}_\om y (\vphi)  +  \partial_\teta H_\e ( i(\vphi)  ) + \zeta  \\
{\cal D}_\om z (\vphi) -  \partial_x \nabla_z H_\e  ( i(\vphi)) 
\end{array}
\right) \!\! 
=  \!\!
\left(
\begin{array}{c} 
{\cal D}_\om   \Theta (\vphi)  -  \partial_y P (i(\vphi) )   \\
\!\! \! {\cal D}_\om  y (\vphi)  + \frac12  \partial_\teta ( N(\theta (\vphi)) z(\vphi), z(\vphi) )_{L^2(\T)} 
	+  \partial_\teta P  ( i(\vphi) ) + \zeta  \!\!\!\!  \\
{\cal D}_\om  z (\vphi) - \partial_x N ( \theta  (\vphi )) z (\vphi)  
-  \partial_x \nabla_z P  ( i(\vphi) ) 
\end{array} 
\right) 
\end{align}
where $ \Theta(\ph) := \teta (\vphi) - \vphi $ is $ (2 \pi)^\nu $-periodic 
and we use the short notation 
\begin{equation}\label{Domega}
{\cal D}_\om := \om \cdot \partial_\vphi \, . 
\end{equation}
The Sobolev norm of the periodic component of the embedded torus 
\begin{equation}\label{componente periodica}
{\mathfrak I}(\vphi)  := i (\vphi) - (\vphi,0,0) := ( {\Theta} (\ph), y(\ph), z(\ph))\,, \quad \Theta(\ph) := \teta (\vphi) - \vphi \, , 
\end{equation}
is 
\begin{equation}\label{norma fracchia}
\|  {\mathfrak I}  \|_s := \| \Theta \|_{H^s_\vphi} +  \| y  \|_{H^s_\vphi} +  \| z \|_s 
\end{equation}
where $ \| z \|_s := \| z \|_{H^s_{\vphi,x}}  $ is defined in \eqref{Sobolev coppia}.
We link the rescaling \eqref{rescaling kdv quadratica}
with the diophantine constant $ \g = \e^{2+a} $ by choosing
\be\label{link gamma b}
\gamma = \e^{2b}\,, \qquad 
b = 1 + ( a \slash 2 ) \, .
\ee
Other choices are possible, 
see Remark \ref{comm3}. 

\begin{theorem}\label{main theorem}
Let the tangential sites $ S $ in \eqref{tang sites} 
satisfy 
$({\mathtt S}1), ({\mathtt S}2)$. 
Then, for all $ \e \in (0, \e_0 ) $, where $ \e_0 $ is small enough,   
there exists a Cantor-like set $ {\cal C}_\e \subset \Omega_\e $,
with  asympotically full measure as $ \e \to 0 $, namely
\begin{equation}\label{stima in misura main theorem}
\lim_{\e\to 0} \, \frac{|{\cal C}_\e|}{|\Omega_\e|} = 1 \, , 
\end{equation}
such that, for all $ \omega \in {\cal C}_\e $, there exists a solution $ i_\infty (\vphi) := i_\infty (\omega, \e)(\vphi) $ 
of  ${\cal D}_\om i_\infty(\vphi) - X_{H_\e}(i_\infty(\vphi)) = 0 $.
Hence the embedded torus 
$ \vphi \mapsto i_\infty (\vphi) $ is invariant for the Hamiltonian vector field $ X_{H_\e (\cdot, \xi)} $  
with $ \xi $ as in \eqref{linkxiomega}, and it is filled by quasi-periodic solutions with frequency $ \om $.
The torus $i_\infty$ satisfies 
\be\label{stima toro finale}
\|  i_\infty (\vphi) -  (\vphi,0,0) \|_{s_0 + \mu}^\Lipg = O(\e^{6 - 2 b} \g^{-1} ) 
\ee
for some $ \mu := \mu (\nu) > 0 $. Moreover, the torus $ i_\infty $ is {\sc linearly stable}. 
\end{theorem}

Theorem \ref{main theorem} is  proved in sections \ref{costruzione dell'inverso approssimato}-\ref{sec:NM}. It implies Theorem \ref{thm:KdV}
where the $ \xi_j $ in \eqref{solution u} are $ \e^2 \xi_j  $, $ \xi_j \in [1,2 ] $, in \eqref{linkxiomega}.
By \eqref{stima toro finale}, 
going back to the variables before the rescaling \eqref{rescaling kdv quadratica}, 
we get
$ \Theta_\infty = O( \e^{6-2b} \g^{-1}) $, $ y_\infty = O( \e^6 \g^{-1} ) $, $ z_\infty = O( \e^{6-b} \g^{-1} ) $, which, 
as $ b \to  1^+ $, tend to the expected optimal estimates. 

\begin{remark} \label{comm3}
There are other possible ways to link the rescaling \eqref{rescaling kdv quadratica}
with the diophantine constant $ \g = \e^{2+a} $.
The choice $ \g > \e^{2b} $
reduces to study perturbations of an isochronous system  (as in \cite{Ku}, \cite{k1}, \cite{Po3}), and it is convenient to 
 introduce $ \xi (\om) $ as a variable. 
The case  $ \e^{2b} > \g $, in particular $ b = 1 $, 
has to be dealt with a perturbation approach  of a non-isochronous system {\`a la} Arnold-Kolmogorov. 
\end{remark}

We now give the tame estimates for the composition operator induced by the Hamiltonian vector fields $ X_{\cal N} $ and  $ X_P $ in \eqref{operatorF}, that we shall use in the next sections. 

We first estimate the composition operator induced by $ v_\e (\teta, y) $  defined in \eqref{def A eps}.
Since the functions $ y \mapsto \sqrt{\xi + \e^{2(b - 1)}|j| y} $, $\theta \mapsto e^{\ii \theta}$ 
are analytic  for $\e$ small enough and $|y| \leq C$, 
the composition Lemma \ref{lemma:composition of functions, Moser} implies that, for all  $ \Theta, y \in H^s(\T^\nu, \R^\nu )$,  
$   \| \Theta \|_{s_0},  \| y \|_{s_0} \leq 1 $, setting $\theta(\ph) := \ph + \Theta (\ph)$, 
$ \|  v_\e (\theta (\vphi) ,y(\vphi) ) \|_s 
\leq_s 1 + \| \Theta  \|_s + \| y \|_s  $. 
Hence, using also 
\eqref{linkxiomega},   the map $ A_\e $ in \eqref{def A eps} satisfies, for all 
$   \| {\mathfrak I} \|_{s_0}^\Lipg  \leq 1 $ (see \eqref{componente periodica})
\be \label{stima Aep}
 \| A_\e (\theta (\vphi),y(\vphi),z(\vphi)) \|_s^{\Lipg} \leq_s \e (1 + \| {\mathfrak I} \|_s^\Lipg) \, . 
\ee 
We now give tame estimates for the Hamiltonian vector fields $ X_{\cal N} $, $ X_P $,  $ X_{H_\e} $, 
see \eqref{Hamiltoniana Heps KdV}-\eqref{Nshape}.

\begin{lemma}\label{lemma quantitativo forma normale}
Let $ \fracchi(\ph) $ in \eqref{componente periodica} satisfy
$ \| {\mathfrak I} \|_{s_0 + 3}^\Lipg  \leq C\e^{6 - 2b} \gamma^{-1} $. 
Then  
%
\begin{alignat}{2} 
\| \partial_y P(i) \|_s^\Lipg   & \leq_s \e^4 + \e^{2b} \| {\mathfrak I}\|_{s+1}^\Lipg \label{D y P} \, ,  
& 
\| \partial_\theta P(i) \|_s^\Lipg
& \leq_s \e^{6 - 2b} (1 + \| {\mathfrak I}  \|_{s + 1}^\Lipg) 
\\ 
\| \nabla_z P(i) \|_s^\Lipg
& \leq_s \e^{5 - b} + \e^{6- b} \g^{-1} \| {\mathfrak I}  \|_{s + 1}^\Lipg \, ,
&
\| X_P(i)\|_s^\Lipg 
& \leq_s \e^{6 - 2b} + \e^{2b} \| {\mathfrak I}\|_{s + 3}^\Lipg  \label{stima XP}
\\
\label{stime linearizzato campo hamiltoniano}
\| \partial_{\theta} \partial_y P(i)\|_s^\Lipg 
& \leq_s \e^{4} + \e^{5} \g^{-1} \| {\mathfrak I}\|_{s + 2}^\Lipg  \, ,
& 
\| \partial_y \nabla_z P(i)\|_s^\Lipg  
& \leq_s \e^{b+3} +  \e^{2b - 1} \| {\mathfrak I} \|_{s+2}^\Lipg   
\end{alignat}
\begin{equation}
\label{D yy P}
\| \partial_{yy} P(i) + 3 \e^{2b} I_{\R^\nu} \|_s^\Lipg 
\leq_s \e^{2+ 2b} + \e^{2b+3} \g^{-1} \| \fracchi \|_{s+2}^\Lipg
\end{equation}
and, for all  $ \widehat \imath :=  (\widehat \Theta, \widehat y, \widehat z) $, 
\begin{align}
\label{D yii}
\| \partial_y d_{i} X_P(i)[\widehat \imath \,]\|_s^\Lipg &\leq_s \e^{2 b - 1} \big( \| \widehat \imath \|_{s + 3}^\Lipg + \| {\mathfrak I}\|_{s + 3}^\Lipg \| \widehat \imath \|_{s_0 + 3}^\Lipg\big)  \\
\label{tame commutatori} \| d_i X_{H_\e}(i) [\widehat \imath \, ] + (0,0, \partial_{xxx} \hat z)\|_s^\Lipg  
& \leq_s \e \big( \| \widehat \imath\|_{s + 3}^\Lipg 
+ \|{\mathfrak I} \|_{s + 3}^\Lipg \| \widehat \imath\|_{s_0 + 3}^\Lipg \big) \\
\label{parte quadratica da P}
\| d_i^2 X_{H_\e}(i) [\widehat \imath, \widehat \imath \,]\|_s^\Lipg 
& \leq_s \e\Big( \| \widehat \imath \|_{s + 3}^\Lipg \| 
\widehat \imath\|_{s_0 + 3}^\Lipg + \| {\mathfrak I}\|_{s + 3}^\Lipg (\| \widehat \imath\|_{s_0 + 3}^\Lipg)^2\Big).
\end{align}
\end{lemma}

In the sequel we will also use that, by the diophantine condition \eqref{omdio}, the operator $ {\cal D}_\om^{-1} $  (see \eqref{Domega})
is defined for all functions $ u  $ with zero $ \vphi $-average, and satisfies 
\be\label{Dom inverso}
 \| {\cal D}_\om^{-1} u \|_s \leq C \g^{-1} \| u \|_{s+ \tau} \, , \quad  \| {\cal D}_\om^{-1} u \|_s^{\Lipg} \leq C \g^{-1} \| u \|_{s+ 2 \tau+1}^{\Lipg} \, . 
\ee

\section{Approximate inverse}\label{costruzione dell'inverso approssimato}

In order to implement a convergent Nash-Moser scheme that leads to a solution of 
$ \mF(i, \zeta) = 0 $ 
 our aim is to construct an \emph{approximate right inverse} (which satisfies tame estimates) of the linearized operator 
\begin{equation}\label{operatore linearizzato}
d_{i, \zeta} {\cal F}(i_0, \zeta_0 )[\widehat \imath \,, \widehat \zeta ] =
d_{i, \zeta} {\cal F}(i_0 )[\widehat \imath \,, \widehat \zeta ]   = 
{\cal D}_\om \widehat \imath - d_i X_{H_\e} ( i_0 (\vphi) ) [\widehat \imath ] + (0, \widehat \zeta, 0 )
 \,,  
\end{equation}
see Theorem \ref{thm:stima inverso approssimato}. Note that 
$ d_{i, \zeta} {\cal F}(i_0, \zeta_0 ) = d_{i, \zeta} {\cal F}(i_0 ) $ is independent of $ \zeta_0 $ (see \eqref{operatorF}).

The notion of approximate right inverse is introduced in \cite{Z1}. It denotes a linear operator 
which is an \emph{exact} right inverse  at a solution $ (i_0, \zeta_0) $  of $ {\cal F}(i_0, \zeta_0) = 0 $.
We want to implement the general strategy in \cite{BB13}-\cite{BB14}
which reduces the search of an approximate right inverse of \eqref{operatore linearizzato} 
to the search of an approximate inverse on the normal directions only.

It is well known that an invariant torus $ i_0 $ with diophantine flow 
is isotropic (see e.g. \cite{BB13}), namely the pull-back $ 1$-form $ i_0^* \Lambda $ is closed, 
where $ \Lambda $ is the contact 1-form in \eqref{Lambda 1 form}. 
This is tantamount to say that the 2-form $ \cal W $ (see \eqref{2form}) vanishes on the torus $ i_0 (\T^\nu )$ 
(i.e. $\cal W$ vanishes on the tangent space at each point $i_0(\ph)$ of the manifold 
$i_0(\T^\nu)$), because 
$ i_0^* {\cal W} =  i_0^* d \Lambda  = d i_0^* \Lambda $. 
For an ``approximately invariant" torus $ i_0 $ the 1-form $ i_0^* \Lambda$ is only  ``approximately closed".
In order to make this statement quantitative we consider
\begin{equation}\label{coefficienti pull back di Lambda}
i_0^* \Lambda = {\mathop \sum}_{k = 1}^\nu a_k (\vphi) d \vphi_k \,,\quad 
a_k(\vphi) := - \big( [\pa_\ph \teta_0 (\vphi)]^T y_0 (\vphi)  \big)_k 
+ \frac12 ( \partial_{\vphi_k} z_0(\ph),  \partial_{x}^{-1} z_0(\ph) )_{L^2(\T)}
\end{equation}
and we quantify how small is 
\begin{equation} \label{def Akj} 
i_0^* {\cal W} = d \, i_0^* \Lambda = {\mathop\sum}_{1 \leq k < j \leq \nu} A_{k j}(\vphi) d \vphi_k \wedge d \vphi_j\,,\quad A_{k j} (\vphi) := 
\partial_{\vphi_k} a_j(\ph) - \partial_{\vphi_j} a_k(\ph) \, .
\end{equation}
Along this section we will always assume the following hypothesis 
(which will be verified at each step of the Nash-Moser iteration):

\begin{itemize}
\item {\sc Assumption.} 
The map $\omega\mapsto i_0(\omega)$ is a Lipschitz function defined on some subset $\Omega_o \subset \Omega_\e$, where $\Omega_\e$ is defined in \eqref{Omega epsilon}, and, for some $ \mu := \mu (\t, \nu) >  0 $,   
\begin{equation}\label{ansatz 0}
\| {\mathfrak I}_0  \|_{s_0+\mu}^{\Lipg} \leq C\e^{6 - 2b} \gamma^{-1}, \quad 
\| Z \|_{s_0 + \mu}^{\Lipg} \leq C \e^{6 - 2b}, \quad 
\gamma = \e^{2 + a}, 
\quad 
 b := 1 + (a/2) \,, 
 \quad a \in (0, 1 / 6), 
\end{equation}
where $\fracchi_0(\ph) := i_0(\ph) - (\ph,0,0)$, and 
\begin{equation} \label{def Zetone}
Z(\vphi) :=  (Z_1, Z_2, Z_3) (\vphi) := {\cal F}(i_0, \zeta_0) (\vphi) =
\om \cdot \pa_\vphi i_0(\vphi) - X_{H_{\e, \zeta_0}}(i_0(\vphi)) \, .
\end{equation}
\end{itemize}
\begin{lemma}  \label{zeta = 0}
$ |\zeta_0|^{\Lipg} \leq C \| Z \|_{s_0}^{\Lipg}$ \!\!\!. If $ {\cal F}(i_0, \zeta_0) = 0 $ then $ \zeta_0 = 0  $, 
 namely the torus $i_0 $ is invariant  for $X_{H_\e}$.
 \end{lemma}

\begin{proof}
It is proved in \cite{BB13} the formula 
$$
\zeta_0 = 
\int_{\T^\nu} - [\pa_\vphi y_0 (\vphi)]^T  Z_1 (\vphi) 
+ [\pa_\vphi \theta_0 (\vphi)]^T Z_2 (\vphi) 
- [\pa_\vphi z_0 (\vphi) ]^T \pa_x^{-1} Z_3 (\vphi) \, d \vphi \, .
$$
Hence the lemma follows by \eqref{ansatz 0} and usual algebra estimate. 
\end{proof}

We now quantify the size of $ i_0^* {\cal W} $ in terms of $ Z $.

\begin{lemma}
The coefficients $A_{kj} (\vphi) $ in \eqref{def Akj} satisfy 
\begin{equation}\label{stima A ij}
\| A_{k j} \|_s^{\Lipg} \leq_s \gamma^{-1} \big(\| Z \|_{s+2\t+2}^{\Lipg} \|  {\mathfrak I}_0 \|_{s_0+ 1}^{\Lipg}  
+ \| Z \|_{s_0+1}^{\Lipg} \|  {\mathfrak I}_0 \|_{s+ 2 \tau + 2}^{\Lipg} \big)\,.
\end{equation}
\end{lemma}

\begin{proof}
We estimate the coefficients of the Lie derivative
$ L_\omega (i_0^* {\cal W}) := \sum_{k < j} {\cal D}_\om A_{k j}(\vphi) d \vphi_k \wedge d \vphi_j $. 
Denoting  by $ \underline{e}_k  $ the $ k $-th versor of $ \R^\nu $ we have  
$$
 {\cal D}_\om A_{k j} 
=  L_\omega( i_0^* {\cal W})(\vphi)[ \underline{e}_k , \underline{e}_j ] 
= {\cal W}\big( \pa_\ph Z(\vphi) \underline{e}_k ,  \pa_\ph i_0(\vphi)  \underline{e}_j \big) 
+ {\cal W} \big(\pa_\ph i_0(\vphi) \underline{e}_k , \pa_\ph Z(\vphi) \underline{e}_j \big) 
$$
(see \cite{BB13}). 
Hence
\begin{equation} \label{bella trovata}
\| {\cal D}_\om A_{k j} \|_s^\Lipg 
\leq_s \| Z \|_{s+1}^\Lipg \| {\mathfrak I}_0 \|_{s_0 + 1}^\Lipg 
+ \| Z \|_{s_0 + 1}^\Lipg \| {\mathfrak I}_0 \|_{s + 1}^\Lipg  \,. 
\end{equation}
The bound \eqref{stima A ij} follows applying $ {\cal D}_\om^{-1}$ and using 
\eqref{def Akj}, \eqref{Dom inverso}.  
\end{proof}

As in \cite{BB13} we first modify the approximate torus $ i_0 $ to obtain an isotropic torus $ i_\d $ which is 
still approximately invariant. We denote the Laplacian $ \Delta_\vphi := \sum_{k=1}^\nu \partial_{\vphi_k}^2 $ .  

\begin{lemma}\label{toro isotropico modificato} {\bf (Isotropic torus)} 
The torus $ i_\delta(\vphi) := (\theta_0(\vphi), y_\delta(\vphi), z_0(\vphi) ) $ defined by 
\begin{equation}\label{y 0 - y delta}
y_\d := y_0 +  [\pa_\ph \theta_0(\vphi)]^{- T}  \rho(\vphi) \, , \qquad 
\rho_j(\vphi) := \Delta_\vphi^{-1} {\mathop\sum}_{ k = 1}^\nu \partial_{\vphi_j} A_{k j}(\vphi) 
\end{equation}
is isotropic. 
If \eqref{ansatz 0} holds, then, for some $ \s := \s(\nu,\t) $,   
\begin{align} \label{stima y - y delta}
\| y_\delta - y_0 \|_s^{\Lipg} 
& \leq_s  \gamma^{-1} \big(\| Z \|_{s + \s}^{\Lipg} \|  {\mathfrak I}_0 \|_{s_0 + \s}^{\Lipg} + 
\| Z \|_{s_0 + \s}^{\Lipg} \|  {\mathfrak I}_0 \|_{s + \s}^{\Lipg} \big) \,,
\\
\label{stima toro modificato}
\| {\cal F}(i_\delta, \zeta_0) \|_s^{\Lipg} 
& \leq_s  \| Z \|_{s + \s}^{\Lipg}  +  \| Z \|_{s_0 + \s}^{\Lipg} \|  {\mathfrak I}_0 \|_{s + \s}^{\Lipg} \\
\label{derivata i delta}
\| \pa_i [ i_\d][ \widehat \imath ] \|_s & \leq_s \| \widehat \imath \|_s +  \| {\mathfrak I}_0\|_{s + \s} \| \widehat \imath  \|_s \, .
\end{align}
\end{lemma}
In the paper we denote equivalently the differential by $ \partial_i $ or  $ d_i $. Moreover we denote 
by $ \s := \s(\nu, \tau ) $ possibly different (larger) ``loss of derivatives"  constants. 

\begin{proof}
In this proof we write $\| \ \|_s$ to denote $\| \ \|_s^{\Lipg}$.
The proof of the isotropy of $i_\d$ is in \cite{BB13}. 
The estimate \eqref{stima y - y delta} follows by \eqref{y 0 - y delta}, \eqref{stima A ij}, \eqref{ansatz 0}
and  the tame bound for the inverse 
$ \Vert [\pa_\ph \theta_0 ]^{-T}\Vert_{s} \leq_s 1 +  \|  {\mathfrak I}_0 \|_{s + 1} $. 
It remains to estimate the difference (see \eqref{operatorF} and note that $ X_{\cal N} $ does not depend on $ y $)
\begin{equation}\label{F diff}
{\cal F}(i_\delta, \zeta_0) - {\cal F}(i_0, \zeta_0) =  
\begin{pmatrix} 0 \\ {\cal D}_\om (y_\delta  - y_0 )  \\ 0 \end{pmatrix}\, 
+ X_P(i_\delta ) - X_P(i_0).
\end{equation}
Using \eqref{stime linearizzato campo hamiltoniano}, \eqref{D yy P}, we get 
$ \| \partial_y X_P (i) \|_s \leq_s \e^{2b} + \e^{2b - 1} \| {\mathfrak I}\|_{s + 3} $.
Hence \eqref{stima y - y delta}, 
\eqref{ansatz 0}  imply
\begin{equation}\label{XP delta - XP}
\|X_ P(i_\delta ) - X_P(i_0 )\|_s 
\leq_s    \|  {\mathfrak I}_0 \|_{s_0 + \s} \|Z \|_{s + \s} + \|  {\mathfrak I}_0 \|_{s + \s} \|Z \|_{s_0 + \s}   \, . 
\end{equation}
Differentiating  \eqref{y 0 - y delta} we have  
\begin{equation}\label{D omega y 0 - y delta}
{\cal D}_\om (y_\delta - y_0 ) 
= [\pa_\ph \theta_0(\vphi)]^{-T} {\cal D}_\om \rho(\vphi) 
+ ( {\cal D}_\om [\pa_\ph \theta_0(\vphi)]^{-T} ) \rho(\vphi)  
\end{equation}
and 
$ {\cal D}_\om  \rho_j(\vphi) = \Delta^{-1}_\vphi \sum_{k = 1}^\nu \partial_{\vphi_j} {\cal D}_\om  A_{kj}(\vphi) $. 
Using \eqref{bella trovata},  we deduce that 
\begin{equation}\label{primo pezzo D omega y 0 - y delta}
\|  [\pa_\ph \theta_0 ]^{-T} {\cal D}_\om \rho \|_s 
\leq_s \| Z\|_{s + 1} \| {\mathfrak I}_0\|_{s_0 + 1} + \| Z \|_{s_0 + 1}  \| {\mathfrak I}_0\|_{s + 1}\,.
\end{equation}
To estimate the second term in \eqref{D omega y 0 - y delta}, 
we differentiate 
$ Z_1(\vphi) = {\cal D}_\om \theta_0(\vphi) - \om -  (\partial_y P)(i_0(\vphi)) $
(which is the first component in \eqref{operatorF}) with respect to  $\ph$. We get
$ {\cal D}_\om \pa_\ph \theta_0(\vphi) 
= \pa_\vphi (\pa_y P)(i_0(\vphi)) + \partial_\vphi Z_1(\vphi) $. 
Then, by \eqref{D y P},
\begin{equation}\label{bella trovata 2}
\|{\cal D}_\om [\pa_\ph \theta_0]^T \|_s 
\leq_s  \e^4 + \e^{2b} \| {\mathfrak I}_0 \|_{s + 2}   + \| Z \|_{s + 1}\,.
\end{equation}
Since 
$ {\cal D}_\om [ \pa_\ph \theta_0(\vphi)]^{-T}  
= - [ \pa_\ph \theta_0(\vphi)]^{-T} 
\big( {\cal D}_\om [\pa_\ph \theta_0(\vphi)]^T \big) [\pa_\ph \theta_0(\vphi)]^{-T} $, 
the bounds \eqref{bella trovata 2}, \eqref{stima A ij}, 
\eqref{ansatz 0} imply
\begin{equation}\label{secondo pezzo D omega y 0 - y delta}
\| ({\cal D}_\om [ \pa_\ph \theta_0]^{-T} )  \rho \|_s 
\leq_s \e^{6-2b} \g^{-1} 
\big( \| Z\|_{s + \s}  \| {\mathfrak I}_0\|_{s_0 + \sigma} + \| Z \|_{s_0 + \s}  \| {\mathfrak I}_0\|_{s + \s}\big) \, . 
\end{equation}
In conclusion \eqref{F diff}, \eqref{XP delta - XP}, \eqref{D omega y 0 - y delta},
\eqref{primo pezzo D omega y 0 - y delta}, \eqref{secondo pezzo D omega y 0 - y delta} imply 
\eqref{stima toro modificato}. The bound  \eqref{derivata i delta} follows by 
\eqref{y 0 - y delta}, \eqref{def Akj}, \eqref{coefficienti pull back di Lambda}, \eqref{ansatz 0}.
\end{proof}

Note that there is no $ \g^{- 1} $ in the right hand side of \eqref{stima toro modificato}.
It turns out that an approximate inverse of $d_{i, \zeta} {\cal F}(i_\delta )$
 is an approximate inverse of $d_{i, \zeta} {\cal F}(i_0 )$ as well.  
In order to find an approximate inverse of the linearized operator $d_{i, \zeta} {\cal F}(i_\delta )$ 
we introduce a suitable set of symplectic coordinates nearby the isotropic torus $ i_\d $. 
We consider the map
$ G_\delta : (\psi, \eta, w) \to (\theta, y, z)$ of the phase space $\T^\nu \times \R^\nu \times H_S^\bot$ defined by
\begin{equation}\label{trasformazione modificata simplettica}
\begin{pmatrix}
\theta \\
y \\
z
\end{pmatrix} := G_\delta \begin{pmatrix}
\psi \\
\eta \\
w
\end{pmatrix} := 
\begin{pmatrix}
\theta_0(\psi) \\
y_\delta (\psi) + [\pa_\psi \theta_0(\psi)]^{-T} \eta + \big[ (\pa_\teta \tilde{z}_0) (\theta_0(\psi)) \big]^T \partial_x^{-1} w \\
z_0(\psi) + w

\end{pmatrix} 
\end{equation}
where $\tilde{z}_0 (\theta) := z_0 (\theta_0^{-1} (\theta))$. 
It is proved in \cite{BB13} that $ G_\delta $ is symplectic, using that the torus $ i_\d $ is isotropic 
(Lemma \ref{toro isotropico modificato}).
In the new coordinates,  $ i_\delta $ is the trivial embedded torus
$ (\psi , \eta , w ) = (\psi , 0, 0 ) $.  
The transformed Hamiltonian $ K := K(\psi, \eta, w, \zeta_0) $ 
is (recall \eqref{hamiltoniana modificata})
\begin{align} 
K := H_{\e, \zeta_0} \circ G_\d  
& = \theta_0(\psi) \cdot \zeta_0 + K_{00}(\psi) + K_{10}(\psi) \cdot \eta + (K_{0 1}(\psi), w)_{L^2(\T)} + 
\frac12 K_{2 0}(\psi)\eta \cdot \eta 
\nonumber \\ & 
\quad +  \big( K_{11}(\psi) \eta , w \big)_{L^2(\T)} 
+ \frac12 \big(K_{02}(\psi) w , w \big)_{L^2(\T)} + K_{\geq 3}(\psi, \eta, w)  
\label{KHG}
\end{align}
where $ K_{\geq 3} $ collects the terms at least cubic in the variables $ (\eta, w )$.
At any fixed $\psi $, 
the Taylor coefficient $K_{00}(\psi) \in \R $,  
$K_{10}(\psi) \in \R^\nu $,  
$K_{01}(\psi) \in H_S^\bot$ (it is a function of $ x  \in \T $),
$K_{20}(\psi) $ is a $\nu \times \nu$ real matrix, 
$K_{02}(\psi)$ is a linear self-adjoint operator of $ H_S^\bot $ and 
$K_{11}(\psi) : \R^\nu \to H_S^\bot$. Note that the above Taylor coefficients do not
depend on the parameter $ \zeta_0 $.

The Hamilton equations associated to \eqref{KHG}  are 
\begin{equation}\label{sistema dopo trasformazione inverso approssimato}
\begin{cases}
\dot \psi \hspace{-30pt} & = K_{10}(\psi) +  K_{20}(\psi) \eta + 
K_{11}^T (\psi) w + \partial_{\eta} K_{\geq 3}(\psi, \eta, w)
\\
\dot \eta \hspace{-30pt} & =- [\partial_\psi \theta_0(\psi)]^T \zeta_0 - 
\partial_\psi K_{00}(\psi) - [\partial_{\psi}K_{10}(\psi)]^T  \eta - 
[\partial_{\psi} K_{01}(\psi)]^T  w  
\\
& \quad -
\partial_\psi \big( \frac12 K_{2 0}(\psi)\eta \cdot \eta + ( K_{11}(\psi) \eta , w )_{L^2(\T)} + 
\frac12 ( K_{02}(\psi) w , w )_{L^2(\T)} + K_{\geq 3}(\psi, \eta, w) \big)
\\
\dot w \hspace{-30pt} & = \partial_x \big( K_{01}(\psi) + 
K_{11}(\psi) \eta +  K_{0 2}(\psi) w + \nabla_w K_{\geq 3}(\psi, \eta, w) \big) 
\end{cases} 
\end{equation}
where $ [\partial_{\psi}K_{10}(\psi)]^T $ is the $ \nu \times \nu $ transposed matrix and 
$ [\partial_{\psi}K_{01}(\psi)]^T $,  $ K_{11}^T(\psi) : {H_S^\bot \to \R^\nu} $ are defined by the 
duality relation $ ( \partial_{\psi} K_{01}(\psi) [\hat \psi ],  w)_{L^2}  = \hat \psi \cdot [\partial_{\psi}K_{01}(\psi)]^T w  $,
$ \forall \hat \psi \in \R^\nu, w \in H_S^\bot $, 
and similarly for $ K_{11} $. 
Explicitly, for all  $ w \in H_S^\bot $, 
and denoting $\underline{e}_k$ the $k$-th versor of $\R^\nu$, 
\begin{equation} \label{K11 tras}
K_{11}^T(\psi) w =  {\mathop \sum}_{k=1}^\nu \big(K_{11}^T(\psi) w \cdot \underline{e}_k\big) \underline{e}_k   =
{\mathop \sum}_{k=1}^\nu  
\big( w, K_{11}(\psi) \underline{e}_k  \big)_{L^2(\T)}  \underline{e}_k  \, \in \R^\nu \, .  
\end{equation}

In the next lemma we estimate the coefficients $ K_{00} $, $ K_{10} $, $K_{01} $ 
in the Taylor expansion \eqref{KHG}.
Note that on an exact solution we have $ Z  = 0 $ and therefore
$ K_{00} (\psi) = {\rm const} $, $ K_{10}  = \om $ and $ K_{01}  = 0 $. 

\begin{lemma} \label{coefficienti nuovi}
Assume \eqref{ansatz 0}. Then there is $ \s := \s(\tau, \nu)$  such that 
\[ 
\|  \partial_\psi K_{00} \|_s^{\Lipg} 
+ \| K_{10} - \om  \|_s^{\Lipg} +  \| K_{0 1} \|_s^{\Lipg} 
\leq_s  \| Z \|_{s + \s}^{\Lipg}  +  \| Z \|_{s_0 + \s}^{\Lipg} \| {\mathfrak I}_0 \|_{s + \s}^{\Lipg}\,.
\]
\end{lemma}

\begin{proof}
Let $ {\cal F}(i_\d, \zeta_0) := Z_\d := (Z_{1,\d}, Z_{2,\d}, Z_{3,\d}) $. 
By a direct calculation as in \cite{BB13} (using \eqref{KHG}, \eqref{operatorF})
\begin{align*}
\partial_\psi K_{00}(\psi) & =  
- [ \pa_\psi \teta_0 (\psi) ]^T \big( - Z_{2, \d} - 
[ \pa_\psi y_\d] [ \pa_\psi \teta_0]^{-1} Z_{1, \d}   
+ [ (\pa_\theta {\tilde z}_0)( \teta_0 (\psi)) ]^T \partial_x^{-1} Z_{3,\d} 
\\ 
&  \quad + [ (\pa_\theta {\tilde z}_0)(\teta_0 (\psi)) ]^T \partial_x^{-1} \partial_\psi z_0 (\psi) [ \pa_\psi \teta_0 (\psi)]^{-1} Z_{1,\d} \big) \, ,  
\\
K_{10}(\psi) & 
=  \omega -   [ \pa_\psi \theta_0(\psi)]^{-1} Z_{1,\d}(\psi) \,, 
\\
K_{01}(\psi) 
& = - \partial_x^{-1} Z_{3,\d} + \partial_x^{-1} \pa_\psi z_0(\psi) [\pa_\psi \theta_0(\psi)]^{-1} Z_{1,\d}(\psi)\,.   
\end{align*}
Then \eqref{ansatz 0},  \eqref{stima y - y delta}, \eqref{stima toro modificato} (using Lemma \ref{lemma:utile}) imply  the lemma.
\end{proof}

\begin{remark} \label{rem:KAM normal form}  
If $ {\cal F} (i_0, \zeta_0) = 0 $ then $\zeta_0 = 0$  by 
Lemma \ref{zeta = 0}, and Lemma \ref{coefficienti nuovi} implies that
\eqref{KHG} simplifies to 
$ K  = const + \om \cdot \eta + \frac12 K_{2 0}(\psi)\eta \cdot \eta 
+ \big( K_{11}(\psi) \eta , w \big)_{L^2(\T)} 
+ \frac12 \big(K_{02}(\psi) w , w \big)_{L^2(\T)} + K_{\geq 3} $. 
\end{remark}

We now estimate  $ K_{20}, K_{11}$ in \eqref{KHG}. 
The norm of $K_{20}$ is the sum of the norms of its matrix entries.  

\begin{lemma} \label{lemma:Kapponi vari}
Assume \eqref{ansatz 0}. Then 
\begin{align}\label{stime coefficienti K 20 11 bassa}
\|K_{20} + 3 \e^{2 b} I \|_s^{\Lipg} 
& \leq_s \e^{2b+2} + \e^{2b} \| {\mathfrak I}_0\|_{s + \s}^{\Lipg}  +  \e^{3} \g^{-1} \| {\mathfrak I}_0\|_{s_0 + \s}^{\Lipg}  \| Z \|_{s + \s}^{\Lipg} 
\\ 
\label{stime coefficienti K 11 alta}  
\| K_{11} \eta \|_s^{\Lipg} 
& \leq_s \e^{5} \g^{-1} \| \eta \|_s^{\Lipg} 
+ \e^{2 b - 1} ( \| {\mathfrak I}_0\|_{s + \s}^{\Lipg} + \g^{-1} \| {\mathfrak I}_0\|_{s_0 + \s}^{\Lipg}  \| Z \|_{s + \s}^{\Lipg} ) 
\| \eta \|_{s_0}^{\Lipg} 
\\
\label{stime coefficienti K 11 alta trasposto}  
\| K_{11}^T w \|_s^{\Lipg} 
& \leq_s \e^{5} \g^{-1} \| w \|_{s + 2}^{\Lipg} 
+ \e^{2 b - 1} ( \| {\mathfrak I}_0\|_{s + \s}^{\Lipg} + \g^{-1} \| {\mathfrak I}_0\|_{s_0 + \s}^{\Lipg}  \| Z \|_{s + \s}^{\Lipg} ) 
\| w \|_{s_0 + 2}^{\Lipg} \, . 
\end{align}
In particular 
$ \| K_{20} + 3 \e^{2 b} I  \|_{s_0 }^{\Lipg}  \leq C \e^{6} \gamma^{-1} $, and 
$$
 \| K_{11} \eta \|_{s_0 }^{\Lipg}  \leq C \e^{5} \g^{-1}\| \eta \|_{s_0}^{\Lipg} , 
 \quad  \| K_{11}^T w \|_{s_0 }^{\Lipg}  \leq C \e^{5} \g^{-1} \| w \|_{s_0}^{\Lipg} \, .
$$ 
\end{lemma}

\begin{proof}
To shorten the notation, in this proof we write $\| \ \|_s$ for $\| \ \|_s^{\Lipg}$.
We have
$$
K_{2 0}(\vphi) = [\pa_\ph \theta_0(\vphi)]^{-1} \partial_{yy} H_\e(i_\delta(\vphi)) 
[\pa_\ph \theta_0(\vphi)]^{-T} 
= [\pa_\ph \theta_0(\vphi)]^{-1} \partial_{yy} P(i_\delta(\vphi)) 
[\pa_\ph \theta_0(\vphi)]^{-T}.
$$
Then \eqref{D yy P}, \eqref{ansatz 0}, \eqref{stima y - y delta} 
imply \eqref{stime coefficienti K 20 11 bassa}.
Now (see also \cite{BB13})
\begin{align*}
K_{11}(\vphi) & = 
\partial_{y} \nabla_z H_\e (i_\delta(\vphi)) [\pa_\ph \theta_0 (\vphi)]^{-T} 
- \partial_x^{-1} (\pa_\theta {\tilde z}_0) (\teta_0(\vphi)) (\partial_{yy} H_\e) (i_\delta(\vphi)) [\pa_\ph \theta_0 (\vphi)]^{-T}   \nonumber\\
& \stackrel{\eqref{Hamiltoniana Heps KdV}} = 
 \partial_{y} \nabla_z P(i_\delta(\vphi)) [\pa_\ph \theta_0 (\vphi)]^{-T} 
- \partial_x^{-1} (\pa_\theta {\tilde z}_0) (\teta_0(\vphi)) (\partial_{yy} P) (i_\delta(\vphi)) [\pa_\ph \theta_0 (\vphi)]^{-T}\,,
\end{align*}
therefore, using \eqref{stime linearizzato campo hamiltoniano}, \eqref{D yy P}, \eqref{ansatz 0}, 
we deduce  \eqref{stime coefficienti K 11 alta}. 
The bound \eqref{stime coefficienti K 11 alta trasposto} for $K_{11}^T$ follows by \eqref{K11 tras}.
\end{proof}

Under the linear change of variables 
\begin{equation}\label{DGdelta}
D G_\delta(\vphi, 0, 0) 
\begin{pmatrix}
\widehat \psi \, \\
\widehat \eta \\
\widehat w
\end{pmatrix} 
:= 
\begin{pmatrix}
\pa_\psi \theta_0(\vphi) & 0 & 0 \\
\pa_\psi y_\delta(\vphi) & [\pa_\psi \theta_0(\vphi)]^{-T} & 
- [(\pa_\theta \tilde{z}_0)(\theta_0(\vphi))]^T \partial_x^{-1} \\
\pa_\psi z_0(\vphi) & 0 & I
\end{pmatrix}
\begin{pmatrix}
\widehat \psi \, \\
\widehat \eta \\
\widehat w
\end{pmatrix} 
\end{equation}
the linearized operator  $d_{i, \zeta}{\cal F}(i_\delta )$ 
transforms (approximately, see \eqref{verona 2}) into the operator obtained linearizing 
\eqref{sistema dopo trasformazione inverso approssimato} at $(\psi, \eta , w, \zeta ) = (\vphi, 0, 0, \zeta_0 )$
(with $ \partial_t \rightsquigarrow {\cal D}_\om $), namely 
\begin{equation}\label{lin idelta}
\hspace{-5pt}
\begin{pmatrix}
{\cal D}_\om \widehat \psi - \partial_\psi K_{10}(\vphi)[\widehat \psi \, ] - 
K_{2 0}(\vphi)\widehat \eta - K_{11}^T (\vphi) \widehat w \\
 {\cal D}_\om  \widehat \eta + [\partial_\psi \theta_0(\vphi)]^T \widehat \zeta + 
\pa_\psi [\partial_\psi \theta_0(\vphi)]^T [ \widehat \psi, \zeta_0] + \partial_{\psi\psi} K_{00}(\vphi)[\widehat \psi]  + 
[\partial_\psi K_{10}(\vphi)]^T \widehat \eta + 
[\partial_\psi  K_{01}(\vphi)]^T \widehat w   \\ 
{\cal D}_\om  \widehat w - \partial_x 
\{ \partial_\psi K_{01}(\vphi)[\widehat \psi] + K_{11}(\vphi) \widehat \eta + K_{02}(\vphi) \widehat w \}
\end{pmatrix} \! .  \hspace{-5pt}
\end{equation}
We now estimate the induced composition operator. 
\begin{lemma} \label{lemma:DG}
Assume \eqref{ansatz 0} and let 
$ \widehat \imath := (\widehat \psi, \widehat \eta, \widehat w)$. 
Then 
\begin{gather} \label{DG delta}
\|DG_\delta(\vphi,0,0) [\widehat \imath] \|_s + \|DG_\delta(\vphi,0,0)^{-1} [\widehat \imath] \|_s 
\leq_s \| \widehat \imath \|_{s} + ( \| {\mathfrak I}_0 \|_{s + \s} + 
\gamma^{-1}  \| {\mathfrak I}_0 \|_{s_0 + \s} \|Z \|_{s + \s} ) \| \widehat \imath \|_{s_0}\,,
\\ 
\| D^2 G_\delta(\vphi,0,0)[\widehat \imath_1, \widehat \imath_2] \|_s 
\leq_s  \| \widehat \imath_1\|_s \| \widehat \imath_2 \|_{s_0} 
+ \| \widehat \imath_1\|_{s_0} \| \widehat \imath_2 \|_{s} 
+ ( \| {\mathfrak I}_0  \|_{s + \s} + \gamma^{-1}  \| {\mathfrak I}_0 \|_{s_0 + \s} \| Z\|_{s + \s} ) \|\widehat \imath_1 \|_{s_0} \| \widehat \imath_2\|_{s_0}
\notag 
\end{gather}
for some $\s := \s(\nu,\t)$. 
Moreover the same estimates hold if we replace the norm $\| \ \|_s$ with $\| \ \|_s^{\Lipg}$.
\end{lemma}

\begin{proof}
The estimate \eqref{DG delta} for $D G_\delta(\vphi,0,0)$ follows by \eqref{DGdelta} and  \eqref{stima y - y delta}.
By \eqref{ansatz 0},
$ \| (DG_\delta(\vphi,0,0) - I) \widehat \imath \|_{s_0} 
\leq $ $ C \e^{6 - 2b} \gamma^{-1} \| \widehat \imath \|_{s_0} 
\leq \| \widehat \imath \|_{s_0} / 2 $.
Therefore $DG_\delta(\vphi,0,0)$ is invertible and, by Neumann series, the inverse 
satisfies \eqref{DG delta}. 
The bound for $D^2 G_\delta$ follows by differentiating $DG_\d$. 
\end{proof}

In order to construct an approximate inverse of \eqref{lin idelta} it is sufficient to solve the equation
\begin{equation}\label{operatore inverso approssimato} 
{\mathbb D} [\widehat \psi, \widehat \eta, \widehat w, \widehat \zeta ] := 
  \begin{pmatrix}
{\cal D}_\om \widehat \psi - 
K_{20}(\vphi) \widehat \eta - K_{11}^T(\vphi) \widehat w\\
{\cal D}_\om  \widehat \eta + [\partial_\psi \theta_0(\vphi)]^T \widehat \zeta  \\
{\cal D}_\om \widehat w  - \partial_x K_{11}(\vphi)\widehat \eta -  \partial_x K_{0 2}(\vphi) \widehat w 
\end{pmatrix} =
\begin{pmatrix}
g_1 \\ g_2 \\ g_3 
\end{pmatrix}
\end{equation}
which is obtained by neglecting in \eqref{lin idelta} 
the terms $ \partial_\psi K_{10} $, $ \partial_{\psi \psi} K_{00} $, $ \partial_\psi K_{00} $, $ \partial_\psi K_{01} $ and
$ \pa_\psi [\partial_\psi \theta_0(\vphi)]^T [ \cdot , \zeta_0] $
(which are naught at a solution by Lemmata \ref{coefficienti nuovi} and \ref{zeta = 0}). 

First we solve the second equation in \eqref{operatore inverso approssimato}, namely 
$ {\cal D}_\om  \widehat \eta  = g_2  -  [\partial_\psi \theta_0(\vphi)]^T \widehat \zeta $. 
We choose $ \widehat \zeta $ so that the $\vphi$-average of the right hand side 
is zero, namely 
\begin{equation}\label{fisso valore di widehat zeta}
\widehat \zeta = \langle g_2 \rangle 
\end{equation}
(we denote $ \langle g \rangle := (2 \pi)^{- \nu} \int_{\T^\nu} g (\vphi) d \vphi $). 
Note that the $\ph$-averaged matrix  $ \langle [\partial_\psi \theta_0 ]^T  \rangle = 
\langle I + [\pa_\psi \Theta_0]^T \rangle = I $ because $\theta_0(\ph) = \ph + \Theta_0(\ph)$ and 
$\Theta_0(\ph)$ is a periodic function.
Therefore 
\begin{equation}\label{soleta}
\widehat \eta := {\cal D}_\om^{-1} \big(
g_2 - [\partial_\psi \theta_0(\ph) ]^T \langle g_2 \rangle \big) + 
\langle \widehat \eta \rangle \, , \quad \langle \widehat \eta \rangle \in \R^\nu \, , 
\end{equation}
where the average $\la \widehat \eta \ra$ will be fixed below. 
Then we consider the third equation
\begin{equation}\label{cal L omega}
{\cal L}_\om \widehat w = g_3 + \partial_x K_{11}(\vphi) \widehat \eta\,,  \
\quad  {\cal L}_\om := \om \cdot \partial_\vphi -  \partial_x K_{0 2}(\vphi) \, .
\end{equation}

\begin{itemize}
\item {\sc Inversion assumption.} 
{\it There exists a set $ \Omega_\infty \subset \Omega_o$ such that  
for all $ \omega \in \Omega_\infty $, for every function
$ g \in H^{s+\mu}_{S^\bot} (\T^{\nu+1}) $  there 
exists a solution $ h :=  {\cal L}_\om^{- 1} g  \in H^{s}_{S^\bot} (\T^{\nu+1})  $ 
of the linear equation $ {\cal L}_\om h = g $
which satisfies}
\begin{equation}\label{tame inverse}
\| {\cal L}_\om^{- 1} g \|_s^{\Lipg} \leq C(s) \g^{-1} 
\big(  \| g \|_{s + \mu}^{\Lipg} + \e \gamma^{-1} 
\big\{ \| {\mathfrak I}_0 \|_{s + \mu}^{\Lipg} + \g^{-1} \| {\mathfrak I}_0 \|_{s_0 + \mu}^{\Lipg} \| Z \|_{s + \mu}^{\Lipg} \big\} \|g \|_{s_0}^{\Lipg}  \big) 
\end{equation}
\emph{for some $ \mu := \mu (\tau, \nu) >  0 $}. 
\end{itemize}

\begin{remark}
The term $ \e \gamma^{-1} 
\{ \| {\mathfrak I}_0 \|_{s + \mu}^{\Lipg} + \g^{-1} \| {\mathfrak I}_0 \|_{s_0 + \mu}^{\Lipg} \| Z \|_{s + \mu}^{\Lipg} \} $ arises because
the remainder $ R_6 $ in section \ref{step5} contains the term 
$ \e ( \| \Theta_0 \|_{s + \mu}^{\Lipg} + \| y_\d \|_{s + \mu}^{\Lipg}) $ 
$\leq \e  \| {\mathfrak I}_\d \|_{s + \mu}^{\Lipg} $, see Lemma \ref{lemma L6}. 
\end{remark}

By the above assumption  there exists a solution 
\begin{equation}\label{normalw}
\widehat w := {\cal L}_\om^{-1} [ g_3 + \partial_x K_{11}(\vphi) \widehat \eta \, ] 
\end{equation}
of \eqref{cal L omega}. 
Finally, we solve the first equation in \eqref{operatore inverso approssimato}, 
which, substituting \eqref{soleta}, \eqref{normalw}, becomes
\begin{equation}\label{equazione psi hat}
{\cal D}_\om \widehat \psi  = 
g_1 +  M_1(\vphi) \langle \widehat \eta \rangle + M_2(\vphi) g_2 + M_3(\vphi) g_3 - 
M_2(\vphi)[\pa_\psi \theta_0]^T \langle g_2 \rangle \,,
\end{equation}
where
\be \label{cal M2}
M_1(\vphi) := K_{2 0}(\vphi) + K_{11}^T(\vphi) {\cal L}_\omega^{-1} \partial_x K_{11}(\vphi)\,, \quad
M_2(\vphi) :=  M_1 (\vphi)  {\cal D}_\om^{-1} \, , \quad 
M_3(\vphi) :=  K_{11}^T (\vphi) {\cal L}_\om^{-1} \, .  
\ee
In order to solve the equation \eqref{equazione psi hat} we have 
to choose $\langle \widehat \eta \rangle$ such that the right hand side in \eqref{equazione psi hat} has zero average.  
By Lemma \ref{lemma:Kapponi vari} and \eqref{ansatz 0}, the $\ph$-averaged matrix 
$ \langle M_1 \rangle =- 3 \e^{2 b} I + O( \e^{10} \gamma^{-3}) $.  
Therefore, for $ \e $ small,  $\langle M_1 \rangle$ is invertible and $\langle M_1 \rangle^{-1} = O(\e^{-2 b}) = O(\gamma^{- 1})$ 
(recall \eqref{link gamma b}). Thus we define 
\begin{equation}\label{sol alpha}
\langle \widehat \eta \rangle  := - \langle M_1 \rangle^{-1} 
[ \langle g_1 \rangle + \langle M_2 g_2 \rangle + \langle M_3 g_3 \rangle - 
\langle M_2 [\pa_\psi \theta_0]^T   \rangle  \langle g_2 \rangle ].
\end{equation}
With this choice of $\langle \widehat \eta \rangle$ the equation \eqref{equazione psi hat} has the solution
\begin{equation}\label{sol psi}
\widehat \psi :=
{\cal D}_\om^{-1} [ g_1 + M_1(\vphi) \langle \widehat \eta \rangle + M_2(\vphi) g_2 + M_3(\vphi) g_3 -
M_2(\vphi)[\pa_\psi \theta_0]^T \langle g_2 \rangle ].
\end{equation}
In conclusion, we have constructed 
a solution  $(\widehat \psi, \widehat \eta, \widehat w, \widehat \zeta)$ of the linear system \eqref{operatore inverso approssimato}. 

\begin{proposition}\label{prop: ai}
Assume \eqref{ansatz 0} and \eqref{tame inverse}. 
Then, $\forall \om \in \Omega_\infty $, $ \forall g := (g_1, g_2, g_3) $,
 the system \eqref{operatore inverso approssimato} has a solution 
$ {\mathbb D}^{-1} g := (\widehat \psi, \widehat \eta, \widehat w, \widehat \zeta ) $
where $(\widehat \psi, \widehat \eta, \widehat w, \widehat \zeta)$ are defined in 
\eqref{sol psi}, \eqref{soleta}, \eqref{sol alpha}, \eqref{normalw}, \eqref{fisso valore di widehat zeta}  satisfying
\begin{equation} \label{stima T 0 b}
\| {\mathbb D}^{-1} g \|_s^{{\rm Lip}(\gamma)}  
\leq_s \gamma^{-1} \big( \| g \|_{s + \mu}^{{\rm Lip}(\gamma)}  
+ \e \gamma^{-1} \big\{ \| {\mathfrak I}_0  \|_{s + \mu}^{{\rm Lip}(\gamma)}  
+  \g^{-1}  \| {\mathfrak I}_0  \|_{s_0 + \mu}^{{\rm Lip}(\gamma)} \|{\cal F}(i_0, \zeta_0) \|^{{\rm Lip}(\gamma)}_{s + \mu} \big\} \| g \|_{s_0 + \mu}^{{\rm Lip}(\gamma)}  \big).
\end{equation}
\end{proposition}

\begin{proof}

Recalling \eqref{cal M2}, by  Lemma \ref{lemma:Kapponi vari}, \eqref{tame inverse}, \eqref{ansatz 0}  we get $ \| M_2 h \|_{s_0} + \| M_3 h \|_{s_0}  \leq C \| h \|_{s_0 + \s} $. 
Then, by \eqref{sol alpha} and $\langle M_1 \rangle^{-1} = O(\e^{-2 b}) = O(\gamma^{-1}) $, 
we deduce 
$ |\langle \widehat \eta\rangle|^{\Lipg} \leq C\gamma^{-1} \| g \|_{s_0+ \s}^{\Lipg} $
and  \eqref{soleta}, \eqref{Dom inverso} imply 
$ \| \widehat \eta \|_s^{\Lipg} \leq_s \gamma^{-1} \big( \| g \|_{s + \s}^\Lipg + \| \fracchi_0 \|_{s + \s } \| g \|_{s_0}^\Lipg  \big)$.
The bound \eqref{stima T 0 b} is sharp for $ \widehat w $ because $ {\cal L}_\om^{-1} g_3 $ in  \eqref{normalw}
is estimated using \eqref{tame inverse}. Finally $  \widehat \psi $ 
satisfies \eqref{stima T 0 b} using
\eqref{sol psi},  \eqref{cal M2}, \eqref{tame inverse}, \eqref{Dom inverso} and Lemma \ref{lemma:Kapponi vari}.
\end{proof}
Finally we prove that the operator 
\begin{equation}\label{definizione T} 
{\bf T}_0 := (D { \widetilde G}_\delta)(\vphi,0,0) \circ {\mathbb D}^{-1} \circ (D G_\delta) (\vphi,0,0)^{-1}
\end{equation}
is an approximate right  inverse for $d_{i,\zeta} {\cal F}(i_0 )$ where
$ \widetilde{G}_\delta (\psi, \eta, w, \zeta) := $  $ \big( G_\delta (\psi, \eta, w), \zeta \big) $ 
 is the identity on the $ \zeta $-component. 
We denote the norm $ \| (\psi, \eta, w, \zeta) \|_s^\Lipg := $ $  \max \{  \| (\psi, \eta, w) \|_s^\Lipg, $ $ | \zeta |^\Lipg  \} $.

\begin{theorem} {\bf (Approximate inverse)} \label{thm:stima inverso approssimato}
Assume \eqref{ansatz 0} and the inversion assumption \eqref{tame inverse}. 
Then there exists $ \mu := \mu (\tau, \nu) >  0 $ such that, for all $ \om \in \Om_\infty $, 
for all $ g := (g_1, g_2, g_3) $,  
the operator $ {\bf T}_0 $ defined in \eqref{definizione T} satisfies  
\begin{equation}\label{stima inverso approssimato 1}
\| {\bf T}_0 g \|_{s}^{{\rm Lip}(\gamma)}  
\leq_s  \gamma^{-1} \big(\| g \|_{s + \mu}^{{\rm Lip}(\gamma)}  
+ \e \gamma^{-1} \big\{ \| {\mathfrak I}_0 \|_{s + \mu}^{{\rm Lip}(\gamma)}  
+\gamma^{-1} \| \fracchi_0 \|_{s_0 + \mu}^\Lipg 
\|{\cal F}(i_0, \zeta_0) \|_{s + \mu}^{{\rm Lip}(\gamma)}  \big\} 
\| g \|_{s_0 + \mu}^{{\rm Lip}(\gamma)}  \big). 
\end{equation}
It is an approximate inverse of $d_{i, \zeta} {\cal F}(i_0 )$, namely 
\begin{align}
& \| ( d_{i, \zeta} {\cal F}(i_0) \circ {\bf T}_0 - I ) g \|_s^{{\rm Lip}(\gamma)}  
\label{stima inverso approssimato 2} 
\\ 
& \leq_s \gamma^{-1} \Big( \| {\cal F}(i_0, \zeta_0) \|_{s_0 + \mu}^\Lipg \| g \|_{s + \mu}^\Lipg  
+ \big\{ \| {\cal F}(i_0, \zeta_0) \|_{s + \mu}^\Lipg 
+ \e \gamma^{-1} \| {\cal F}(i_0, \zeta_0) \|_{s_0 + \mu}^\Lipg \| {\mathfrak I}_0 \|_{s + \mu}^\Lipg \big\} \| g \|_{s_0 + \mu}^\Lipg \Big).
\nonumber
\end{align}
\end{theorem}

\begin{proof}
We denote $\| \ \|_s$ instead of $\| \ \|_s^{\Lipg}$. 
The bound \eqref{stima inverso approssimato 1} 
follows from \eqref{definizione T}, \eqref{stima T 0 b},  \eqref{DG delta}.
By \eqref{operatorF}, since $ X_\mN $ does not depend on $ y $,  
and $ i_\d $ differs from $ i_0 $ only for the $ y$ component,  
we have 
\begin{align} \label{verona 0}
d_{i, \zeta} {\cal F}(i_0 )[\, \widehat \imath, \widehat \zeta \, ]  - d_{i, \zeta} {\cal F}(i_\delta ) [\, \widehat \imath, \widehat \zeta \, ]
&  =   d_i X_P (i_\delta)  [\, \widehat \imath \, ]  - d_i X_P (i_0) [\, \widehat \imath \, ]
\\ & 
=  \int_0^1 \partial_y d_i X_P (\theta_0, y_0 + s (y_\delta - y_0), z_0) [y_\d - y_0,   \widehat \imath  \,  ] ds 
=: {\cal E}_0 [\, \widehat \imath, \widehat \zeta \, ]  \,. \nonumber
\end{align}
By \eqref{D yii}, \eqref{stima y - y delta}, \eqref{ansatz 0},  we estimate
\begin{equation}\label{stima parte trascurata 1}
\| {\cal E}_0 [\, \widehat \imath, \widehat \zeta \, ] \|_s \leq_s 
\| Z \|_{s_0 + \s} \| \widehat \imath \|_{s + \s} +  
\| Z \|_{s + \s} \| \widehat \imath \|_{s_0 + \s} +  \e^{2b-1}\g^{-1}
\| Z \|_{s_0 + \s} \| \widehat \imath \|_{s_0 + \s} \| \fracchi_0 \|_{s+\s} 
\end{equation}
where $Z := \mF(i_0, \zeta_0)$ (recall \eqref{def Zetone}). 
 Note that $\mE_0[\widehat \imath, \widehat \zeta]$ is, in fact, independent of $\widehat \zeta$.
Denote the set of variables $  (\psi, \eta, w) =: {\mathtt u} $. 
Under the transformation $G_\delta $, the nonlinear operator ${\cal F}$ in \eqref{operatorF} transforms into 
\be \label{trasfo imp}
{\cal F}(G_\delta(  {\mathtt u} (\vphi) ), \zeta ) 
= D G_\delta( {\mathtt u}  (\vphi) ) \big(  {\cal D}_\om {\mathtt u} (\vphi) - X_K ( {\mathtt u} (\vphi), \zeta)  \big) \, , \quad 
K = H_{\e, \zeta} \circ G_\delta   \, ,
\ee
see \eqref{sistema dopo trasformazione inverso approssimato}. 
Differentiating  \eqref{trasfo imp} at the trivial torus 
$ {\mathtt u}_\delta (\vphi) = G_\delta^{-1}(i_\delta) (\vphi) = (\ph, 0 , 0 ) $, 
at  $ \zeta = \zeta_0 $, 
in the directions $(\widehat {\mathtt u}, \widehat \zeta)
= (D G_\d ({\mathtt u}_\d)^{-1} [\, \widehat \imath \, ], \widehat \zeta) 
= D {\widetilde G}_\d ({\mathtt u}_\d)^{-1} [\, \widehat \imath , \widehat \zeta \, ] $, 
we get
\begin{align} \label{verona 2}
d_{i , \zeta} {\cal F}(i_\delta ) [\, \widehat \imath, \widehat \zeta \, ]
=  & D G_\delta( {\mathtt u}_\delta) 
\big( {\cal D}_\om \widehat {\mathtt u} 
- d_{\mathtt u, \zeta} X_K( {\mathtt u}_\delta, \zeta_0) [\widehat {\mathtt u}, \widehat \zeta \, ] 
\big) 
+ {\cal E}_1 [ \, \widehat \imath , \widehat \zeta \, ]\,,
\\
\label{E1}
{\cal E}_1 [\, \widehat \imath , \widehat \zeta \, ] 
:=  & 
D^2 G_\delta( {\mathtt u}_\delta) \big[ D G_\delta( {\mathtt u}_\delta)^{-1} {\cal F}(i_\delta, \zeta_0), \,  D G_\d({\mathtt u}_\d)^{-1} 
[ \, \widehat \imath \,  ] \big] \,, 
\end{align}
where  $ d_{\mathtt u, \zeta} X_K( {\mathtt u}_\delta, \zeta_0) $ is expanded in \eqref{lin idelta}.
In fact, ${\cal E}_1$ is independent of $\widehat \zeta$. 
We split  
\[ 
{\cal D}_\om \widehat {\mathtt u} 
- d_{\mathtt u, \zeta} X_K( {\mathtt u}_\delta, \zeta_0) [\widehat {\mathtt u}, \widehat \zeta] 
= \mathbb{D} [\widehat {\mathtt u}, \widehat \zeta \, ] + R_Z [ \widehat {\mathtt u}, \widehat \zeta \, ], 
\]
where $ {\mathbb D} [\widehat {\mathtt u}, \widehat \zeta] $ is defined in \eqref{operatore inverso approssimato} and 
\be\label{R0}
R_Z [  \widehat \psi, \widehat \eta, \widehat w, \widehat \zeta]
:= \begin{pmatrix}
 - \partial_\psi K_{10}(\vphi) [\widehat \psi ] \\
\pa_\psi [\partial_\psi \theta_0(\vphi)]^T [  \widehat \psi, \zeta_0] + \partial_{\psi \psi} K_{00} (\vphi) [ \widehat \psi ] + 
 [\partial_\psi K_{10}(\vphi)]^T \widehat \eta + 
 [\partial_\psi K_{01}(\vphi)]^T \widehat w  \\
 - \partial_x \{ \partial_{\psi} K_{01}(\vphi)[ \widehat \psi ] \}
 \end{pmatrix}
\ee
($R_Z$ is independent of $\widehat \zeta$). 
By \eqref{verona 0} and \eqref{verona 2}, 
\begin{equation} \label{E2}
d_{i, \zeta} {\cal F}(i_0 ) 
= D G_\delta({\mathtt u}_\delta) \circ {\mathbb D} \circ D {\widetilde G}_\delta ({\mathtt u}_\delta)^{-1} 
+ {\cal E}_0 + {\cal E}_1 + \mE_2 \,,
\quad
\mE_2 := D G_\delta( {\mathtt u}_\delta) \circ R_Z \circ D {\widetilde G}_\delta ({\mathtt u}_\delta)^{-1}  \, .
\end{equation}
By Lemmata \ref{coefficienti nuovi}, \ref{lemma:DG},    \ref{zeta = 0}, and \eqref{stima toro modificato}, \eqref{ansatz 0},
the terms $\mE_1, \mE_2 $ (see \eqref{E1}, \eqref{E2}, \eqref{R0})  satisfy 
the same bound \eqref{stima parte trascurata 1} as $\mE_0$
(in fact even better).
Thus the sum $\mE := \mE_0 + \mE_1 + \mE_2$ satisfies \eqref{stima parte trascurata 1}.
Applying $ {\bf T}_0 $ defined in \eqref{definizione T} to the right in \eqref{E2}, 
since $ {\mathbb D} \circ  {\mathbb D}^{-1} = I $ (see Proposition \ref{prop: ai}), 
we get $d_{i, \zeta} {\cal F}(i_0 ) \circ {\bf T}_0  - I 
= \mE \circ {\bf T}_0$. 
Then \eqref{stima inverso approssimato 2} follows from 
\eqref{stima inverso approssimato 1} and the bound \eqref{stima parte trascurata 1} for $\mE$. 
\end{proof}

\section{The linearized operator in the normal directions}\label{linearizzato siti normali}

The goal of this section is to write an explicit  expression of the linearized operator 
$\mL_\om$ defined in \eqref{cal L omega}, see Proposition \ref{prop:lin}. 
To this aim, we compute $ \frac12 ( K_{02}(\psi) w, w )_{L^2(\T)} $, $ w \in H_S^\bot$, 
which collects all the components of $(H_\e \circ G_\d)(\psi, 0, w)$ that are quadratic in $w$, see \eqref{KHG}.

\smallskip

We first prove some preliminary lemmata. 

\begin{lemma}\label{lemma astratto potente}
Let $ H $ be a Hamiltonian of class $C^2 (  H^1_0(\T_x), \R )$ and consider a map
$ \Phi(u) := u + \Psi(u) $ satisfying $\Psi (u) = \Pi_E \Psi(\Pi_E u)$, for all $ u $, 
where $E$ is a finite dimensional subspace as in \eqref{def E finito}. Then 
\begin{equation}\label{lint2}
\partial_u \big[\nabla ( H \circ \Phi)\big] (u)  [h] = (\partial_u  \nabla H )(\Phi(u)) [h] + {\cal R}(u)[h]\,,
\end{equation}
where  $ {\cal R}(u) $ 
has the ``finite dimensional" form 
\begin{equation}\label{forma buona resto}
{\cal R}(u)[h] =  {\mathop\sum}_{|j| \leq C} \big( h , g_j(u) \big)_{L^2(\T)} \chi_j(u) 
\end{equation}
with $ \chi_j (u) = e^{\ii j x} $ or  $ g_j(u) = e^{\ii j x} $. The remainder  
$ {\cal R} (u)  =  {\cal R}_0 (u)  + {\cal R}_1 (u) + {\cal R}_2 (u)  $ with
\begin{align}\label{resti012} 
{\cal R}_0 (u) & :=  (\partial_u \nabla H)(\Phi(u)) \pa_u \Psi (u), \qquad  
{\cal R}_1 (u) :=  [\partial_{u }\{ \Psi'(u)^T\}] [ \cdot , \nabla H(\Phi(u)) ], \nonumber \\
\,  {\cal R}_2 (u)  & :=  [\pa_u \Psi (u)]^T (\partial_u \nabla H)(\Phi(u)) \pa_u \Phi(u). 
\end{align}
\end{lemma}

\begin{proof}
By a direct calculation, 
\begin{equation}\label{nabla composto} 
\nabla (H \circ \Phi)(u) = [\Phi'(u)]^T \nabla H(\Phi(u)) = \nabla H(\Phi(u)) + [\Psi'(u)]^T \nabla H(\Phi(u))
\end{equation}
where $ \Phi' (u) := ( \pa_u \Phi) (u) $ and $ [ \ ]^T $ denotes the transpose with respect to the $ L^2 $ scalar product.
Differentiating \eqref{nabla composto}, we get \eqref{lint2} and \eqref{resti012}.

Let us show that each $ {\cal R}_m $ has the form \eqref{forma buona resto}. 
We have 
\begin{equation}\label{Psiuh}
\Psi'(u) = \Pi_E \Psi'(\Pi_E u) \Pi_E  \, \, , \quad  [\Psi'(u)]^T = \Pi_E   [\Psi'( \Pi_E u)]^T \Pi_E \, . 
\end{equation}
Hence, setting $ A  := (\partial_u \nabla H)(\Phi(u)) \Pi_E \Psi' ( \Pi_E u)  $, we get 
\[
{\cal R}_0(u)[h] 
= A [ \Pi_E h ] 
=  {\mathop\sum}_{|j| \leq C } h_j A ( e^{\ii j x} ) 
=  {\mathop\sum}_{|j| \leq C} (h, g_j )_{L^2(\T)} \chi_j 
\]
with $ g_j := e^{\ii j x} $, $ \chi_j :=  A( e^{\ii jx } )$.
Similarly, using \eqref{Psiuh},  and setting  
$ A  := [ \Psi' (\Pi_E u) ]^T \Pi_E (\partial_u \nabla H)(\Phi(u)) \Phi'(u) $,  
we get
$$
{\cal R}_2 (u)[h] =  \Pi_E [ A h ] =  {\mathop\sum}_{|j | \leq C } (A h ,  e^{\ii j x} )_{L^2(\T)} e^{\ii jx}  
 =  {\mathop\sum}_{|j| \leq C} (h, A^T   e^{\ii j x}  )_{L^2(\T)}  e^{\ii j x} \,,  
$$
which has the form \eqref{forma buona resto} with $ g_j :=  A^T( e^{\ii jx } )$ and  $ \chi_j := e^{\ii j x} $.
Differentiating the second equality in \eqref{Psiuh}, we see that  
$$
{\cal R}_1 (u)[h] =  \Pi_E [ A h ] \, , \quad A h := \partial_u \{ \Psi' ( \Pi_E u)^T \} [\Pi_E h, \Pi_E  
(\nabla H)(\Phi(u)) ] \,,
$$
which has the same form of $ {\cal R}_2 $ and so \eqref{forma buona resto}. 
\end{proof}

\begin{lemma} \label{sifulo}
Let  $ H(u ) := \int_\T f(u) X(u) d x $
where $ X(u) = \Pi_{E} X ( \Pi_E u) $ and $  f(u)(x) := f( u(x)) $ is the composition operator for a function of class
$  C^2 $. Then
\begin{equation}\label{lint1}
(\partial_u \nabla H) (u) [h] = f''(u)  X(u) \, h + {\cal R} (u) [h] 
\end{equation}
where $ {\cal R} (u) $ has  the form \eqref{forma buona resto}
with $ \chi_j (u) = e^{\ii j x} $ or  $ g_j(u) = e^{\ii j x} $.
\end{lemma}

\begin{proof}
A direct calculation proves  that $\nabla H(u) = f'(u) X(u) + X'(u)^T [f(u)]$, 
and \eqref{lint1} follows with
$ {\cal R} (u) [h] = $ $  f'(u) X'(u)[h] + $ $  \partial_u \{  X'(u)^T\} [h, f(u)] + $ $  X'(u)^T [ f'(u) h ]$, 
which has the form \eqref{forma buona resto}. 
\end{proof}

We conclude this section with a technical lemma  used from the end of section \ref{step3} about the decay norms of ``finite dimensional operators".
Note that operators of the form \eqref{forma buona con gli integrali} (that will appear in section \ref{step1}) 
reduce to those in \eqref{forma buona resto} when the functions $ g_j(\tau) $, $ \chi_j (\tau)$ are independent of $ \tau $

\begin{lemma}\label{remark : decay forma buona resto}
Let $ {\cal R} $ be an operator of the form
\begin{equation}\label{forma buona con gli integrali}
{\cal R} h = \sum_{|j| \leq C } \int_0^1 \big(h\,,\,g_j(\tau) \big)_{L^2(\T)} \chi_j (\tau)\,d \tau\,,
\end{equation}
where the functions $g_j(\tau),\,\chi_j(\tau) \in H^s$, $\tau \in [0, 1]$ depend in a Lipschitz way on the parameter $\omega$. Then its matrix $ s$-decay norm (see \eqref{matrix decay norm}-\eqref{matrix decay norm Lip}) satisfies 
$$ 
| {\cal R} |_s^\Lipg \leq_s {\mathop \sum}_{|j| \leq C} {\rm sup}_{\tau \in [0,1]} \big\{ \| \chi_j(\tau) \|_s^\Lipg \| g_j(\tau) \|_{s_0}^\Lipg 
+ \| \chi_j(\tau) \|_{s_0}^\Lipg \| g_j(\tau) \|_s^\Lipg \big\} \, . 
$$ 
\end{lemma}

\begin{proof}
For each $\tau \in [0, 1]$, the operator $ h \mapsto (h,g_j(\tau)) \chi_j(\tau) $ is the 
composition $ \chi_j(\tau) \circ \Pi_0 \circ g_j(\tau) $ of the multiplication operators for $ g_j(\tau), \chi_j(\tau) $ and 
$ h \mapsto \Pi_0 h := \int_{\T} h dx $. Hence
the lemma follows by the interpolation estimate \eqref{interpm Lip} 
and \eqref{multiplication Lip}. 
\end{proof}

\subsection{Composition with the map $G_\d$} \label{section:appr}

In the sequel we shall use that 
$   \fracchi_\d := \fracchi_\d (\ph ; \om) := i_\d (\ph; \om ) - (\ph,0,0) $ satisfies, 
by Lemma \ref{toro isotropico modificato} and \eqref{ansatz 0},  
\begin{equation}\label{ansatz delta}
\| {\mathfrak I}_\d \|_{s_0+\mu}^{\Lipg} 
\leq C\e^{6 - 2b} \gamma^{-1}\, .   
\end{equation}
We now study the Hamiltonian 
$ K := H_\e \circ G_\d = \e^{-2b} \mH \circ A_\e \circ G_\d $ defined in \eqref{KHG}, \eqref{def H eps}. 

Recalling  \eqref{def A eps} and \eqref{trasformazione modificata simplettica} the map $A_\e \circ G_\d$ has the form  
\begin{equation}  \label{A eps G delta}
A_\e \circ G_\d(\psi, \eta, w) 
= \e \sum_{j \in S} \sqrt{\xi_j + \e^{2(b-1)} |j| [ y_\d(\psi) + L_1(\psi) \eta + L_2(\psi) w ]_j } \, e^{\ii [\theta_0(\psi)]_j} e^{\ii jx} + \e^b (z_0(\psi) + w)
\end{equation}
where 
\be\label{L1 L2}
L_1(\psi) := [\pa_\psi \theta_0(\psi)]^{-T} \, , \quad
L_2(\psi) := \big[ (\pa_\teta \tilde{z}_0) (\theta_0(\psi)) \big]^T \partial_x^{-1} \, . 
\ee
 By Taylor's formula, we develop \eqref{A eps G delta} in $w$ at $\eta=0$, $w=0$, and we get
$ A_\e \circ G_\d(\psi, 0, w)  = $ $ T_\delta(\psi) + T_1(\psi) w + T_2(\psi)[w,w] + $ $  T_{\geq 3}(\psi, w) $, 
where 
\begin{equation}\label{T0}
T_\delta(\psi) := (A_\e \circ G_\d)(\psi, 0, 0)
= \e v_\d(\psi) + \e^b z_0(\psi) \, , \ \  v_\d (\psi):= \sum_{j \in S} \sqrt{\xi_j + \e^{2(b-1)} |j| [ y_\d(\psi) ]_j } \, e^{\ii [\theta_0(\psi)]_j} e^{\ii jx}
\end{equation}
is the approximate isotropic torus in  phase space (it corresponds to $ i_\d $ in Lemma \ref{toro isotropico modificato}), 
\begin{align}
T_1(\psi) w & = \e \sum_{j \in S} 
\frac{\e^{2(b-1)} |j| [ L_2(\psi) w ]_j \, e^{\ii [\theta_0(\psi)]_j}} 
{2 \sqrt{ \xi_j + \e^{2(b-1)} |j| [ y_\d(\psi) ]_j }} \,  e^{\ii jx} 
+ \e^b w =: \e^{2b-1} U_1 (\psi) w + \e^b w \, \label{T1} \\
T_2(\psi)[w,w] &  
= - \e \sum_{j \in S} 
\frac{\e^{4(b-1)} j^2 [ L_2(\psi) w ]_j^2 \, e^{\ii [\theta_0(\psi)]_j}} 
{8 \{ \xi_j + \e^{2(b-1)} |j| [ y_\d(\psi) ]_j \}^{3/2} } \,  
e^{\ii jx} =:  \e^{4b - 3} U_2(\psi)[w,w] \label{T2}
\end{align}
and $T_{\geq 3}(\psi, w)$ collects all the terms of order at least cubic in $w$. In the notation of  \eqref{def A eps},
the function $v_\d(\psi) $ in \eqref{T0} is 
$v_\d(\psi) = v_\e( \theta_0(\psi), y_\d(\psi))$.
The terms  $U_1, U_2 = O(1)$ in $\e$.
Moreover, using that $ L_2 (\psi) $ in \eqref{L1 L2} vanishes as $ z_0 = 0 $,  
they satisfy
\begin{equation}\label{extra piccolezza}
\| U_1 w \|_s \leq \| \fracchi_\d \|_s \| w \|_{s_0}   +  \| \fracchi_\d \|_{s_0}  \| w \|_s  \, , \quad
\| U_2 [w,w] \|_s \leq  \| \fracchi_\d \|_s \| \fracchi_\d \|_{s_0} \| w \|_{s_0}^2   +  
\| \fracchi_\d \|_{s_0}^2 \| w \|_{s_0} \| w \|_s  
\end{equation}
and also in the  $ \|  \ \|_s^\Lipg $-norm.

By Taylor's formula 
$ \mH(u+h) 
= \mH(u) + ( (\gr \mH)(u), h )_{L^2(\T)} 
+ \frac12 ( (\pa_u \gr \mH)(u) [h], h )_{L^2(\T)} 
+ O(h^3) $. 
Specifying at $u = T_\delta(\psi)$ and $ h = T_1(\psi) w + T_2(\psi)[w,w] + T_{\geq 3}(\psi,w)$, 
we obtain that the sum of all the components of $ K = \e^{-2b} (\mH \circ A_\e \circ G_\d)(\psi, 0, w) $ 
that are quadratic in $w$ is  
$$
\frac12 ( K_{02}w, w )_{L^2(\T)} 
= \e^{-2b} ( (\gr \mH)(T_\delta ), T_2 [w,w] )_{L^2(\T)} 
+ \e^{-2b} \frac12 ( (\pa_u \gr \mH)(T_\delta ) [T_1 w], 
T_1 w )_{L^2(\T)} \, . 
$$
Inserting the expressions \eqref{T1}, \eqref{T2} we get
\begin{align}
K_{02}(\psi) w 
& = (\pa_u \gr \mH)(T_\delta) [w]  
+ 2 \e^{b-1} (\pa_u \gr \mH)(T_\delta) [U_1 w] 
+ \e^{2(b-1)} U_1^T (\pa_u \gr \mH)(T_\delta) [U_1 w]  \nonumber
\\ & \quad 
+ 2 \e^{2b- 3} U_2[w, \cdot]^T (\gr \mH)(T_\delta). \label{K02}
\end{align}

\begin{lemma}\label{dopo l'approximate inverse}
\begin{equation}\label{piccolezza resti}
 ( K_{02}(\psi) w, w )_{L^2(\T)} 
=  ( (\pa_u \gr \mH)(T_\delta) [w], w )_{L^2(\T)} 
+  ( R(\psi) w, w )_{L^2(\T)}
\end{equation}
where $R(\psi)w $ has the ``finite dimensional" form 
\begin{equation}\label{forma buona resto con psi}
R(\psi) w  =  {\mathop\sum}_{|j| \leq C} \big( w , g_j(\psi) \big)_{L^2(\T)} \chi_j(\psi) 
\end{equation}
where, for some $\sigma := \sigma (\nu, \tau) > 0$, 
\begin{align}
\label{piccolo FBR}
 \| g_j \|_s^\Lipg \| \chi_j \|_{s_0}^\Lipg + \| g_j \|_{s_0}^\Lipg \| \chi_j \|_s^\Lipg 
& \leq_s \e^{b+1} \| {\mathfrak I}_\delta \|_{s + \sigma}^\Lipg \\
\| \partial_i g_j [\widehat \imath ]\|_s \| \chi_j \|_{s_0} 
+ \| \partial_i g_j [\widehat \imath ]\|_{s_0}  \| \chi_j \|_{s} + \| g_j \|_{s_0} \| \partial_i \chi_j [\widehat \imath ] \|_s 
+ \| g_j \|_{s} \| \partial_i \chi_j [\widehat \imath ]\|_{s_0}  
& \leq_s \e^{b + 1} \| \widehat \imath \|_{s + \sigma}\label{derivata piccolo FBR} \\ 
& + \e^{2b-1} \| {\mathfrak I}_\delta\|_{s + \sigma}  \|\widehat \imath \|_{s_0 + \sigma} \,,  \nonumber
\end{align}
and, as usual, $i = (\theta, y, z)$ (see \eqref{embedded torus i}), 
$\widehat \imath = (\widehat \theta, \widehat y, \widehat z)$.
\end{lemma}

\begin{proof}
Since $ U_1 = \Pi_S U_1 $ and  $ U_2 = \Pi_S U_2 $, 
the last three terms in \eqref{K02} have all the form \eqref{forma buona resto con psi}
(argue as in Lemma \ref{lemma astratto potente}). 
We now prove that they are also small in size.  

The contributions in \eqref{K02}  from $ H_2 $ are better analyzed by the expression
$$
\e^{-2b} H_2 \circ A_\e \circ G_\delta (\psi, \eta, w) =   const + \sum_{j \in S^+} j^3 
\big[  y_\delta (\psi) + L_1(\psi)\eta + L_2 (\psi) w  \big]_j  + \frac{1}{2} \int_{\T} ( z_0 (\psi)  + w)_x^2 \, dx 
$$
which follows by \eqref{shape H2}, \eqref{trasformazione modificata simplettica}, \eqref{L1 L2}.
 Hence the only contribution to 
$ (K_{02} w,w)  $  is $  \int_{\T} w_x^2 \, dx  $. 
Now we consider the cubic term $ \mH_3 $ in \eqref{H3tilde}.
A direct calculation shows that for $ u = v + z $, 
$ \nabla \mH_3 ( u )  = 3 z^2 + 6  \Pi_S^\bot (v z) $, and
$ \partial_u \nabla \mH_3 ( u ) [U_1  w] = 6 \Pi_S^\bot ( z U_1 w) $ (since $ U_1 w \in H_S $).
Therefore
\begin{equation}\label{mH 3 T delta}
\nabla \mH_3(T_\delta) = 3 \e^{2 b} z_0^2 + 6 \e^{b + 1} \Pi_S^\bot(v_\delta z_0 ) 
\,,\quad \partial_u \nabla \mH_3 ( T_\delta ) [U_1 w] = 6 \e^b \Pi_S^\bot ( z_0 \, U_1 w) \, . 
\end{equation}
By \eqref{mH 3 T delta} one has 
$  ( (\pa_u \gr \mH_3)(T_\delta) [U_1 w ], U_1 w )_{L^2(\T)}  = 0 $, 
and since also $U_2 = \Pi_S U_2$,  
\be\label{contributi H3}
\e^{b - 1}\partial_u \nabla {\cal H}_3(T_\delta)[U_1 w] + \e^{2 b - 3} U_2[w, \cdot]^T \nabla {\cal H}_3(T_\delta) = 6 \e^{2 b - 1} \Pi_S^\bot(z_0  U_1 w) + 3 \e^{4 b - 3} U_2[w, \cdot]^T  z_0^2 \,.
\ee
These terms have  the form  \eqref{forma buona resto con psi} and, using  \eqref{extra piccolezza}, \eqref{ansatz 0}, 
they satisfy \eqref{piccolo FBR}. 

Finally we consider all the terms which arise from ${\cal H}_{\geq 4} = O(u^4)$.  
The operators
$ \e^{b - 1} \partial_u \nabla {\cal H}_{\geq 4}(T_\delta) U_1  $, 
$ \e^{2(b - 1)} U_1^T (\pa_u \nabla {\cal H}_{\geq 4})(T_\delta) U_1$, 
$ \e^{2 b - 3} U_2^T \nabla {\cal H}_{\geq 4}(T_\delta) $
have the form  \eqref{forma buona resto con psi} and, using $ \| T_\delta \|_s^\Lipg \leq \e (1 + \| \fracchi_\d \|_s^\Lipg) $,
\eqref{extra piccolezza}, \eqref{ansatz 0}, 
the bound  \eqref{piccolo FBR} holds. Notice that the biggest term is  $ \e^{b - 1} \partial_u \nabla {\cal H}_{\geq 4}(T_\delta) U_1  $. 

By \eqref{derivata i delta} and using explicit formulae \eqref{L1 L2}-\eqref{T2} we get estimate \eqref{derivata piccolo FBR}.
\end{proof}

The conclusion of this section is that,  after the composition with 
the action-angle variables, the rescaling \eqref{rescaling kdv quadratica}, 
and the transformation $ G_\delta $, the linearized operator 
to analyze is $  H_S^\bot \ni w \mapsto (\pa_u \gr \mH)(T_\delta) [w] $,  up to finite dimensional operators which have the form \eqref{forma buona resto con psi} and size \eqref{piccolo FBR}.

\subsection{The linearized operator in the normal directions}

In view of \eqref{piccolezza resti} we now compute  
$  ( (\pa_u \gr \mH)(T_\delta) [w], w )_{L^2(\T)} $, $ w \in H_S^\bot $, where
$ \mH  = H \circ \Phi_B $ and $\Phi_B $ is the Birkhoff map of Proposition \ref{prop:weak BNF}. 
It is convenient to estimate separately the terms in 
\be\label{mH H2H3H5}
\mH  = H \circ \Phi_B = (H_2 + H_3) \circ \Phi_B + H_{\geq 5} \circ \Phi_B 
\ee
where $ H_2, H_3, H_{\geq 5}$ are defined in \eqref{H iniziale KdV}.

We first consider $ H_{\geq 5} \circ \Phi_B $. 
By \eqref{H iniziale KdV} we get
$ \nabla H_{\geq 5}(u) = \pi_0[ (\partial_u f)(x, u, u_x) ] - \pa_x \{ (\pa_{u_x} f)(x, u,u_x) \} $, see \eqref{def pi 0}. 
Since the Birkhoff transformation $ \Phi_B $ has the form \eqref{finito finito}, 
Lemma \ref{lemma astratto potente} (at $ u = T_\delta $, see \eqref{T0}) implies that
\begin{align} 
\pa_u \nabla ( H_{\geq 5} \circ \Phi_B ) (T_\delta)  [h] 
& = 
(\pa_u \gr H_{\geq 5})(\Phi_B(T_\delta)) [h] + {\cal R}_{H_{\geq 5}}(T_\delta)[h] 
\notag \\ 
& = \partial_x (r_1(T_\delta) \partial_x h ) + r_0(T_\delta) h + {\cal R}_{H_{\geq 5}}(T_\delta)[h] 
\label{der grad struttura separata5}
\end{align}
where the multiplicative functions $r_0(T_\d)$, $r_1(T_\d)$ are
\begin{alignat}{2} \label{r0r1 def}
r_0 (T_\delta) & := \s_0(\Phi_B(T_\delta)), \qquad & 
\s_0(u) & :=  (\partial_{uu} f)(x, u, u_x)  -  \pa_x \{ (\partial_{u u_x} f)(x, u, u_x) \}, 
\\
\label{sigma0sigma1 def}
r_1 (T_\delta) & := \s_1(\Phi_B(T_\delta)), \quad &
\s_1(u) & := -  (\partial_{u_x u_x} f)(x, u, u_x) , 
\end{alignat}
the remainder $ {\cal R}_{H_{\geq 5}}(u) $ has the form \eqref{forma buona resto}
with $\chi_j = e^{\ii jx}$ or $g_j = e^{\ii jx}$ and, using \eqref{resti012}, 
it satisfies,  
for some $ \sigma := \sigma (\nu, \tau) > 0$, 
\begin{align*}
 \| g_j \|_s^\Lipg \| \chi_j \|_{s_0}^\Lipg + \| g_j \|_{s_0}^\Lipg \| \chi_j \|_s^\Lipg 
& \leq_s \e^4 (1 + \| \fracchi_\d \|_{s+2}^\Lipg )  \\
\| \partial_i g_j [\widehat \imath ]\|_s \| \chi_j \|_{s_0} 
+ \| \partial_i g_j [\widehat \imath ]\|_{s_0}  \| \chi_j \|_{s} 
+ \| g_j \|_{s_0} \| \partial_i \chi_j [\widehat \imath ] \|_s 
+ \| g_j \|_{s} \| \partial_i \chi_j [\widehat \imath ]\|_{s_0}  
& \leq_s  \e^4 ( \| \widehat \imath \|_{s+\s} 
+ \| \fracchi_\d \|_{s+2} \| \widehat \imath \|_{s_0 + 2} ).  
\end{align*}
Now we consider the contributions from $ (H_2 + H_3) \circ \Phi_B $.
By Lemma \ref{lemma astratto potente} and the expressions of $ H_2, H_3 $ in \eqref{H iniziale KdV} we deduce that
$$
\pa_u \nabla ( H_2 \circ \Phi_B) (T_\delta) [h] 
= - \partial_{xx} h + {\cal R}_{H_2}(T_\delta)[h] \,, \quad  
\pa_u \nabla ( H_3 \circ \Phi_B) (T_\delta) [h] 
= 6 \Phi_B (T_\delta)   h + {\cal R}_{H_3}(T_\delta)[h] \,,
$$
where $  \Phi_B (T_\delta) $ is a function with zero space average, 
because $ \Phi_B: H^1_0 (\T_x) \to H^1_0 (\T_x)$ (Proposition \ref{prop:weak BNF}) 
and  $ {\cal R}_{H_2}(u) $, $ {\cal R}_{H_3}(u) $ have the form \eqref{forma buona resto}. 
By \eqref{resti012}, the size $ ( {\cal R}_{H_2} + {\cal R}_{H_3}) (T_\delta )  = O( \e ) $. 
We expand
$$
( {\cal R}_{H_2} + {\cal R}_{H_3}) (T_\delta ) = \e {\cal R}_1 + \e^2 {\cal R}_2 +  {\tilde {\cal R}}_{> 2} \,, 
$$  
where $\tilde \mR_{>2}$ has size $o(\e^2)$, 
and we get, $ \forall h \in H_S^\bot $, 
\begin{equation}\label{useful repr}
\Pi_S^\bot  \pa_u \nabla ((H_2 + H_3) \circ \Phi_B) (T_\delta)  [h] 
= - \partial_{xx} h +  \Pi_S^\bot (6 \Phi_B (T_\delta) h ) + 
\Pi_S^\bot  ( \e {\cal R}_1 + \e^2 {\cal R}_2 + {\tilde {\cal R}}_{> 2} ) [h] \, . 
\end{equation}
We also develop the function $ \Phi_B (T_\delta) $ is powers of $ \e $. 
Expand $\Phi_B (u) = u + \Psi_2 (u) + \Psi_{\geq 3}(u) $, 
where $ \Psi_2 (u) $ is quadratic, $ \Psi_{\geq 3} (u) = O(u^3)$, and both map  
$ H_0^1(\T_x)  \to H_0^1(\T_x) $.
At $ u = T_\d = \e v_\d + \e^b z_0 $ 
we get
\begin{align}
\Phi_B ( T_\delta ) & = T_\delta + \Psi_2 (T_\delta) + \Psi_{\geq 3}(T_\delta) 
= \e v_\d + \e^2 \Psi_2 ( v_\d ) + \tilde q 
\label{funzione moltiplicativa}
\end{align}
where $ \tilde q := \e^b z_0 + \Psi_2 ( T_\d ) - \e^2 \Psi_2 (v_\d) 
+ \Psi_{\geq 3} (T_\d) $ 
has zero space average and it satisfies  
$$
\| \tilde q \|_s^\Lipg 
\leq_s \e^3 + \e^b \| \fracchi_\delta \|_s^\Lipg\,,\quad 
\| \partial_i \tilde q [\widehat \imath ] \|_s 
\leq_s \e^b \big( \| \widehat \imath \|_s + \| {\mathfrak I}_\delta \|_s \| \widehat \imath \|_{s_0} \big)\,.
$$
In particular, its low norm $\| \tilde q \|_{s_0}^\Lipg \leq_{s_0} \e^{6-b} \g^{-1} = o(\e^2)$.  

We need an exact expression of the terms of order $ \e $ and $ \e^2 $ in  \eqref{useful repr}. 
We compare the Hamiltonian \eqref{widetilde cal H} with \eqref{mH H2H3H5},  
noting that $ (H_{\geq 5} \circ \Phi_B)(u) = O(u^5) $ 
because $ f $ satisfies \eqref{order5} and $ \Phi_B (u) = O(u) $. Therefore 
$$
(H_2 + H_3) \circ \Phi_B =  H_2 + \mH_3 + \mH_4  + O(u^5) \,,
$$
and the homogeneous terms of $ (H_2 + H_3) \circ \Phi_B $ of degree $ 2, 3, 4 $ in $ u $ are $ H_2 $, $\mH_3 $, $ \mH_4 $ respectively. 
As a consequence, the terms of order $ \e $ and $ \e^2 $ in \eqref{useful repr} (both in the function 
$ \Phi_B (T_\d ) $ and in the remainders $ \mR_1, \mR_2 $) come only from $ H_2 + \mH_3 + \mH_4 $. 
Actually they come from $ H_2 $, $ \mH_3 $ and $ \mH_{4,2} $  (see \eqref{H3tilde}, \eqref{mH3 mH4})
because, at $ u = T_\d = \e v_\d + \e^b z_0 $, 
for all $ h \in H_S^\bot $,  
$$
\Pi_S^\bot (\pa_u \gr \mH_4)(T_\delta) [h] = 
\Pi_S^\bot (\pa_u \gr \mH_{4,2})(T_\delta) [h] + o(\e^2) \,. 
$$
A direct calculation based on the expressions \eqref{H3tilde}, \eqref{mH3 mH4}
shows that, for all $ h \in H_S^\bot $, 
\begin{align}
\Pi_S^\bot (\pa_u \nabla ( H_2 + \mH_3 + \mH_4)) (T_\delta)[h] 
& =  - \partial_{xx} h +   6 \e \Pi_S^\bot ( v_\d  h )  + 6 \e^b \Pi_S^\bot ( z_0  h ) 
+ \e^2 \Pi_S^\bot \big\{ 6 \pi_0 [ (\partial_x^{-1} v_\d)^2  ]   h  \nonumber \\
& 
+ 6 v_\d  \Pi_S [ (\partial_x^{-1} v_\d) (\partial_x^{-1} h) ] -
 6  \partial_x^{-1}  \{  (\partial_x^{-1} v_\d) \Pi_S [v_\d  h] \} \big\}  + 
o(\e^2). \label{lin basso}
\end{align}
Thus, comparing the terms of order $ \e, \e^2  $ 
in \eqref{useful repr} (using \eqref{funzione moltiplicativa}) with those in \eqref{lin basso} we deduce that the operators $\mR_1, \mR_2$ and the function $\Psi_2 (v_\d)$ are
\begin{equation}\label{R nullo1R2}
{\cal R}_1 = 0 ,  \quad 
{\cal R}_2 [h ] = 6 v_\d  \Pi_S \big[ (\partial_x^{-1} v_\d) (\partial_x^{-1} h)  \big] 
- 6 \partial_x^{-1} \{ (\partial_x^{-1} v_\d) \Pi_S [v_\d  h] \} \,, \quad 
\Psi_2 (v_\d) =  \pi_0 [ (\partial_x^{-1} v_\d)^2  ] . 
\end{equation}
In conclusion,  by \eqref{mH H2H3H5}, 
\eqref{useful repr}, \eqref{der grad struttura separata5},
\eqref{funzione moltiplicativa}, \eqref{R nullo1R2}, we get, 
for all $ h \in H_{S^\bot } $, 
\begin{align}
\Pi_S^\bot \pa_u \nabla \mH (T_\delta)[h]  
& =  - \partial_{xx} h +  \Pi_S^\bot 
\big[ \big( \e 6 v_\d + 
\e^2 6 \pi_0 [ (\partial_x^{-1} v_\d)^2  ] +   q_{>2} + p_{\geq 4 } 
\big) h \big] \nonumber \\
& \quad +  
 \Pi_S^\bot  \partial_x (r_1(T_\delta) \partial_x h ) + 
 \e^2   \Pi_S^\bot {\cal R}_2[h] + \Pi_S^\bot  {\cal R}_{> 2} [h] \label{lafinalina}
\end{align}
where  $ r_1 $ is defined in \eqref{r0r1 def}, $ {\cal R}_2 $ in \eqref{R nullo1R2}, the remainder 
$  {\cal R}_{> 2} :=     {\tilde {\cal R}}_{> 2} + {\cal R}_{H_{\geq 5}} (T_\d)  $ 
and the functions
(using also \eqref{r0r1 def}, \eqref{sigma0sigma1 def},  \eqref{order5}), 
\begin{align}\label{def p3}
q_{>2} & :=   6 \tilde q 
 + \e^3 \big( (\partial_{uu} f_5)(v_\d, (v_\d)_x) -  \pa_x \{ (\partial_{u u_x} f_5)(v_\d, (v_\d)_x) \}  \big)  \\
p_{\geq 4 } & :=    
r_0 (T_\d) - \e^3 \big[ (\partial_{uu} f_5)(v_\d, (v_\d)_x) -  \pa_x \{ (\partial_{u u_x} f_5)(v_\d, (v_\d)_x) \} \big] \, . 
 \label{def p>4} 
\end{align}

\begin{lemma}\label{p3 zero average}
$ \int_{\T} q_{>2} dx = 0 $. 
\end{lemma}

\begin{proof}
We already 
observed that $ \tilde q $ has zero $x$-average as well as
the derivative $ \pa_x \{ (\partial_{u u_x} f_5)(v, v_x) \}  $.
Finally 
\be\label{unico pezzo}
(\partial_{uu} f_5)(v, v_x) = \sum_{j_1, j_2, j_3 \in S } c_{j_1j_2j_3} v_{j_1} v_{j_2} v_{j_3} e^{\ii (j_1+ j_2 + j_3) x } \, , 
\quad v := \sum_{j \in S} v_j e^{\ii j x}
\ee
for some coefficient  $ c_{j_1j_2j_3} $, and therefore it
has zero average by hypothesis (${\mathtt S}1$). 
\end{proof}

By Lemma \ref{dopo l'approximate inverse}
and the results of this section (in particular  \eqref{lafinalina}) we deduce: 

\begin{proposition}\label{prop:lin}
Assume \eqref{ansatz delta}. Then  the Hamiltonian operator $ {\cal L}_\om $ has the form, 
$  \forall h \in H_{S^\bot}^s ( \T^{\nu+1}) $,
\be\label{Lom KdVnew}
{\cal L}_\om h  :=   \om \!\cdot \!\partial_{\vphi} h - \partial_x K_{02} h   
=  \Pi_S^\bot \big( \om \! \cdot \! \partial_\vphi h + \partial_{xx} (a_1 \partial_x h)  +
\partial_x ( a_0 h ) 
- \e^2 \partial_x {\cal R}_2 h - \partial_x \mR_* h    \big)  
\ee
where $ {\cal R}_2 $ is defined in \eqref{R nullo1R2}, 
$ {\mR}_* := {\cal R}_{> 2} + R(\psi) $ (with $R(\psi)$  defined in Lemma \ref{dopo l'approximate inverse}), the functions 
\begin{equation}\label{a1p1p2}
a_1  := 1 - r_1 ( T_\delta ) \, , \quad 
a_0 :=  - ( \e p_1  + \e^2 p_2 + q_{>2} +  p_{\geq 4} ) \, , \quad 
p_1 :=  6 v_\d  \, , \quad p_2 :=   6 \pi_0 [ (\partial_x^{-1} v_\d)^2  ]\,, 
\end{equation}
the function $ q_{>2} $ is defined in \eqref{def p3} and satisfies $ \int_{\T} q_{>2} dx = 0 $, the function $ p_{ \geq 4} $ is defined in \eqref{def p>4},  
$ r_1 $  in \eqref{sigma0sigma1 def},  
$ T_\delta $ and $ v_\d $ in \eqref{T0}.
For $ p_k = p_1, p_2 $, 
\begin{alignat}{2} \label{stime pk}
\| p_k \|_s^\Lipg  
& \leq_s 1 + \| {\mathfrak I}_\d \|_s^\Lipg, 
& \quad \qquad 
\| \pa_i p_k [ \widehat \imath ] \|_s   
& \leq_s \| \widehat \imath \|_{s+1} +  \| {\mathfrak I}_\delta \|_{s+1} \| \widehat \imath \|_{s_0+1}, 
\\
\label{stima q>2}
\| q_{>2} \|_s^\Lipg 
& \leq_s \e^3 + \e^b \| \fracchi_\delta \|_{s}^\Lipg\,,
& \quad  
\| \pa_i q_{>2} [\widehat \imath ] \|_s 
& \leq_s \e^b \big( \| \widehat \imath \|_{s+1} + \| {\mathfrak I}_\delta \|_{s+1} \| \widehat \imath \|_{s_0+1} \big),
\\
\| a_1 -1 \|_s^\Lipg 
& \leq_s \e^3 \big(  1 + \| {\mathfrak I}_\d \|_{s+1}^\Lipg \big) \,,  
& \quad  
\| \pa_i a_1[\widehat \imath ] \|_s  
& \leq_s \e^3 \big( \| \widehat \imath \|_{s + 1} + \| {\mathfrak I}_\delta \|_{s + 1} \| \widehat \imath \|_{s_0 + 1} \big) \label{stima a1}  \\
\| p_{\geq 4} \|_s^\Lipg & \leq_s \e^4 + \e^{b + 2} \| \fracchi_\delta \|_{s + 2}^\Lipg\,,& \quad \| 
\pa_i p_{\geq 4}[\widehat \imath ] \|_s & \leq_s \e^{b + 2} \big( \| \widehat \imath \|_{s + 2} + \| \fracchi_\delta\|_{s + 2} \| \widehat \imath  \|_{s_0 + 2} \big) \label{p geq 4}
\end{alignat}
where  $ \fracchi_\d (\ph) := (\theta_0(\ph) - \ph, y_\d(\ph), z_0(\ph)) $ corresponds to $T_\delta$. 
The remainder $ {\cal R}_{2} $ has the form \eqref{forma buona resto} with 
\begin{align} \label{stime resto 1}
\| g_j \|_s^\Lipg + \| \chi_j \|_s^\Lipg 
\leq_s  1+ \| {\mathfrak I}_\delta \|_{s + \sigma}^\Lipg \, , \quad 
\| \partial_i g_j [\widehat \imath ]\|_s  + \| \partial_i \chi_j [\widehat \imath ]\|_s
  \leq_s   \| \widehat \imath \|_{s + \sigma} +   \| {\mathfrak I}_\delta\|_{s + \sigma}  \|\widehat \imath \|_{s_0 + \sigma} 
\end{align}
and also $ {\cal R}_* $ has the form \eqref{forma buona resto} with
\begin{align} \label{stima cal R*}
 \| g_j^* \|_s^\Lipg \| \chi_j^* \|_{s_0}^\Lipg + \| g_j^* \|_{s_0}^\Lipg \| \chi_j^* \|_s^\Lipg 
& \leq_s  \e^3 + \e^{b+1} \| {\mathfrak I}_\delta \|_{s + \sigma}^\Lipg
 \\
\| \partial_i g_j^* [\widehat \imath ]\|_s \| \chi_j^* \|_{s_0} 
+ \| \partial_i g_j^* [\widehat \imath ]\|_{s_0}  \| \chi_j^* \|_{s} + \| g_j^* \|_{s_0} \| \partial_i \chi_j^* [\widehat \imath ] \|_s 
+ \| g_j^* \|_{s} \| \partial_i \chi_j^* [\widehat \imath ]\|_{s_0}  
& \leq_s  \e^{b + 1} \| \widehat \imath \|_{s + \sigma}
\label{derivate stima cal R*} \\ 
& +  \e^{2b-1} \| {\mathfrak I}_\delta\|_{s + \sigma}  \|\widehat \imath \|_{s_0 + \sigma} \nonumber  \, . 
\end{align} 
\end{proposition}

The bounds \eqref{stime resto 1}, \eqref{stima cal R*} imply, by 
Lemma \ref{remark : decay forma buona resto}, estimates 
for the $ s $-decay norms of $ {\cal R}_2 $ and 
$  {\cal R}_* $.   The linearized operator $ {\cal L}_\om := {\cal L}_\om (\om, i_\d (\om))$ 
depends on the parameter $ \om $ both directly  and also through the dependence on the 
torus $   i_\d (\om )  $.
We have estimated also the partial derivative $ \pa_i $ with respect to  the variables $ i $ (see \eqref{embedded torus i})
in order to 
control,  along the nonlinear Nash-Moser iteration, 
the Lipschitz variation of the eigenvalues of $ {\cal L}_\om $ with respect to  $ \om $
and  the approximate solution $ i_\d $.

\section{Reduction of the linearized operator in the normal directions}\label{operatore linearizzato sui siti normali}

The goal of this section is to conjugate the Hamiltonian operator $ {\cal L}_\om $ in \eqref{Lom KdVnew} to the 
diagonal  operator  $ {\cal L}_\infty $ defined in \eqref{Lfinale}. 
The proof is obtained applying different kind of symplectic transformations. We shall always assume \eqref{ansatz delta}.

\subsection{Change of the space variable }\label{step1}

The first task is to conjugate $ {\cal L}_\om $ in \eqref{Lom KdVnew} to  $ {\cal L}_1 $ in
 \eqref{cal L1 Kdv},  which has 
the coefficient of $\partial_{xxx}$ independent on the space variable. 
We look for a  $ \vphi $-dependent family of {\it symplectic} diffeomorphisms $\Phi (\vphi) $ of $ H_S^\bot $ 
which differ from 
\begin{equation}\label{primo cambio di variabile modi normali}
{\cal A}_{\bot} := \Pi_S^\bot {\cal A} \Pi_S^\bot  \, ,  \quad  
({\cal A} h)(\vphi,x) := (1 + \beta_x(\vphi,x)) h(\vphi,x + \beta(\vphi,x)) \, , 
\end{equation}
up to a small ``finite dimensional" remainder, see  \eqref{forma buona resto cambio di variabile hamiltoniano}.
Each 
$ {\cal A}(\vphi) $ is a  symplectic  map 
of the phase space, see  \cite{BBM}-Remark 3.3. 
If  $ \| \b \|_{W^{1,\infty}} < 1 / 2 $ then 
 $ {\cal A} $ is invertible, see Lemma \ref{lemma:utile}, and 
its inverse and  adjoint maps are 
\begin{equation}\label{cambio di variabile inverso}
({\cal A}^{-1} h)(\vphi,y) := (1 + \tilde{\beta}_y(\vphi,y)) h(\vphi, y + \tilde{\beta}(\vphi,y)) \,, 
\quad 
({\cal A}^T h) (\vphi,y) = h(\vphi, y + \tilde{\beta}(\vphi,y)) 
\end{equation}
where
$ x = y + \tilde{\b} (\vphi, y) $ is the inverse diffeomorphism (of $\T$) of $ y = x + \b (\vphi, x) $. 

The restricted maps $ {\cal A}_\bot (\vphi): H_S^\bot \to H_S^\bot $ are not symplectic.  
In order to find a symplectic diffeomorphism near  $ {\cal A}_\bot (\vphi) $, 
the first observation is 
that each $ {\cal A }(\vphi ) $ can be seen as the time $1$-flow of a time dependent 
Hamiltonian PDE. Indeed $ {\cal A }(\vphi ) $ (for simplicity we skip the dependence on $ \vphi $) 
is homotopic to the identity 
via the path of symplectic diffeomorphisms 
$$ 
u \mapsto  
(1+ \tau \b_x ) u ( x+ \tau \b(x) ), \quad  \tau \in [0,1 ] \, , 
$$
which is the trajectory solution of the time dependent, linear Hamiltonian PDE
\be \label{transport-free}
\partial_\tau u = \partial_x (b(\tau, x) u) \, , \quad b (\tau, x) := \frac{\beta(x)}{1 + \tau \beta_x(x)}\, , 
\ee
with value $ u (x) $ at $ \tau = 0 $ and $ {\cal A}u  = (1+ \b_x(x) ) u ( x+  \b(x) ) $ at $ \t = 1 $. 
The equation \eqref{transport-free} is a {\it transport} equation. Its  
associated 
charactheristic ODE is 
 \begin{equation}\label{equazione delle caratteristiche}
 \frac{d}{d\tau} x  = - b(\tau, x ) \, .
 \end{equation}
We  denote  its flow by $ \gamma^{\tau_0, \t}  $, namely $  \gamma^{\tau_0, \t} (y) $ is the 
 solution of \eqref{equazione delle caratteristiche} with $  \gamma^{\tau_0, \tau_0} (y) = y $. 
Each  $ \gamma^{\tau_0, \t} $ is a diffeomorphism of the torus $ \T_x $. 

\begin{remark}\label{rem:ca}
Let $ y \mapsto y + \tilde \beta(\tau, y)$ be the inverse diffeomorpshim of 
$x \mapsto x + \tau \beta(x)$. 
Differentiating the identity $\tilde \beta(\tau, y) + \tau \beta(y + \tilde \beta(\tau, y)) = 0$ with respect to $\tau$ it results that 
 $ \g^\t (y) := \g^{0,\tau} (y) =  y + \tilde \beta(\tau, y) $. 
\end{remark}

Then we define a symplectic map $\Phi $ of $ H_S^\bot $  as the time-1 flow of the Hamiltonian PDE
 \begin{equation}\label{problemi di cauchy} 
 \partial_\tau u = \Pi_S^\bot \partial_x (b(\tau, x) u) = \partial_x (b(\tau, x) u) - \Pi_S \partial_x (b(\tau, x) u)  \, ,
 \quad u \in H_S^\bot  \, .
 \end{equation}
Note that $ \Pi_S^\bot \partial_x (b(\tau, x) u) $ is the Hamiltonian vector field generated by 
$ \frac12 \int_{\T} b(\tau, x) u^2 dx  $ restricted to $ H_S^\bot $. 
We denote by $ \Phi^{\tau_0,\tau} $ the flow of 
 \eqref{problemi di cauchy}, namely $ \Phi^{\tau_0, \tau} (u_0 ) $ is the solution of  \eqref{problemi di cauchy} 
 with initial condition $  \Phi^{\tau_0, \tau_0} (u_0 ) = u_0 $.
The flow is   well defined in Sobolev spaces 
$ H^s_{S^\bot} (\T_x) $  for $ b(\tau, x)  $ is smooth enough
(standard theory of linear hyperbolic PDEs, see e.g. 
section 0.8 in \cite{Taylor}). 
 It is natural to expect that the 
difference between the flow map $ \Phi := \Phi^{0,1} $ and $ {\cal A}_\bot $ is a  ``finite-dimensional" remainder
of the size of $  \b $. 

\begin{lemma}\label{modifica simplettica cambio di variabile}
For  $ \| \beta \|_{W^{s_0 + 1,\infty}} $ small,  there exists 
an invertible symplectic transformation $  \Phi = {\cal A}_\bot + {\cal R}_\Phi  $ of $ H_{S^\bot}^s $, 
 where $ {\cal A}_\bot $ is defined in \eqref{primo cambio di variabile modi normali} 
and $ {\cal R}_\Phi $ is a ``finite-dimensional" remainder 
\begin{equation}\label{forma buona resto cambio di variabile hamiltoniano}
{\cal R}_\Phi h= \sum_{j \in S} \int_0^1 (h, g_j (\tau)  )_{L^2(\T)} \chi_j (\tau) d \tau + \sum_{j \in S} \big(h, \psi_j \big)_{L^2(\T)} e^{\ii j x}
\end{equation}
for some functions $ \chi_j (\tau), g_j (\tau) , \psi_j \in H^s $ satisfying 
\begin{equation}\label{stime forma buona resto cambio di variabile hamiltoniano}
\| \psi_j\|_s\,,\, \| g_j(\tau)\|_s \leq_s \| \beta\|_{W^{s + 2, \infty}}\,,
\quad \| \chi_j(\tau)\|_s \leq_s  1 + \| \beta \|_{W^{s + 1, \infty}} \,,\quad \forall \tau \in [0, 1]\, .
\end{equation}
Furthermore,  
the following tame estimates holds
\begin{equation}\label{stime Phi Phi -1}
\| \Phi^{\pm 1}h\|_s \leq_s \| h \|_s + \| \beta \|_{W^{s + 2, \infty}} \| h \|_{s_0} \, , \quad \forall h \in H^s_{S^\bot} \,  . 
\end{equation}
\end{lemma}

\begin{proof}
Let $ w (\tau, x ) := (\Phi^\t u_0)(x) $ denote  the solution  of \eqref{problemi di cauchy} with initial condition 
$ \Phi^0 (w) = u_0 \in H_S^\bot $.  
The difference 
\be\label{differenza voluta}
({\cal A}_\bot - \Phi) u_0  = \Pi_S^\bot {\cal A} u_0  - w(1, \cdot) 
= {\cal A}u_0  - w(1, \cdot ) - \Pi_S {\cal A} u_0  \, , \quad \forall u_0 \in H_S^\bot \, , 
\ee
and 
\be\label{pezzo2}
\Pi_S {\cal A} u_0  = \Pi_S ({\cal A} - I) \Pi_S^\bot u_0  = \sum_{j \in S} \big(u_0\,,\,\psi_j \big)_{L^2(\T)} e^{\ii j x}\,,
\quad \psi_j := ({\cal A}^T- I) e^{\ii j x} \, .
\ee
We claim that the difference
\be\label{differenza}
{\cal A} u_0  - w(1, x) =  (1 +  \beta_x (x) )\int_0^1 (1 + \t \b_x (x) )^{-1}
\big[ \Pi_S \partial_x (b (\tau ) w(\tau) ) \big] ( \g^\t ( x + \b (x) )) \,d \tau 
\ee
where  $ \g^\t (y)  := \g^{0,\t} (y) $ is the flow of \eqref{equazione delle caratteristiche}.
Indeed the  solution $ w(\tau, x) $ of \eqref{problemi di cauchy} satisfies 
$$
\partial_\tau \{ w(\tau, \g^\t (y) ) \} = b_{x}(\tau, \g^\t (y)) w(\tau, \g^\t (y)) - \big[ \Pi_S \partial_x (b (\tau ) w(\tau) ) \big] ( \g^\t (y))\, .
$$
Then, by the variation of constant formula, we find
$$
w(\tau, \g^\t (y)) =  e^{\int_0^\tau b_x (s, \g^s (y))\,d s} \Big( u_0(y) - \int_0^\tau  e^{- \int_0^s b_x(\zeta, 
\g^\zeta (y) )\,d\zeta} \big[ \Pi_S \partial_x (b ( s ) w(s) ) \big] ( \g^s (y)) \, d s  \Big) \, . 
$$
Since $ \partial_y \g^\t (y) $ solves the variational equation 
$ \pa_\tau (\partial_y \g^\t (y)) = - b_x (\t, \g^\t (y) ) (\partial_y \g^\t (y))  $
with  $ \partial_y \g^0 (y)  = 1 $ we have that 
\begin{equation}\label{identita equazione variazionale}
e^{\int_0^\tau b_x(s , \g^s (y) )d s} = \big( {\partial_{y} \g^\t (y) } \big)^{-1} = 1 + \tau   \beta_x (x)
 \end{equation}
 by remark \ref{rem:ca}, and so we derive the expression 
$$
w(\tau, x) = (1 + \tau \beta_x (x)) \Big\{ u_0(x +  \tau \beta(x))  - \int_0^\tau 
( 1+ s \b_x (x) )^{-1} 
 \big[ \Pi_S \partial_x (b ( s ) w(s) ) \big] ( \g^s ( x + \t \b (x) )) \,d s \Big\} \, .
$$
Evaluating at $ \t = 1 $, formula  \eqref{differenza} follows. 
Next, we develop (recall $ w(\t) = \Phi^\tau (u_0 ) $) 
\begin{equation}\label{espressione esplicita gj}
 [\Pi_S \partial_x (b(\tau) w(\tau))] (x)  
 =  \sum_{j \in S} \big(u_0, g_j(\tau) \big)_{L^2(\T)} e^{\ii j x} \, , \quad
 g_j(\tau) := - (\Phi^\tau)^T[b(\tau) \partial_x e^{\ii j x}]\,, 
\end{equation}
and  \eqref{differenza} becomes  
\begin{equation}\label{formula utile 4}
{\cal A} u_0 - w(1, \cdot) = - \int_0^1 \sum_{j \in S} \big(u_0\,,\, g_j(\tau) \big)_{L^2(\T)} \chi_j(\tau, \cdot )\,d\tau\,,
\end{equation}
where 
\begin{equation}\label{espressione chi j}
\chi_j(\tau, x) :=  - ( 1 + \b_x (x) ) (1 + \t \b_x (x))^{-1} e^{\ii j \g^\t (x + \b(x) )} \, .  
\end{equation}

By \eqref{differenza voluta}, \eqref{pezzo2}, \eqref{differenza}, \eqref{formula utile 4} we deduce that 
$ \Phi = {\cal A}_\bot + {\cal R}_\Phi $ as in \eqref{forma buona resto cambio di variabile hamiltoniano}.

\smallskip

We now prove the estimates \eqref{stime forma buona resto cambio di variabile hamiltoniano}. 
Each function $ \psi_j  $ in \eqref{pezzo2} satisfies $  \|\psi_j \|_s \leq_s \| \beta \|_{W^{s, \infty}} $, 
see \eqref{cambio di variabile inverso}.
The bound $ \| \chi_j(\tau)\|_s \leq_s  1 + \| \beta \|_{W^{s + 1, \infty}} $ follows by \eqref{espressione chi j}.
The  tame estimates for $ g_j(\tau) $ defined in \eqref{espressione esplicita gj} are more difficult because 
require tame estimates  for the adjoint $(\Phi^\tau)^T$,  $ \forall \tau \in [0, 1] $.
The adjoint of the flow map can be represented 
as  the flow map of the ``adjoint'' PDE 
\begin{equation}\label{equazione aggiunta}
\partial_\tau z =  \Pi_S^\bot \{ b(\tau, x) \partial_x \Pi_S^\bot z \}  =  
b(\tau, x) \partial_x z - \Pi_S (b(\tau, x)  \partial_x  z ) \, , \quad  z \in H_S^\bot \, , 
\end{equation}
where  $ - \Pi_S^\bot b(\tau,x) \pa_x $ is the $ L^2 $-adjoint of the Hamiltonian 
vector field in \eqref{problemi di cauchy}.
We denote by $ \Psi^{\tau_0, \tau} $ 
the flow of  \eqref{equazione aggiunta}, namely $ \Psi^{\tau_0, \tau} (v) $ 
is the solution of 
\eqref{equazione aggiunta} with $ \Psi^{\tau_0, \tau_0} (v)  = v  $. 
Since the derivative $  \partial_\t (\Phi^\t (u_0)  , \Psi^{\tau_0,\t} (v) )_{L^2(\T)}  = 0  $, $ \forall \t  $,  we deduce that 
$  ( \Phi^{\tau_0} (u_0)  , \Psi^{\tau_0,\tau_0} (v) )_{L^2(\T)} =  ( \Phi^0 (u_0)  , \Psi^{\tau_0,0} (v) )_{L^2(\T)} $, namely
(recall that $ \Psi^{\tau_0,\tau_0} (v) = v $) the adjoint 
\be\label{adjoint flow}
(\Phi^{\tau_0})^T = \Psi^{\tau_0, 0}  \,, \quad \forall \tau_0 \in [0,1] \, . 
\ee
Thus it is sufficient to prove tame estimates for the flow $ \Psi^{\t_0, \tau} $. 
We first provide a useful expression for the solution  $ z (\tau, x) := \Psi^{\tau_0, \tau} (v) $ 
of  \eqref{equazione aggiunta}, obtained by the methods of characteristics. 
Let $ \gamma^{\t_0,\tau} (y) $ be the  flow of \eqref{equazione delle caratteristiche}.
Since $ \pa_\tau z (\tau, \g^{\t_0, \t} (y)) = - [\Pi_S (b(\t) \partial_x z (\t) ] ( \gamma^{\t_0,\t} (y))  $ we get 
$$
z(\tau,  \gamma^{\t_0,\tau} (y)) = v(y) + \int_\tau^{\tau_0} 
[\Pi_S (b(s) \partial_x z (s) ] ( \gamma^{\t_0,s} (y)) \,d s\, , \quad \forall \t \in [0,1] \, . 
$$
Denoting by $ y = x + \sigma(\tau, x)$ 
the inverse diffeomorphism of $ x = \gamma^{\tau_0, \tau} (y) = y +  {\tilde  \sigma}(\tau, y) $, we get
\begin{align}
\Psi^{\tau_0, \tau}(v) = z ( \tau, x)  & =  
v(x + \sigma(\tau, x)) + \int_\tau^{\tau_0}  
[\Pi_S (b(s) \partial_x z (s) ] ( \gamma^{\t_0,s} (x + \sigma(\tau, x))) \, d s \nonumber \\
& = v(x + \sigma(\tau, x)) +\int_\tau^{\tau_0} \sum_{j \in S} (z(s), p_j(s)) \kappa_j(s, x)\,d s = v( x +   \sigma(\tau,x) ) +  {\cal R}_\tau v\,,   \label{Psi-expression} 
\end{align}
where 
$  p_j (s) := - \partial_x (b(s) e^{\ii j x}) $, 
$ \kappa_j(s, x) :=  e^{\ii j \gamma^{\tau_0, s} (x + \sigma(\tau, x))} $ and 
$$
({\cal R}_\tau v)(x) :=  \int_\tau^{\tau_0} \sum_{j \in S} (\Psi^{\tau_0,s}(v), p_j(s))_{L^2(\T)} \kappa_j(s, x)\,d s\,.
$$
Since 
$\| \sigma(\tau, \cdot) \|_{W^{s,\infty}} $, $ \| \tilde \sigma(\tau, \cdot)\|_{W^{s,\infty}} 
\leq_s \| \beta \|_{W^{s + 1,\infty}}$ (recall also \eqref{transport-free}), we derive   
$ \| p_j \|_s \leq_s \| \beta \|_{W^{s + 2,\infty}} $,  
$ \| \kappa_j\|_s \leq_s 1 + \| \beta\|_{W^{s + 1,\infty}} $
and 
$ \| v(x +  \sigma(\tau, x)) \|_s \leq_s \| v\|_s + \| \beta\|_{W^{s + 1,\infty}} \| v\|_{s_0} $, $ \forall \tau \in [0,1] $. 
Moreover 
$$
\| {\cal R}_\tau v\|_s \leq_s {\rm sup}_{\tau \in [0, 1]} \| \Psi^{\tau_0,\tau}(v) \|_s 
\| \beta\|_{W^{s_0 + 2,\infty}} +{\rm sup}_{\tau \in [0, 1]} \| \Psi^{\tau_0, \tau}(v) \|_{s_0} \| \beta\|_{W^{s + 2,\infty}}\, .
$$
Therefore, for all $ \tau \in [0, 1] $, 
\be\label{tame aggiunto flusso parziale}
\| \Psi^{\tau_0,\tau} v\|_s \leq_s  \| v\|_s + \| \beta\|_{W^{s + 1,\infty}} 
\| v\|_{s_0} + {\rm sup}_{\tau \in [0, 1]} \big\{ \| \Psi^{\tau_0, \tau} v\|_s  \| \beta\|_{W^{s_0 + 2,\infty}}  +
\| \Psi^{\tau_0, \tau} v\|_{s_0} \| \beta\|_{W^{s + 2,\infty}} \big\}\,. 
\ee
For $s = s_0$ it implies 
$$
{\rm sup}_{\tau \in [0, 1]} \| \Psi^{\tau_0, \tau}(v)\|_{s_0} \leq_{s_0} \| v \|_{s_0}(1  + 
\| \beta \|_{W^{s_0 + 1,\infty}}) + {\rm sup}_{\tau \in [0, 1]} \| \Psi^{\tau_0, \tau}(v)\|_{s_0} \| \beta\|_{W^{s_0 + 2,\infty}} 
$$
and so, for  $\| \beta\|_{W^{s_0 + 2,\infty} } \leq c(s_0) $  small enough, 
\begin{equation}\label{tame norma bassa aggiunto flusso}
{\rm sup}_{\tau \in [0, 1]} \| \Psi^{\tau_0, \tau}(v)\|_{s_0} \leq_{s_0} \| v \|_{s_0} \, .
\end{equation}
Finally  \eqref{tame aggiunto flusso parziale}, \eqref{tame norma bassa aggiunto flusso} imply 
the tame estimate
\begin{equation}\label{tame aggiunto flusso}
{\rm sup}_{\tau \in [0, 1]} \| \Psi^{\tau_0, \tau} (v)\|_s \leq_s \| v \|_s + \| \beta\|_{W^{s + 2,\infty}} \| v \|_{s_0}\,.
\end{equation}
By \eqref{adjoint flow} and \eqref{tame aggiunto flusso} 
we deduce the bound \eqref{stime forma buona resto cambio di variabile hamiltoniano} 
for $ g_j $  defined in \eqref{espressione esplicita gj}.
The tame estimate \eqref{stime Phi Phi -1} for $\Phi $ follows by  that of $ {\cal A} $ and 
\eqref{stime forma buona resto cambio di variabile hamiltoniano}
 (use Lemma \ref{lemma:utile}). 
The estimate for $\Phi^{- 1}$ follows in the same way because 
$\Phi^{-1 } = \Phi^{1,0} $ is the backward flow.
\end{proof}

We  conjugate  $ {\cal L}_\om $  in  \eqref{Lom KdVnew} 
via the symplectic map $ \Phi = {\cal A}_\bot + {\cal R}_\Phi $
of Lemma \ref{modifica simplettica cambio di variabile}. We compute 
 (split $ \Pi_S^{\bot} = I - \Pi_S $)
\begin{equation}\label{LAbot}
{\cal L}_\om \Phi =  \Phi {\cal D}_\omega + 
\Pi_S^\bot {\cal A} \big( 
b_3 \partial_{yyy}  + b_2 \partial_{yy}   + b_1 \partial_{y}  + b_0   \big) \Pi_S^\bot +  {\cal R}_I \,,
\end{equation}
where the coefficients are
\begin{align}
& b_3 (\vphi,y) := {\cal A}^T [ a_1 ( 1 + \b_x)^3  ] \label{step1: b3KdV} 
\qquad \qquad
b_2 (\vphi,y) := {\cal A}^T \big[ 2 (a_1)_x (1 + \beta_x )^2 + 
6 a_1  \beta_{xx} (1 + \beta_x )\big]  \\ 
& b_1 (\vphi,y) := 
 {\cal A}^T \Big[ ({\cal D}_\om \beta) + 3 a_1 \frac{\beta_{xx}^2 }{1 + \beta_x} + 
 4 a_1 \beta_{xxx} + 6 (a_1)_x \beta_{xx}  + 
 (a_1)_{xx} (1 + \beta_x) + a_0 (1 + \b_x) \Big] \label{tilde b1KdV} 
\\ 
& b_0 (\vphi,y) := 
{\cal A}^T \Big[ \frac{({\cal D}_\om \beta_x)}{1 + \b_x} + 
 a_1 \frac{\beta_{xxxx}}{1+ \b_x} + 
 2 ( a_{1})_{x} \frac{\beta_{xxx}}{1+ \b_x} +  ( a_{1})_{xx} \frac{\beta_{xx}}{1+ \b_x} +
a_0  \frac{\beta_{xx}}{1+ \b_x}  + (a_0)_x   \Big] \label{tilde b0KdV}
\end{align}
 and the remainder
\begin{align}
 {\cal R}_I &:= -
\Pi_S^\bot \partial_x (   \e^2 {\cal R}_2 + \mR_{*}    ) {\cal A}_\bot  - \Pi_S^\bot
 \big( a_1 \partial_{xxx}  + 2 
(a_1)_x \partial_{xx}  + ( (a_{1})_{xx} + a_0)\partial_x  
+ (a_0)_x \big)  \Pi_{S} {\cal A}  \Pi_S^\bot \,  \nonumber\\
 & \quad +[{\cal D}_\omega, {\cal R}_\Phi] + ({\cal L}_\omega - {\cal D}_\omega) {\cal R}_\Phi\,. \label{calR1KdV}
\end{align}
The commutator $[{\cal D}_\omega, {\cal R}_\Phi] $ has the form \eqref{forma buona resto cambio di variabile hamiltoniano} with ${\cal D}_\om g_j$ or ${\cal D}_\om \chi_j$, ${\cal D}_\om \psi_j$ instead of $\chi_j$, $g_j$, $\psi_j$ respectively. Also the last term $({\cal L}_\omega - {\cal D}_\omega) {\cal R}_\Phi$ in \eqref{calR1KdV} has the form \eqref{forma buona resto cambio di variabile hamiltoniano} (note that ${\cal L}_\omega - {\cal D}_\omega$ does not contain derivatives with respect to $\vphi$). By \eqref{LAbot}, and decomposing $ I = \Pi_S + \Pi_S^\bot $, we get 
\begin{align} \label{L A bot finaleKdV}
{\cal L}_\om \Phi 
= {} & \Phi  ( {\cal D}_\omega  +  b_3 \partial_{yyy}  + b_2 \partial_{yy}   + b_1 \partial_{y}  + b_0 ) \Pi_S^\bot
+ {\cal R}_ {II} \,,
\\
\label{calR2KdV}
{\cal R}_{II}  
:= {} &  \big\{\Pi_S^\bot ({\cal A} - I) \Pi_{S} - {\cal R}_\Phi  \big\} 
( b_3 \partial_{yyy}  + b_2 \partial_{yy}   + b_1 \partial_{y}  + b_0 ) \Pi_S^\bot   + {\cal R}_I \,. 
\end{align}
Now we choose the function $ \beta = \b (\vphi, x)  $  such that 
\begin{equation}\label{choice beta}
a_1(\vphi, x) (1 + \b_x (\vphi, x))^3 = b_3 (\vphi) 
\end{equation}
so that the coefficient $ b_3 $ in \eqref{step1: b3KdV} depends only on $ \vphi $
(note that $ {\cal A}^T [b_3 (\vphi)]  = b_3 (\vphi) $). 
The only solution of \eqref{choice beta} with zero space average is (see e.g. \cite{BBM}-section 3.1)
\begin{equation}\label{sol beta}
\b  := \partial_x^{-1} \rho_0 , \quad 
\rho_0 := b_3 (\vphi)^{1/3} (a_1 (\vphi, x))^{-1/3} - 1,  \quad
b_3 (\vphi) := \Big( \frac{1}{2 \pi} \int_{\T} (a_1 (\vphi, x))^{-1/3} dx \Big)^{-3}. 
\end{equation}
Applying  the symplectic map  $ \Phi^{-1} $ in \eqref{L A bot finaleKdV} 
we  obtain the Hamiltonian operator (see Definition \ref{operatore Hamiltoniano})
\begin{equation}\label{cal L1 Kdv}
{\cal L}_1 := \Phi^{-1} {\cal L}_\om \Phi
= \Pi_S^\bot \big( \om \cdot \partial_\vphi + b_3(\vphi) \partial_{yyy} 
+ b_1 \partial_y + b_0 \big) \Pi_S^\bot +  {\mathfrak R}_1 
\end{equation}
where $ {\mathfrak R}_1 := \Phi^{-1} {\cal R}_{II} $. We used that, 
by the Hamiltonian nature of $ {\cal L}_1 $, the coefficient 
$ b_2 = 2 (b_3)_y $ (see \cite{BBM}-Remark 3.5) and so, by the choice \eqref{sol beta}, we have
$ b_2  =  2 (b_3)_y = 0 $. 
In the next Lemma we analyse the structure of the remainder ${\mathfrak R}_1$.
\begin{lemma} \label{cal R3}
The operator $ {\mathfrak R}_1 $ has the form \eqref{forma buona con gli integrali}.
\end{lemma}  

\begin{proof}
The remainders $ {\cal R}_I $ and $ {\cal R}_{II} $ have the form \eqref{forma buona con gli integrali}. Indeed
 $ {\cal R}_2,  {\cal R}_* $ in \eqref{calR1KdV} have the form \eqref{forma buona resto}
 (see Proposition \ref{prop:lin}) and the term
$ \Pi_S {\cal A} w = \sum_{j \in S } ( {\cal A}^T e^{\ii j x}, w)_{L^2(\T) } e^{\ii jx} $
has the same form. By \eqref{forma buona resto cambio di variabile hamiltoniano}, the terms of ${\cal R}_I$, ${\cal R}_{II}$ which involves the operator ${\cal R}_\Phi$ have the form \eqref{forma buona con gli integrali}.
All the operations involved preserve this structure: 
if $R_\tau w = \chi(\tau) (w, g(\tau) )_{L^2(\T)} $, $\tau \in [0, 1]$, then 
\begin{alignat*}{3}
R_\tau \Pi_S^\bot w & = \chi(\tau) (\Pi_S^\bot g(\tau) ,  w)_{L^2(\T)}\,, \  & 
R_\tau {\cal A} w & = \chi(\tau) ({\cal A}^T g(\tau) ,  w)_{L^2(\T)} \,, \ &
\partial_x R_\tau w & = \chi_x(\tau) (g(\tau) , w)_{L^2(\T)} \,, 
\\
\Pi_S^\bot R_\tau w & = (\Pi_S^\bot \chi(\tau)) (g(\tau) , w)_{L^2(\T)} \,, \   & 
{\cal A} R_\tau w & = ({\cal A} \chi(\tau)) (g(\tau) ,  w)_{L^2(\T)} \,, \  &
\Phi^{-1} R_\tau w & = (\Phi^{-1} \chi(\tau)) (g(\tau), w)_{L^2(\T)} 
\end{alignat*}
(the last equality holds because $ \Phi^{-1} ( f (\vphi) w ) = f (\vphi) \Phi^{-1} (  w ) $ for all function $f(\ph)$).  
 Hence $ {\mathfrak R}_1 $ has the form \eqref{forma buona con gli integrali} 
  where $ \chi_j(\tau)  \in H_S^\bot $ for all $\tau \in [0, 1]$.
\end{proof}

We now put in evidence the terms of order $ \e, \e^2, \ldots $, in $ b_1 $, $ b_0  $, $ \mathfrak R_1 $, recalling that $ a_1 -1 = O(\e^3 )$ (see \eqref{stima a1}), $ a_0 = O( \e) $ (see \eqref{a1p1p2}-\eqref{p geq 4}),
and  $ \beta  = O( \e^3)  $  (proved below in \eqref{stima beta}).
We expand  $ b_1$ in \eqref{tilde b1KdV} as
\begin{equation}\label{b1 Kdv}
b_1   = - \e p_1 -\e^2 p_2 - q_{>2} +  {\cal D}_\om \beta +  4  \b_{xxx} + (a_1)_{xx} 
  + b_{1, \geq 4}
\end{equation}
where $ b_{1, \geq 4} = O(\e^4) $ is defined by difference (the precise estimate is in Lemma \ref{lemma:stime coeff mL1}). 
 
\begin{remark}\label{media beta}
The function 
$ {\cal D}_\om \beta $ has zero average in $ x $ by \eqref{sol beta} as well as $ (a_1)_{xx}, \b_{xxx} $. 
\end{remark}

Similarly, we expand $ b_0 $ in \eqref{tilde b0KdV} as 
\be \label{b0 Kdv} 
b_0 = - \e (p_1)_x - \e^2 (p_2)_x - (q_{>2})_x  +  {\cal D}_\om \beta_x + \b_{xxxx} + b_{0, \geq 4} 
\ee
where $ b_{0, \geq 4} = O(\e^4) $ is defined by difference.

Using the equalities \eqref{calR2KdV},  \eqref{calR1KdV} and $ \Pi_S {\cal A} \Pi_S^\bot = \Pi_S ({\cal A} - I) \Pi_S^\bot $ we get
\begin{equation}\label{resto1 Kdv}
{\mathfrak R}_1 := \Phi^{-1} {\cal R}_{II} = 
-  \e^2 \Pi_S^\bot \partial_x  {\cal R}_2 + {\cal R}_{*} 
\end{equation}
where $ {\cal R}_2 $ is defined in \eqref{R nullo1R2} and we have renamed $\mR_*$ 
the term of order $ o(\e^2) $ in $\mathfrak{R}_1$. 
The remainder $ {\cal R}_{*} $ in \eqref{resto1 Kdv} has the form \eqref{forma buona con gli integrali}. 

\begin{lemma} \label{lemma:stime coeff mL1}
There is $\s = \s(\t,\nu) > 0$ such that 
\begin{align} \label{stima beta}
\| \b \|_s^\Lipg   \leq_s \e^3 (1 + \| \fracchi_\d \|_{s + 1}^\Lipg ), \qquad 
\| \pa_i \b [\widehat \imath ] \|_s 
& \leq_s \e^3 \big( \| \widehat \imath \|_{s + \s} +  \| {\mathfrak I}_\d \|_{s + \sigma}  \| \widehat \imath  \|_{s_0 + \s}\big)\,, \\
\label{stima b1 b0}
 \| b_3 - 1 \|_s^\Lipg  
 \leq_s \e^4 + \e^{b + 2} \| \fracchi_\d \|_{s + \s}^\Lipg, \qquad  \| \pa_i b_3[\widehat \imath ]\|_s & \leq_s \e^{b + 2} \big( \| \widehat 
 \imath \|_{s + \s} + \| \fracchi_\delta\|_{s + \s} \| \widehat \imath \|_{s_0 + \s}\big)
\\
\label{stima Di b1 b0}
\| b_{1, \geq 4} \|_s^\Lipg + \| b_{0, \geq 4} \|_s^\Lipg & \leq_s \e^4 + \e^{b + 2} \| \fracchi_\delta\|_{s + \s}^\Lipg  \\
\label{stima Di b geq 4}
  \| \pa_i b_{1, \geq 4}[\widehat \imath ] \|_s + \| \pa_i b_{0, \geq 4}[\widehat  \imath ]\|_s  & \leq_s \e^{b + 2}\big( \| \widehat \imath \|_{s + \s} + \| \fracchi_\delta\|_{s + \s} \| \widehat \imath \|_{s_0 + \s}\big) . 
\end{align}
The transformations $\Phi$, $\Phi^{-1}$ satisfy
\begin{align}
\label{stima cal A bot}
\|\Phi^{\pm 1} h \|_s^{\Lipg} & \leq_s \| h \|_{s + 1}^{\Lipg} + \| {\mathfrak I}_\delta \|_{s + \sigma}^{\Lipg} \| h \|_{s_0 + 1}^{\Lipg} \\
\label{stima derivata cal A bot}
\| \partial_i (\Phi^{\pm 1}h) [\widehat \imath] \|_s & \leq_s 
\| h \|_{s + \sigma} \| \widehat \imath \|_{s_0 + \sigma} + 
\| h\|_{s_0 + \sigma} \| \widehat \imath \|_{s + \sigma} + 
\| {\mathfrak I}_\delta\|_{s + \sigma} \| h\|_{s_0 + \sigma} \| \widehat \imath \|_{s_0 + \sigma}\,.
\end{align}
Moreover the remainder ${\cal R}_{*}$ has the form \eqref{forma buona con gli integrali},
where the functions $\chi_j(\tau)$, $g_j(\tau)$ satisfy the estimates 
\eqref{stima cal R*}-\eqref{derivate stima cal R*} uniformly in $\tau \in [0, 1]$.
\end{lemma}

\begin{proof}
The estimates \eqref{stima beta} follow by  \eqref{sol beta}, \eqref{stima a1}, 
and the usual interpolation and tame estimates in Lemmata 
\ref{lemma:composition of functions, Moser}-\ref{lemma:utile} (and Lemma \ref{stima Aep}) and \eqref{ansatz delta}.  
For the estimates of $ b_3 $, by \eqref{sol beta} and  \eqref{a1p1p2} we consider 
the function $ r_1 $ defined in \eqref{sigma0sigma1 def}. Recalling also
\eqref{finito finito} and \eqref{T0}, the  function 
$$ 
r_1 (T_\delta ) =  \e^3 (\partial_{u_x u_x} f_5) (v_\d, (v_\d)_x ) + r_{1, \geq 4} \, , \quad 
r_{1, \geq 4} := r_1 (T_\delta ) -  \e^3 (\partial_{u_x u_x} f_5) (v_\d, (v_\d)_x ) \, . 
$$
Hypothesis (${\mathtt S}1$) implies, as in the proof of Lemma \ref{p3 zero average}, that the space average
$ \int_{\T} (\partial_{u_x u_x} f_5) (v_\d, (v_\d)_x ) dx = 0 $. Hence the bound 
\eqref{stima b1 b0} for $ b_3 - 1 $ follows.
For the estimates on $\Phi$, $\Phi^{-1} $ we apply Lemma \ref{modifica simplettica cambio di variabile} and the estimate \eqref{stima beta} for $\beta$ .
We estimate the remainder ${\cal R}_*$ in \eqref{resto1 Kdv}, using \eqref{calR1KdV}, \eqref{calR2KdV} and
 \eqref{stima cal R*}-\eqref{derivate stima cal R*}. 
\end{proof}

\subsection{Reparametrization of time}\label{step2}

The goal of this section is to make constant the coefficient of the highest order 
spatial derivative operator $ \partial_{yyy} $, by a quasi-periodic reparametrization of time. 
We consider the change of variable 
$$
(B w)(\vphi, y) := w(\vphi + \omega \alpha(\vphi), y), 
\qquad  
( B^{-1} h)(\vartheta  , y ) := h(\vartheta + \omega \tilde{\alpha}(\vartheta), y)\,,
$$
where $ \vphi = \vartheta + \omega \tilde{\alpha}(\vartheta )$ 
is the inverse diffeomorphism of $ \vartheta =  \vphi + \omega \alpha(\vphi) $ in $\T^\nu$.
By conjugation, the differential operators become 
\begin{equation}  \label{anche def rho}
B^{-1} \om \cdot \partial_\ph B = \rho(\vartheta)\,   \omega \cdot \partial_{\vartheta} ,
\quad 
B^{-1}  \partial_y B = \partial_y, 
\quad 
\rho := B^{-1} (1 + \omega \cdot \partial_{\varphi} \a).
\end{equation}
By \eqref{cal L1 Kdv}, using also that   $  B $ and $ B^{-1} $ commute with $ \Pi_S^\bot $, we get 
\begin{equation} \label{secondo coniugio siti normali Kdv}
 B^{-1} {\cal L}_1 B 
= \Pi_S^\bot [ \rho  \omega \cdot \partial_{\vartheta} 
 + (B^{-1} b_3)  \partial_{yyy}  
 + ( B^{-1}  b_1 ) \partial_y + ( B^{-1} b_0 ) ] \Pi_S^\bot 
 + B^{-1} {\mathfrak R}_1 B.
\end{equation}
We choose $ \alpha $ such that 
\begin{equation}\label{B3solu}
(B^{-1}b_3 )(\vartheta ) = m_3 \rho (\vartheta) \,, \quad m_3 \in \R \, , \quad \text{namely} \ \ \  
 b_3 (\vphi) = m_3 ( 1 + \om \cdot \partial_\vphi \a (\vphi) ) 
\end{equation}
(recall  \eqref{anche def rho}). 
The unique solution with zero average of \eqref{B3solu} is
\begin{equation}  \label{def alpha m3}
\a (\vphi) := \frac{1}{m_3} ( \om \cdot \partial_\vphi )^{-1} ( b_3 - m_3 ) (\vphi) , 
\qquad 
m_3 := \frac{1}{(2 \pi)^\nu} \int_{\T^\nu} b_3 (\vphi) d \vphi  \,. 
\end{equation}
Hence, by \eqref{secondo coniugio siti normali Kdv}, 
\begin{alignat}{2} \label{L2 Kdv}
& B^{-1} {\cal L}_1 B = \rho {\cal L}_2 \, , \qquad &
& {\cal L}_2 := \Pi_S^\bot ( \omega \cdot \partial_{\vartheta} + m_3 \partial_{yyy} 
+ c_1 \partial_y + c_0 ) \Pi_S^\bot +  {\mathfrak R}_2
\\
\label{tilde c1 Kdv}
& c_1 := \rho^{-1} (B^{-1} b_1 ) \,, \qquad &
& c_0 :=  \rho^{-1} (B^{-1}  b_0 ) \, , \qquad 
{\mathfrak R}_2 := \rho^{-1} B^{-1}{\mathfrak R}_1 B \, .
\end{alignat}
The transformed operator ${\cal L}_2$ in \eqref{L2 Kdv} is still Hamiltonian, since the reparametrization of time preserves the Hamiltonian structure (see Section 2.2 and Remark 3.7 in \cite{BBM}).

We now put in evidence the terms of order $ \e, \e^2, \ldots $ in  $ c_1, c_0 $.
To this aim, we anticipate the following estimates: 
$ \rho (\vartheta) = 1 + O(\e^4) $,  
$ \alpha = O( \e^4 \g^{-1}) $, 
$ m_3 = 1 + O(\e^4) $, $ B^{-1} - I = O( \alpha ) $ (in low norm), 
which are proved in Lemma \ref{lemma:stime coeff mL2} below. 
Then, by \eqref{b1 Kdv}-\eqref{b0 Kdv}, we expand the functions $ c_1, c_0 $ in \eqref{tilde c1 Kdv} as
\begin{equation} \label{tilde c1 KdV}
c_1 =  - \e p_1 - \e^2 p_2 - B^{-1} q_{>2} + \e ( p_1 - B^{-1} p_1)  + \e^2 ( p_2 - B^{-1} p_2 )    + {\cal D}_\om \beta +  4 \b_{xxx} +  (a_1)_{xx}   + c_{1, \geq 4}  \, ,  
\end{equation}
\begin{equation} \label{tilde c0 KdV}
c_0 = - \e (p_1)_x  - \e^2 (p_2)_x - (B^{-1} q_{>2})_x 
+ \e ( p_1 - B^{-1} p_1)_x + \e^2 ( p_2 - B^{-1} p_2 )_x + ({\cal D}_\om \beta)_x +  \b_{xxxx}  + c_{0, \geq 4}\,, 
\end{equation}
where  $ c_{1, \geq 4}, c_{0, \geq 4} = O( \e^4 ) $ are defined by difference.

\begin{remark}
The functions $\e ( p_1 - B^{-1} p_1) = O( \e^5 \g^{-1} )$  and 
$\e^2 ( p_2 - B^{-1} p_2) = O( \e^6 \g^{-1} )$, see \eqref{commu B p1 p2}. 
For the reducibility scheme,  the terms of order $ \partial_x^0 $
with size  $  O(\e^5 \g^{-1}) $ are perturbative, since
$ \e^5 \g^{-2} \ll 1 $.
\end{remark}

The remainder $ {\mathfrak R}_2 $ in \eqref{tilde c1 Kdv} has still the form \eqref{forma buona con gli integrali} and, by \eqref{resto1 Kdv}, 
\begin{equation}\label{mathfrakR2}
{\mathfrak R}_2 :=  - \rho^{-1} B^{-1} {\mathfrak  R}_1 B   = 
-  \e^2 \Pi_S^\bot \partial_x  {\cal R}_2 +   {\cal R}_{*} 
\end{equation}
where $ {\cal R}_2 $ is defined in \eqref{R nullo1R2} and we have renamed $  {\cal R}_{*}  $ the  term of order $ o( \e^2 ) $ in $\mathfrak{R}_2$.

\begin{lemma} \label{lemma:stime coeff mL2}
There is $ \s = \s(\nu,\t) > 0 $ (possibly larger than $ \s $ in Lemma \ref{lemma:stime coeff mL1}) such that 
\begin{align} \label{stima m3}
| m_3 - 1 |^\Lipg    \leq  C \e^4 , \qquad  
| \pa_i m_3 [\widehat \imath ]| & \leq C 
\e^{b+2} \| \widehat \imath  \|_{s_0 + \s}  
\\
\label{stima Di b3} 
\| \a \|_s^\Lipg  \leq_s \e^4 \g^{-1} + \e^{b + 2} \gamma^{- 1}\| \fracchi_\d \|_{s + \s}^\Lipg, \qquad 
 \| \pa_i \a [\widehat \imath] \|_s
& \leq_s \e^{b + 2} \g^{-1} \big( \| \widehat \imath \|_{s + \s} +  \| {\mathfrak I}_\d \|_{s + \sigma}  \| \widehat \imath \|_{s_0 + \s}\big)\,, 
\\
\label{stima c1 c0}
 \| \rho -1 \|_s^\Lipg 
 \leq_s \e^4 + \e^{b + 2} \| \fracchi_\d \|_{s + \s}^\Lipg ,\qquad    \| \pa_i \rho[\widehat \imath ]\|_s &  \leq_s   \e^{b + 2} \big(\| \widehat 
 \imath \|_{s + \s} + \| \fracchi_\delta\|_{s + \s} \| \widehat \imath \|_{s_0 + \s} \big)      \\
\label{commu B p1 p2}
\| p_k - B^{-1} p_k \|_s^\Lipg  
& \leq_s \e^4 \g^{-1}  + \e^{b + 2} \gamma^{- 1} \| \fracchi_\d \|_{s + \s}^\Lipg , \quad k = 1, 2 
\\
\label{commu B p1 p2 der}
\| \partial_i  (p_k - B^{-1} p_k) [\widehat \imath] \|_s 
& \leq_s \e^{b + 2} \g^{-1} \big( \| \widehat \imath \|_{s + \s} +  \| {\mathfrak I}_\d \|_{s + \sigma}  \| \widehat \imath \|_{s_0 + \s}\big)\, 
\\
\label{stima z0 - B-1 z0}
\| B^{-1} q_{>2} \|_s^\Lipg  
& \leq_s  \e^3 + \e^b \| {\mathfrak I}_\delta\|_{s + \sigma}^\Lipg ,
\\
\label{stima derivata z0 - B-1 z0}
\|\partial_i (B^{-1} q_{>2}) [\widehat \imath] \|_s 
& \leq_s  \e^b \big( \| \widehat \imath \|_{s + \s} +  \| {\mathfrak I}_\d \|_{s + \sigma}  \| \widehat \imath \|_{s_0 + \s} \big) \,.
\end{align}
The terms  
$ c_{1, \geq 4},  c_{0, \geq 4} $ satisfy the bounds \eqref{stima Di b1 b0}-\eqref{stima Di b geq 4}. 
The transformations $B$, $B^{-1}$ satisfy the estimates \eqref{stima cal A bot}, \eqref{stima derivata cal A bot}.
The remainder $ {\cal R}_{*} $ has the form \eqref{forma buona con gli integrali}, and the functions $g_j(\tau)$, $\chi_j(\tau)$ satisfy 
the estimates \eqref{stima cal R*}-\eqref{derivate stima cal R*} for all $\tau \in [0,1]$.
\end{lemma}

\begin{proof}
\eqref{stima m3} follows from \eqref{def alpha m3},\eqref{stima b1 b0}. 
The estimate  $ \| \a \|_s \leq_s \e^4 \g^{-1}  + \e^{b + 2} \gamma^{- 1} \| \fracchi_\d \|_{s + \s} $ and the inequality for 
$ \pa_i \alpha $ in \eqref{stima Di b3} follow 
by \eqref{def alpha m3},\eqref{stima b1 b0},\eqref{stima m3}. 
For the first bound in \eqref{stima Di b3} we also differentiate 
\eqref{def alpha m3} 
with respect to the parameter $ \om $. The estimates for $\rho$ follow from $\rho - 1 = B^{-1}(b_3 - m_3) / m_3$.
\end{proof}

\subsection{Translation of the space variable}\label{step3}

In view of the next linear Birkhoff normal form steps (whose goal is to eliminate the terms of size $ \e $ and $ \e^2 $), 
in the expressions \eqref{tilde c1 KdV}, \eqref{tilde c0 KdV} we split 
$ p_1 = {\bar p}_1 + ( p_1 - {\bar p}_1)$, $p_2 = {\bar p}_2 + ( p_2 - {\bar p}_2)$ (see \eqref{a1p1p2}), where
\begin{equation}\label{def bar pi}
{\bar p}_1 :=  6 {\bar v},  \qquad 
{\bar p}_2 :=   6 \pi_0 [ (\partial_x^{-1} {\bar v})^2  ], \qquad 
{\bar v} (\vphi, x) := {\mathop \sum}_{j \in S} \sqrt{\xi_j} e^{\ii \ell (j) \cdot \vphi } e^{\ii j x}, 
\end{equation}
and $\ell : S \to \Z^\nu$ is the odd injective map (see \eqref{tang sites})
\be\label{del ell}
\ell : S \to \Z^\nu, \quad 
\ell(\bar \jmath_i) := e_i, \quad 
\ell(-\bar \jmath_i) := - \ell(\bar \jmath_i) = - e_i, \quad 
i = 1,\ldots,\nu,
\ee
denoting by $e_i = (0,\ldots,1, \ldots,0)$  the $i$-th vector of the canonical basis of $\R^\nu$.

\begin{remark}\label{remarkp1p2}
All  the functions $ {\bar p}_1  $, $ {\bar p}_2 $,  $ p_1 - {\bar p}_1  $,  $ p_2 - {\bar p}_2 $ 
have zero average in $ x $.
\end{remark}

We write the variable coefficients $c_1, c_0$ of the operator $\mL_2$ in \eqref{L2 Kdv} 
(see \eqref{tilde c1 KdV}, \eqref{tilde c0 KdV}) as
\begin{equation}\label{tilde c1 KdVbis} 
c_1 = - \e {\bar p}_1 - \e^2 {\bar p}_2 +  q_{c_1} + c_{1, \geq 4}\, , 
\qquad 
c_0 = - \e ({\bar p}_1)_x - \e^2 ({\bar p}_2)_x + q_{c_0} + c_{0, \geq 4}\, , 
\end{equation}
where we define 
\begin{gather} \label{def qc1 qc0}
q_{c_1} := q + 4 \b_{xxx} + (a_1)_{xx} \,, \quad 
q_{c_0} := q_x + \b_{xxxx}, 
\\
\label{def q}
q := \e ( p_1 - B^{-1} p_1) + \e ( {\bar p}_1 - p_1)
+ \e^2 ( p_2 - B^{-1} p_2 )  + \e^2 ( {\bar p}_2 - p_2)
- B^{-1} q_{>2} + {\cal D}_\om \beta \,.
\end{gather}

\begin{remark}  \label{rem:qc1 qc0}
The functions $ q_{c_1}, q_{c_0} $ have zero average in $ x $ (see Remarks \ref{remarkp1p2}, \ref{media beta} and Lemma \ref{p3 zero average}). 
\end{remark}

\begin{lemma} \label{lemma:stime scorporo}
 The functions $\bar p_k - p_k$, $k=1,2$ and $q_{c_m}$, 
$m = 0,1, $  satisfy
\begin{alignat}{2} 
\label{stima bar pk - pk}
& \| {\bar p}_k - p_k \|_s^\Lipg 
\leq_s \| \fracchi_\d \|_{s}^\Lipg , 
\quad & 
& \| \partial_i ({\bar p}_k - p_k)[ \widehat \imath] \|_s 
\leq_s \| \widehat \imath \|_{s} + \| \fracchi_\d \|_{s} \| \widehat \imath \|_{s_0}\,, 
\\
\label{stima q}
& \| q_{c_m} \|_s^\Lipg 
\leq_s \e^{5}\gamma^{-1} + \e \| {\mathfrak I}_\delta\|_{s + \s}^\Lipg\,, 
\quad & 
& \| \partial_i q_{c_m}[\widehat \imath ] \|_s^\Lipg
\leq_s \e \big( \| \widehat \imath \|_{s + \s} + \|{\mathfrak I}_\delta \|_{s + \s} \| \widehat \imath \|_{s_0 + \s} \big)\,. 
\end{alignat}
\end{lemma}

\begin{proof}
The bound \eqref{stima bar pk - pk} follows from 
\eqref{def bar pi}, \eqref{a1p1p2}, \eqref{T0}, \eqref{ansatz delta}. Then use \eqref{stima bar pk - pk},
\eqref{commu B p1 p2}-\eqref{stima derivata z0 - B-1 z0},
\eqref{stima beta},
\eqref{stima a1} to prove \eqref{stima q}. The biggest term comes from $ \e ({\bar p}_1 - p_1 ) $. 
\end{proof}

We now apply the transformation  $\mathcal{T}$ defined in \eqref{gran tau} 
whose goal  is to remove the space average from the coefficient in front of $ \partial_y $. 

Consider the change of the space variable $ z = y + p(\vartheta ) $ which induces on 
$ H^s_{S^\bot} (\T^{\nu+1}) $ the operators
\begin{equation}\label{gran tau}
 ({\cal T} w)(\vartheta, y ) := w(\vartheta , y + p(\vartheta)) \, , \quad 
 ({\cal T}^{-1} h) (\vartheta ,z ) = h(\vartheta, z - p(\vartheta))
 \end{equation}
(which are a particular case of those used in section \ref{step1}). 
The differential operator becomes 
$ {\cal T}^{-1} \omega \cdot \partial_{\vartheta} {\cal T} $ 
$ =  \omega \cdot \partial_{\vartheta} 
+ \{ \omega \cdot \partial_{\vartheta}p (\vartheta) \} \partial_z $, 
$ {\cal T}^{-1} \partial_{y} {\cal T}  =  \partial_{z} $. 
Since $\mT, \mT^{-1}$ commute with $ \Pi_S^\bot $, we get
\begin{align} \label{L3 KdV}
\mL_3 & := {\cal T}^{-1}{\cal L}_2 {\cal T} = \Pi_S^\bot \big(\omega \cdot \partial_{\vartheta} 
 + m_3 \partial_{zzz} + d_1 \partial_z + d_0 \big) \Pi_S^\bot 
 + {\mathfrak R}_3 \,,
\\
\label{d1d0R3}
d_1 & := ({\cal T}^{-1} c_1)   + \omega \cdot \partial_{\vartheta} p  \, , 
\qquad 
d_0 :=  {\cal T}^{-1}  c_0   \, , 
\qquad 
{\mathfrak R}_3 :=   {\cal T}^{-1} {\mathfrak R}_2 {\cal T}.
\end{align}
We choose 
\begin{equation} \label{def m1 p Kdv}
m_1 :=  \frac{1}{(2\pi)^{\nu + 1}} \int_{\T^{\nu + 1}}   c_1 d\vartheta dy \, , \quad 
p := (\omega \cdot \partial_\vartheta)^{-1} 
\Big( m_1 -  \frac{1}{2 \pi} \int_{\T}  c_1  d y \Big) \, ,
\end{equation}
so that $\frac{1}{2 \pi} \int_{\T} d_1 (\vartheta, z) \, dz = m_1$ for all $\vartheta \in \T^\nu$.
Note that, by \eqref{tilde c1 KdVbis}, 
\begin{equation} \label{media migliore}
\int_\T c_1(\th,y) \, dy = \int_\T c_{1, \geq 4}(\th,y) \, dy\,,
\quad 
\om \cdot \pa_{\vartheta} p(\th) = m_1 - \frac{1}{2\p} \int_\T c_{1, \geq 4} (\th,y)\, dy 
\end{equation}
because $\bar p_1, \bar p_2, q_{c_1}$ have all zero space-average.  
Also note that $\mathfrak R_3$ has the form \eqref{forma buona con gli integrali}. 
Since ${\cal T}$ is symplectic, the operator ${\cal L}_3$ in \eqref{L3 KdV} is Hamiltonian.
\begin{remark}
We require  Hypothesis (${\mathtt S}1$) so that the function $ q_{>2}  $ has zero space average (see Lemma \ref{p3 zero average}). 
If  $ q_{>2} $ did not have zero average,  
then $ p $ in \eqref{def m1 p Kdv} would have size $O(\e^3 \g^{-1})$ 
(see \eqref{def p3}) and, since  $\mT^{-1} - I = O(\e^3 \g^{-1}) $, 
the function $ \tilde d_0 $ in \eqref{d0 tilde KdV} would satisfy $  \tilde d_0  =  O(\e^4  \g^{-1}) $.
Therefore it would remain a term of order $ \partial_x^0 $ which is  not 
perturbative for the reducibility scheme of section \ref{subsec:mL0 mL5}.
\end{remark}

We put in evidence the terms of size $ \e, \e^2 $ in $ d_0 $, $ d_1 $,  $ {\mathfrak R}_3 $. 
Recalling \eqref{d1d0R3}, \eqref{tilde c1 KdVbis}, we split 
\begin{equation} \label{d1 d0 KdV}
d_1  = - \e {\bar p}_1 - \e^2 {\bar p}_2 +  \tilde d_{1}  \, , \quad  
d_0  = - \e ({\bar p}_1)_x - \e^2 ({\bar p}_2)_x + \tilde d_0  \, , \quad  
{\mathfrak R}_3 =  
-  \e^2 \Pi_S^\bot \partial_x \bar {\cal R}_2 +  \widetilde \mR_{*}
\end{equation}
where $ \bar {\cal R}_2 $ is obtained replacing 
$ v_\d $ with ${\bar v}$ in $  {\cal R}_2 $  (see \eqref{R nullo1R2}), and 
\begin{align} 
\label{d1 tilde KdV}
\tilde d_1 
& := \e (\bar p_1 - {\cal T}^{-1}{\bar p}_1 ) 
+ \e^2 (\bar p_2 - {\cal T}^{-1}{\bar p}_2 ) 
+ {\cal T}^{-1} ( q_{c_1} + c_{1, \geq 4}) 
+ \omega \cdot \partial_{\vartheta} p , 
\\
\label{d0 tilde KdV}
\tilde d_0 
& := \e ( \bar p_1 - {\cal T}^{-1}{\bar p}_1 )_x 
+ \e^2 ( \bar p_2 - {\cal T}^{-1} {\bar p}_2 )_x
+ {\cal T}^{-1} ( q_{c_0} + c_{0, \geq 4} ),
\\
\label{tilde mR 3}
\widetilde \mR_{*} 
& := \mT^{-1} \mR_{*} \mT 
+ \e^2 \Pi_S^\bot \pa_x (\mR_2 - \mT^{-1} \mR_2 \mT) + \e^2  \Pi_S^\bot \pa_x  (  \bar {\cal R}_2 - {\cal R}_2 ),
\end{align}
and $ \mR_{*} $ is defined in \eqref{mathfrakR2}. 
We have also used that ${\cal T}^{-1}$ commutes with $\partial_x$ and with $\Pi_S^\bot$. 

\begin{remark}\label{d1 media}
The space average 
$ \frac{1}{2 \pi} \int_{\T} \tilde d_1 (\vartheta, z) \, dz =  \frac{1}{2 \pi} \int_{\T} 
d_1 (\vartheta, z) \, dz = m_1$ for all $\vartheta \in \T^\nu$.
\end{remark}

\begin{lemma} \label{lemma:stime coeff mL3}
 There is $ \s := \s(\nu,\t) > 0 $ (possibly larger than in Lemma \ref{lemma:stime coeff mL2}) such that 
\begin{alignat}{2} \label{stima m1}
| m_1 |^\Lipg   & \leq C \e^4 , 
\quad & | \partial_i m_1 [\widehat \imath ]|  & \leq C \e^{b+2}  \| \widehat{\imath} \|_{s_0 + \sigma} 
\\ \label{stima p} 
\| p \|_s^\Lipg   & \leq_s \e^4 \g^{-1} + \e^{b + 2} \gamma^{- 1} \| \fracchi_\d \|_{s + \s}^\Lipg\,, 
\quad & \| \partial_i p [\widehat \imath] \|_s
& \leq_s \e^{b + 2} \g^{-1} \big( \| \widehat \imath \|_{s + \s} +  \| {\mathfrak I}_\d \|_{s + \sigma}  \| \widehat \imath \|_{s_0 + \s}\big)\,, 
\\ \label{tilde d1 d0 KdV}
\| \tilde d_k  \|_s^\Lipg   & \leq_s \e^{5} \gamma^{-1} + \e \| \fracchi_\d \|_{s + \s}^\Lipg \,, 
\quad & \| \partial_i \tilde d_k [\widehat \imath] \|_s 
& \leq_s \e \big(\| \widehat \imath \|_{s + \s} +  \| {\mathfrak I}_\d \|_{s + \sigma} 
\| \widehat \imath \|_{s_0 + \s}  \big)
\end{alignat}
for $k = 0,1$. Moreover the matrix $ s $-decay norm (see \eqref{matrix decay norm}) 
\begin{align}\label{nuove R3}
|\widetilde{\cal R}_{*}|_s^{\Lipg} & \leq_s   \e^{3} +  \e^{2}\| {\mathfrak I}_\delta \|_{s + \sigma}^\Lipg  \, , \quad 
|\partial_i \widetilde{\cal R}_{*} [\widehat \imath ]|_s  \leq_s \e^{2} \| \widehat \imath \|_{s + \sigma} + \e^{2 b - 1} \| {\mathfrak I}_\delta\|_{s + \sigma} \| \widehat \imath \|_{s_0 + \sigma}  \, .
\end{align}
The transformations ${\cal T}$, ${\cal T}^{-1}$ satisfy \eqref{stima cal A bot}, \eqref{stima derivata cal A bot}.
\end{lemma}

\begin{proof}
The estimates \eqref{stima m1}, \eqref{stima p} follow by \eqref{def m1 p Kdv},\eqref{tilde c1 KdVbis},\eqref{media migliore}, and the bounds for $ c_{1, \geq 4}, c_{0, \geq 4} $ in Lemma \ref{lemma:stime coeff mL2}. 
The estimates \eqref{tilde d1 d0 KdV} follow similarly  by 
\eqref{stima q}, 
\eqref{media migliore}, \eqref{stima p}.
The estimates \eqref{nuove R3} follow because
$ \mT^{-1} \mR_{*} \mT $  satisfies the bounds \eqref{stima cal R*} 
like $ \mR_{*} $ does (use Lemma \ref{remark : decay forma buona resto} and \eqref{stima p}) 
and 
$ |\e^2  \Pi_S^\bot \pa_x  (  \bar {\cal R}_2 - {\cal R}_2 )|_s^\Lipg \leq_s \e^2 \| {\mathfrak I}_\delta \|_{s + \sigma}^\Lipg $. 
\end{proof}
It is sufficient to estimate  $ \widetilde \mR_{*} $ (which has the form \eqref{forma buona con gli integrali})
only in the $ s $-decay norm (see \eqref{nuove R3}) because the next transformations will preserve it. Such norms are used 
in the reducibility scheme of section \ref{subsec:mL0 mL5}. 

\subsection{Linear Birkhoff normal form. Step 1} \label{BNF:step1}

Now we  eliminate the terms of order $ \e$ and $ \e^2  $ of $ {\cal L}_3 $.
This step is different from the 
reducibility steps that we shall perform in section \ref{subsec:mL0 mL5}, 
because the diophantine constant $ \g = o(\e^2 ) $ (see \eqref{omdio}) and so 
terms $ O( \e), O(\e^2) $ are not perturbative. 
This reduction is possible thanks to the special form of the terms 
$ \e {\cal B}_1$, $ \e^2 {\cal B}_2 $ defined in \eqref{def cal B1 B2}: 
the harmonics of $ \e {\cal B}_1$, and $ \e^2 T $ in \eqref{mL4 T R4}, which 
correspond to a possible small divisor
are naught, see  Corollary \ref{B1 zero KdV}, and Lemma \ref{pezzo epsilon 2 A}. 
In this section we eliminate the term  $ \e {\cal B}_1 $. In section \ref{BNF:step2}
we eliminate the terms of order $\e^2$.  

Note that, since the previous transformations $ \Phi $, $ B $, $ {\cal T} $ are 
$ O(\e^4 \g^{-1} ) $-close to the identity, the terms of order $\e $ and $ \e^2 $ in $ {\cal L}_3 $ 
are the same as in the original linearized operator.

\smallskip

We first collect all the terms of order $ \e $ and $ \e^2 $ in the operator $ {\cal L}_3 $ defined in \eqref{L3 KdV}. 
By  \eqref{d1 d0 KdV},  \eqref{R nullo1R2}, \eqref{def bar pi}
we have, renaming $ \vartheta = \vphi $, $ z = x $, 
$$
 \mL_3 = \Pi_S^\bot \big( \om  \cdot \pa_\vphi
 + m_3 \partial_{xxx}+ \e {\cal B}_1 + \e^2 {\cal B}_2 +
 {\tilde d}_1 \partial_x +  {\tilde d}_0 \big) \Pi_S^\bot 
 +  {\widetilde {\cal R}}_{*} 
$$
where $  {\tilde d}_1 $, $  {\tilde d}_0 $, $  {\widetilde {\cal R}}_{*}  $ are defined in \eqref{d1 tilde KdV}-\eqref{tilde mR 3} and 
(recall also \eqref{def pi 0})
\begin{equation} \label{def cal B1 B2}
{\cal B}_1 h := - 6  \partial_x ( {\bar v}  h), 
\quad
{\cal B}_2  h := - 6  \partial_x \{  {\bar v}  
\Pi_S [ (\partial_x^{-1} {\bar v} )\,\partial_x^{-1} h ] 
+ h \pi_0 [ ( \partial_x^{-1} {\bar v} )^2 ] \}  +
 6 \pi_0 \{ (\partial_x^{-1} {\bar v}) \Pi_S [ {\bar v} h ] \}.
\end{equation}
Note that ${\cal B}_1$ and ${\cal B}_2$ are the linear Hamiltonian vector fields of $ H_{S}^\bot $
generated, respectively, by the   Hamiltonian $ z \mapsto 3 \int_\T v z^2 $ in \eqref{H3tilde}, and 
the fourth order Birkhoff Hamiltonian $  {\cal H}_{4,2}$  in \eqref{mH3 mH4} at $ v = \bar v $.

We transform $ {\cal L}_3 $ by a symplectic operator 
$ \Phi_1 : H_{S^\bot}^s(\T^{\nu + 1}) \to H_{S^\bot}^s(\T^{\nu + 1}) $ of the form 
\begin{equation}\label{Phi_1}
\Phi_1 := {\rm exp}(\e A_1) = I_{H_S^\bot} + \e A_1 + \e^2 \frac{A_1^2}{2} +
\e^3 \widehat A_1, \quad 
\widehat A_1 := {\mathop \sum}_{k \geq 3} \frac{\e^{k - 3}}{k !} A_1^k  \, ,
\end{equation}
where $ A_1(\vphi) h = {\mathop \sum}_{j,j' \in S^c} ( A_1)_j^{j'}(\vphi) h_{j'} e^{\ii j x} $
is a Hamiltonian vector field.
The map  $ \Phi_1 $ is symplectic, because it is the time-1 flow of a Hamiltonian vector field. 
Therefore
\begin{align}\label{L3 new KdV}
& {\cal L}_3 \Phi_1 - \Phi_1 \Pi_S^\bot ( {\cal D}_\om + m_3 \partial_{xxx} ) \Pi_S^\bot \\
& = \Pi_S^\bot ( \e \{ {\cal D}_\om A_1 + m_3 [\partial_{xxx}, A_1] +  {\cal B}_1 \} 
+ \e^2 \{ {\cal B}_1 A_1 + {\cal B}_2 + \frac12 m_3 [\partial_{xxx}, A_1^2 ] + 
\frac12 ({\cal D}_\omega A_1^2)\} + {\tilde d}_1 \partial_x + R_3 ) \Pi_S^\bot
\nonumber
\end{align}
where 
\begin{align}\label{def R3}
R_3 & :=   {\tilde d}_1 \partial_x (\Phi_1 - I) \! +\! {\tilde d}_0 \Phi_1\! +\! \widetilde {\cal R}_{*} \Phi_1\! + \! \e^2 {\cal B}_2 (\Phi_1 - I) \! + \!   \e^3 \big\{ {\cal D}_\om \widehat A_1 \! + \!  m_3 [\partial_{xxx}, \widehat A_1]  \! + \! \frac12 {\cal B}_1 A_1^2 \!+ \! \e {\cal B}_1 \widehat A_1 \big\}\, . 
\end{align}

\begin{remark}
$ R_3 $ has no longer the form \eqref{forma buona con gli integrali}. 
However $ R_3 = O( \partial_x^0 ) $ because 
$ A_1 =  O(  \partial_x^{-1} ) $  (see Lemma \ref{lemma:Dx A bounded}), and therefore $\Phi_1 - I_{H_S^\bot} = O(\partial_x^{-1})$. 
Moreover the matrix decay norm of $ R_3 $ is $ o(\e^2) $. 
\end{remark}

In order to eliminate the order $\e$ from \eqref{L3 new KdV}, we choose 
\begin{equation}\label{cal A 1}
( A_1)_j^{j'}(l) := \begin{cases}
- \dfrac{( {\cal B}_1)_{j}^{j'}(l)}{\ii (\omega \cdot l + m_3( j'^3 - j^3))} & \text{if} \  \bar{\omega} \cdot l + j'^3 - j^3 \neq 0 \, ,  \\
0 & \text{otherwise},
\end{cases}
\qquad j, j' \in S^c, \ l \in \Z^\nu.
\end{equation}
This definition is well posed. 
Indeed, by \eqref{def cal B1 B2} and \eqref{def bar pi}, 
\begin{equation}\label{def cal B1 b}
( {\cal B}_1)_{j}^{j'}(l):=
\begin{cases}
-6 \ii j \sqrt{\xi_{j - j'}}   & \text{if} \ j - j' \in S\,, \  \  l = \ell(j - j') \\
0 & \text{otherwise}.
\end{cases}
\end{equation}
In particular $ ( {\cal B}_1)_{j}^{j'}(l) = 0 $ unless $ | l | \leq 1 $. Thus, for 
$ \bar{\omega} \cdot l + j'^3 - j^3 \neq 0 $, 
the denominators in \eqref{cal A 1} satisfy 
\begin{align}
|\omega \cdot l +m_3( j'^3 - j^3)| 
& = | m_3 (\bar \omega \cdot l + j'^3 - j^3) + ( \om - m_3 \bar \om ) \cdot l | 
\nonumber \\ & 
\geq  |m_3| | \bar \omega \cdot l + j'^3 - j^3 |
- | \om - m_3 \bar \om | |l| \geq  1/2  \, , \quad \forall |l| \leq  1 \, ,  
\label{BNFdeno}
\end{align}
for $ \e $ small, because the non zero integer 
$ | \bar{\omega} \cdot l + j'^3 - j^3| \geq 1 $, \eqref{stima m3}, and $ \om = \bar \om + O(\e^2) $. 

$ A_1$ defined in \eqref{cal A 1} is a Hamiltonian vector field as $\mB_1$. 

\begin{remark} 
\label{rem:Ham solving homolog}
This is a general fact: the denominators $ \d_{l,j,k} := \ii (\omega \cdot l + m_3( k^3 - j^3)) $ 
satisfy $ \overline{ \d_{l,j,k} } = \d_{-l,k,j} $ and 
an operator $G(\ph)$ is self-adjoint if and only if its matrix elements satisfy 
$ \overline{ G_j^k(l) } = G_k^j(-l) $, see \cite{BBM}-Remark 4.5. In a more intrinsic way, 
we could solve 
the  homological equation of this Birkhoff step directly for the Hamiltonian function whose flow generates $ \Phi_1 $.
\end{remark}

\begin{lemma} \label{lem:cubetto}  If $j,j' \in S^c$, $j - j' \in S$, $l =  \ell(j - j')$, then  
$ \bar{\omega} \cdot l + j'^3 - j^3  = 3 j j' (j' - j) \neq 0 $.
\end{lemma}

\begin{proof}
We have $ \bar \om \cdot l =  \bar \om \cdot \ell( j - j' ) = (j-j')^3$ because $ j - j' \in S$ (see \eqref{bar omega} and \eqref{del ell}).
Note that $ j, j' \neq 0 $ because $ j, j' \in S^c $, and $ j - j' \neq 0 $ because $ j - j' \in S $.
\end{proof}

\begin{corollary}\label{B1 zero KdV}
Let $j, j' \in S^c$. 
If $\bar{\omega} \cdot l + j'^3 - j^3 = 0$ then $( {\cal B}_1)_{j}^{j'}(l) = 0$. 
\end{corollary}

\begin{proof}
If $({\cal B}_1)_{j}^{j'}(l) \neq 0$ then $j - j' \! \in \! S, l =  \ell(j - j')$ 
by \eqref{def cal B1 b}. 
Hence $\bar{\omega} \cdot l + j'^3 - j^3 \! \neq 0$ by Lemma \ref{lem:cubetto}.
\end{proof}

By \eqref{cal A 1} and the previous corollary, the term of order $\e$ in \eqref{L3 new KdV} is
\be\label{primo termine BNF1}
\Pi_S^\bot \big( {\cal D}_\om A_1 + m_3 [\partial_{xxx}, A_1]  +  {\cal B}_1 \big) \Pi_S^\bot = 0 \, . 
\ee
We now estimate the transformation $ A_1 $.

\begin{lemma}\label{lem: A decay} 
$(i)$ For all $l \in \Z^\nu$, $j,j' \in S^c$, 
\begin{equation}\label{Acoef}
| (A_1)_{j}^{j'}(l)| \leq C (| j | +  | j'  |)^{-1}\, , \quad 
| (A_1)_{j}^{j'}(l)|^{\rm lip} \leq \e^{-2} (|j| + |j'|)^{-1} \,.
\end{equation}
$(ii)$ $ (A_1)_j^{j'}(l) = 0$ for all $l \in \Z^\nu$, $j,j' \in S^c$ such that $|j - j'| >  C_S $, where $C_S := \max\{ |j| : j \in S \}$. 
\end{lemma}

\begin{proof}
$(i)$ We already noted that $ (A_1)_j^{j'}(l) = 0$, $ \forall |l| > 1$. 
Since  $ | \om | \leq |\bar \om | + 1 $, one has, for $ |l| \leq 1 $, $ j \neq j' $,
\[
|\om \cdot l + m_3 (j'^3 - j^3)| 
\geq |m_3||j'^3 - j^3| - |\om \cdot l| 
\geq \frac14 (j'^2 + j^2) - |\om| 
\geq \frac18  (j'^2 + j^2) \, ,  \quad \forall (j'^2 + j^2) \geq C, 
\]
for some constant $C > 0$. 
Moreover, recalling that also \eqref{BNFdeno} holds, we deduce that for $  j \neq j' $, 
\begin{equation} \label{lower bound}
 (A_1)_j^{j'}(l) \neq 0 
\quad \Rightarrow \quad
|\om \cdot l + m_3 (j'^3 - j^3)| \geq c ( | j | +  | j' | )^2 \, . 
\end{equation}
On the other hand, if $ j = j' $, $ j \in S^c$,  the matrix $ (A_1)_j^j (l) = 0 $, $ \forall l \in \Z^\nu$,  
because $ ({\cal B}_1)_j^j (l) = 0 $ by \eqref{def cal B1 b} 
(recall that $ 0 \notin S $).  Hence \eqref{lower bound} holds for all $ j, j'  $.
By  \eqref{cal A 1}, \eqref{lower bound},  \eqref{def cal B1 b}  we deduce the first 
bound in \eqref{Acoef}. 
The Lipschitz bound follows similarly (use also $ |j - j'| \leq C_S $).  
$(ii)$ follows by \eqref{cal A 1}-\eqref{def cal B1 b}. 
\end{proof}

The previous lemma  means that $ A = O(| \partial_x|^{-1})$. 
More precisely we deduce that

\begin{lemma} \label{lemma:Dx A bounded}
$ | A_1 \pa_x  |_s^\Lipg + | \pa_x A_1 |_s^\Lipg \leq C(s) $.
\end{lemma}

\begin{proof}
Recalling the definition of the (space-time) matrix norm in \eqref{decayTop},  
since $(A_1)_{j_1}^{j_2}(l) = 0$ outside the set of indices $|l| \leq 1, |j_1 - j_2| \leq C_S$, 
we have
\begin{align*}
| \pa_x A_1 |_s^2 
& = \sum_{|l| \leq 1, \, |j| \leq C_S} 
\Big( \sup_{j_1 - j_2 = j} |j_1| | (A_1)_{j_1}^{j_2}(l)| \Big)^2 
\la l,j \ra^{2s}  
\leq C(s) 
\end{align*}
by Lemma \ref{lem: A decay}. The estimates for $ |A_1 \pa_x |_s $ and the Lipschitz bounds 
follow similarly. 
\end{proof}

It follows that the symplectic map $ \Phi_1 $ in \eqref{Phi_1} is invertible for $ \e $ small, 
with inverse
\begin{equation}\label{A1 check}
 \Phi_1^{-1} = {\rm exp}(-\e A_1) =  I_{H_S^\bot} + \e {\check A}_1 \, , \ 
  {\check A}_1 := {\mathop \sum}_{n \geq 1} \frac{\e^{n-1}}{n !} (-A_1)^n \, , \ 
| {\check A}_1 \pa_x  |_s^\Lipg + | \pa_x {\check A}_1 |_s^\Lipg \leq C(s) \, . 
\end{equation}
Since $ A_1 $ solves the homological equation \eqref{primo termine BNF1}, the $ \e  $-term in
\eqref{L3 new KdV} is zero, and, 
with a straightforward calculation, 
the $ \e^2 $-term simplifies to $ {\cal B}_2  + \frac12 [{\cal B}_1, A_1] $. 
We obtain the Hamiltonian operator
\begin{align} \label{bernardino1}
{\cal L}_4 &  :=    \Phi_1^{-1} {\cal L}_3 \Phi_1 
= \Pi_S^\bot ( {\cal D}_\om + m_3 \partial_{xxx} +  {\tilde d}_1 \pa_x 
+ \e^2  \{  {\cal B}_2  + \tfrac12 [{\cal B}_1, A_1]  \}  +     \tilde{R}_4 ) \Pi_S^\bot 
\\
\label{bernardino 2}
{\tilde R}_4 & := (\Phi_1^{-1} - I) \Pi_S^\bot [ \e^2  ( {\cal B}_2 + \tfrac12 [{\cal B}_1, A_1] ) + {\tilde d}_1 \partial_x ]  +  \Phi_1^{-1} \Pi_S^\bot R_3 \, .
\end{align}
We split $ A_1 $ defined in  \eqref{cal A 1}, \eqref{def cal B1 b} into $ A_1 = \bar{A}_1 + \widetilde{A}_1$ where, for all $ j, j' \in S^c $, $l \in \Z^\nu$, 
\begin{equation}\label{bar A1} 
( {\bar A}_1)_j^{j'}(l) := 
\dfrac{6 j \sqrt{\xi_{j - j'}}}{ \bar\omega \cdot l + j'^3 - j^3} 
\qquad \text{if } \  \bar{\omega} \cdot l + j'^3 - j^3 \neq 0, \ \  j-j' \in S , \ \   
l = \ell(j - j'), 
\end{equation}
and $( {\bar A}_1)_j^{j'}(l) := 0$ otherwise. 
By Lemma \ref{lem:cubetto}, for all $ j , j' \in S^c $, $l \in \Z^\nu$,  
$ ( {\bar A}_1)_j^{j'}(l) = 
\frac{2  \sqrt{\xi_{j - j'}}}{ j' ( j' - j )} 
$ if $  j-j' \in S $, $  l = \ell(j - j') $, and    
$ ( {\bar A}_1)_j^{j'}(l) = 0 $  otherwise, 
namely (recall the definition of $ \bar v $ in \eqref{def bar pi}) 
\begin{equation}\label{bar A1 esplicita}
{\bar A}_1h = 2 \Pi_S^\bot [( \partial_x^{-1} {\bar v}  )(\partial_x^{-1} h)]   \, , 
\quad \forall h \in H_{S^\bot}^s(\T^{\nu+1}) \,.
\end{equation}
The difference is
\begin{equation}\label{widetilde A1}
(\widetilde{A}_1)_j^{j'}(l) = (A_1 - \bar{A}_1)_j^{j'}(l) = - 
\frac{6 j \sqrt{\xi_{j - j'}}\big\{ ( \om - \bar \om) \cdot l + (m_3 - 1)(j'^3 - j^3) \big\}}{\big( \omega \cdot l + m_3(j'^3 - j^3) \big) \big(\bar\omega \cdot l + j'^3 - j^3 \big)} 
\end{equation}
for $ j, j' \in S^c $, $  j - j' \in S $,  $ l = \ell(j - j') $, 
and $ (\widetilde{A}_1)_j^{j'}(l) = 0 $ otherwise. 
Then, by \eqref{bernardino1}, 
\begin{equation}\label{mL4 T R4}
{\cal L}_4 = \Pi_S^\bot \big( {\cal D}_\om + m_3 \partial_{xxx} 
+  {\tilde d}_1 \pa_x 
+ \e^2 T +  R_4 \big) \Pi_S^\bot \,, 
\end{equation}
where
\be\label{T hamiltoniano}
T := {\cal B}_2 +  \frac12 [{\cal B}_1, {\bar A}_1]   \,, \qquad R_4 := \frac{\e^2}{2} [{\cal B}_1, \widetilde A_1] + \tilde R_4\,.
\ee
The operator $ T $ is Hamiltonian as $ {\cal B}_2 $,  $ {\cal B}_1 $, $ {\bar A}_1 $ (the commutator of two Hamiltonian vector fields is Hamiltonian). 

\begin{lemma} \label{lemma:R4}
 There is $ \s = \s(\nu,\t) > 0 $ (possibly larger than in Lemma \ref{lemma:stime coeff mL3}) such that 
\begin{align} \label{stima Lip R4} 
| R_4 |_s^{{\rm Lip}(\gamma)} 
& \leq_s \e^{5} \gamma^{-1} + \e \| {\mathfrak I}_\d \|_{s + \s}^{{\rm Lip}(\gamma)}  \,, \quad 
| \partial_i R_4 [\widehat \imath] |_s 
\leq_s \e \big( \| \widehat \imath \|_{s + \s} + \| {\mathfrak I}_\d \|_{s + \s} 
 \| \widehat \imath \|_{ s_0 + \s} \big)  \, . 
 \end{align}
\end{lemma}

\begin{proof} 
We first estimate $ [{\cal B}_1, \widetilde A_1] = 
({\cal B}_1 \partial_x^{-1}) (\pa_x {\widetilde A}_1)- ({\widetilde A}_1 \partial_x)( \partial_x^{-1} {\cal B}_1) $.
By  \eqref{widetilde A1}, $ | \om - \bar \om | \leq C \e^2 $ 
(as $ \om \in \Omega_\e $ in \eqref{Omega epsilon}) and \eqref{stima m3}, 
arguing as in Lemmata \ref{lem: A decay}, \ref{lemma:Dx A bounded}, we deduce that 
$ | {\widetilde A}_1 \pa_x  |_s^\Lipg +  $ $ | \pa_x {\widetilde A}_1 |_s^\Lipg \leq_s \e^2 $.
By  \eqref{def cal B1 B2} 
the norm $ |{\cal B}_1  \partial_x^{-1}|_s^\Lipg + |\partial_x^{-1} {\cal B}_1|^\Lipg \leq  C(s) $.  
Hence $ \e^2 |[{\cal B}_1,  {\widetilde A}_1]|_s^\Lipg \leq_s \e^4 $. 
Finally \eqref{T hamiltoniano}, 
\eqref{bernardino 2}, \eqref{A1 check}, \eqref{def R3}, \eqref{tilde d1 d0 KdV}, \eqref{nuove R3}, 
 and the interpolation estimate \eqref{interpm Lip} imply \eqref{stima Lip R4}.
\end{proof}

\subsection{Linear Birkhoff normal form. Step 2}\label{BNF:step2}

The goal of this section is to remove the term $ \e^2 T$ 
from the operator $\mL_4$ defined in   \eqref{mL4 T R4}.
We conjugate the Hamiltonian operator $\mL_4$  via a symplectic map
\be\label{def Phi2}
\Phi_2 := {\rm exp}(\e^2 A_2) = I_{H_S^\bot} + \e^2 A_2 + \e^4 \widehat A_2\,,\quad 
\widehat A_2 := {\mathop \sum}_{k \geq 2} \frac{\e^{2(k - 2)}}{k !} A_2^k
\ee
where $A_2(\ph) = {\mathop \sum}_{j,j' \in S^c} ( A_2)_j^{j'}(\vphi) h_{j'} e^{\ii j x} $ 
is a Hamiltonian vector field. We compute 
\begin{gather} \label{L4 diff}
{\cal L}_4 \Phi_2 - \Phi_2 \Pi_S^\bot \big( {\cal D}_\om + m_3 \partial_{xxx} \big) \Pi_S^\bot 
= \, \Pi_S^\bot ( \e^2 \{ {\cal D}_\om  A_2 + m_3 [\partial_{xxx}, A_2 ] + T \}  
+ \tilde{d}_1 \pa_x + \tilde R_5 ) \Pi_S^\bot \,,
\\
\label{def tilde R5}
\tilde R_5  := \Pi_S^\bot \{ 
\e^4 ( ({\cal D}_\omega \widehat A_2) + m_3 [\pa_{xxx}, \widehat A_2] )
+ (\tilde{d}_1 \partial_x  + \e^2 T)(\Phi_2 - I) + R_4 \Phi_2 \} \Pi_S^\bot \, . 
\end{gather}
We define
\begin{equation}\label{A2}
( A_2)_j^{j'}(l) := 
- \dfrac{T_j^{j'}(l)}{ \ii (\omega \cdot l + m_3 (j'^3 - j^3) )} 
\quad \text{if } \  \bar\omega \cdot l + j'^3 - j^3 \neq 0; 
\qquad 
( A_2)_j^{j'}(l) := 0 \quad \text{otherwise.}
\end{equation}
This definition is well posed. 
Indeed, by \eqref{T hamiltoniano}, \eqref{def cal B1 b}, \eqref{bar A1}, \eqref{def cal B1 B2}, 
the matrix entries $ T_j^{j'} (l) = 0 $ for all $ | j - j' | > 2 C_S $, $l \in \Z^\nu $, 
where $C_S := \max \{ |j| \, , j \in S \} $. 
Also $ T_j^{j'} (l) = 0 $ for all $ j, j'  \in S^c $, $ | l | > 2 $ (see also 
\eqref{forma funzionale B1 A1}, \eqref{B1B2 Fourier}, \eqref{B3 Fourier} below).
Thus, arguing as in \eqref{BNFdeno}, if $ \bar\omega \cdot l + j'^3 - j^3 \neq 0 $, then 
$ |\omega \cdot l + m_3 (j'^3 - j^3)| \geq 1 / 2 $.
The operator $ A_2 $ is a Hamiltonian vector field because $ T $ is Hamiltonian and by Remark \ref{rem:Ham solving homolog}.

Now we prove that the Birkhoff map 
$\Phi_2$  removes completely the term $\e^2 T$. 

\begin{lemma} \label{pezzo epsilon 2 A}
 Let $j, j' \in S^c$. If $\bar \omega \cdot l + j'^3 - j^3 = 0$, then $T_j^{j'}(l) = 0$.
\end{lemma}

\begin{proof}
By  \eqref{def cal B1 B2}, \eqref{bar A1 esplicita} we get
$ {\cal B}_1 {\bar A}_1  h 
= - 12 \partial_x \{ \bar v \Pi_S^\bot[( \partial_x^{-1} \bar v )(\partial_x^{-1} h)] \}
$, ${\bar A}_1{\cal B}_1 h = $ $ - 12 \Pi_S^\bot [ (\partial_x^{-1} \bar v) \Pi_S^\bot (\bar v h) ] $ for all $h \in H_{S^\bot}^s $, 
whence, recalling \eqref{def bar pi}, for all $ j, j' \in S^c $, $ l \in \Z^\nu $, 
\begin{equation}\label{forma funzionale B1 A1}
( [{\cal B}_1,  {\bar A}_1])_{j}^{j'}(l) = 12 \ii  \!\!\! \sum_{\begin{subarray}{c}
j_1, j_2 \in S, \, j_1 + j_2 = j - j' \\
j' + j_2 \in S^c, \, \ell(j_1) + \ell(j_2) = l 
\end{subarray}}   \!\!\!  \frac{ j j_1 -j' j_2 }{j' j_1 j_2} \sqrt{\xi_{j_1} \xi_{j_2}}\,, 
\end{equation}
If $([{\cal B}_1,  {\bar A}_1])_j^{j'}(l) \neq 0 $
there are $ j_1, j_2 \in S $ such that $ j_1 + j_2 = j - j' $, $j' + j_2 \in S^c$, $ \ell(j_1) + \ell(j_2) = l  $.
Then 
\begin{equation} \label{con lo zero}
\bar \omega \cdot l + j'^3 - j^3 
 = \bar \omega \cdot \ell(j_1 ) + \bar \om \cdot \ell(j_2) + j'^3 - j^3 
  \stackrel{\eqref{del ell}} 
 = j_1^3 + j_2^3 + j'^3 - j^3\, .
\end{equation}
Thus, if $\bar \omega \cdot l + j'^3 - j^3 = 0$, 
Lemma \ref{lemma:interi} implies $ (j_1 + j_2 )( j_1 + j') ( j_2 + j' ) = 0 $. Now $ j_1 + j' $, $  j_2 + j'  \neq 0 $
because $ j_1, j_2 \in S $, $ j' \in S^c $ and $ S $ is symmetric. 
Hence $ j_1 + j_2 = 0 $, which implies $ j = j' $ and $ l = 0 $ (the map $ \ell $ in \eqref{del ell} is odd).
In conclusion, if  $\bar \omega \cdot l + j'^3 - j^3 = 0$, the only nonzero matrix entry 
$ ([{\cal B}_1, {\bar A}_1])_j^{j'}(l) $   is 
\begin{equation}\label{parte diagonale zero}
([{\cal B}_1, \bar A_1])_j^j(0) 
\stackrel{\eqref{forma funzionale B1 A1}} = 24 \ii \sum_{j_2 \in S, \, j_2 + j \in S^c } 
\xi_{j_2} {j_2^{-1}}.
\end{equation}
Now we consider $ {\cal B}_2 $ in \eqref{def cal B1 B2}. Split $ {\cal B}_2 =  B_1 + B_2 + B_3 $, 
where 
$B_1 h := - 6 \partial_x \{ {\bar v} \Pi_S [ (\partial_x^{-1} {\bar v}) \partial_{x}^{-1} h ] \}$,  
$B_2 h := - 6 \partial_x \{ h \pi_0 [(\partial_x^{-1} {\bar v} )^2] \}$,  
$B_3 h := 6 \pi_0 \{ \Pi_S ({\bar v} h) \partial_x^{-1} {\bar v} \}$.
Their Fourier matrix representation is
\begin{gather}\label{B1B2 Fourier}
(B_1)_j^{j'}(l) =  6 \ii j \!\!\!\!\!\!\!\!\!\!\!\!\! \sum_{\begin{subarray}{c}
j_1, j_2 \in S, \, j_1 + j' \in S \\
j_1 + j_2 = j - j', \, \ell(j_1) + \ell(j_2) = l 
\end{subarray}} \!\!\!\! \!\!\!\!\!\!\! \frac{\sqrt{\xi_{j_1} \xi_{j_2}}}{ j_1 j'} \, ,  
\qquad 
(B_2)_{j}^{j'}(l) = 6 \ii j  \!\!\!\!\!\!\!\!\!\!\!\!\! \sum_{\begin{subarray}{c}
j_1 , j_2 \in S, \, j_1 + j_2 \neq 0 \\
j_1 + j_2 = j - j', \, \ell(j_1) + \ell(j_2) = l
\end{subarray}}   \!\!\!\!\!\!\!\!\!\!\! \frac{\sqrt{\xi_{j_1} \xi_{j_2}}}{j_1 j_2} \,, 
\\
\label{B3 Fourier}
(B_3)_j^{j'}(l) =  6 \!\!\!\!\!\!\!\!\!\!\!\!\! \sum_{\begin{subarray}{c}
j_1 , j_2 \in S,  \, j_1 + j' \in S \\
j_1 + j_2 = j - j', \,  \ell(j_1) + \ell(j_2) = l
\end{subarray}} \!\!\!\!\!\!\!\!\!\!\!\!\! \frac{\sqrt{\xi_{j_1} \xi_{j_2}}}{\ii j_2}\,,
\qquad j,j' \in S^c, \ l \in \Z^\nu.
\end{gather}
We study the terms $ B_1 $, $ B_2 $, $ B_3$ separately.
If $(B_1)_j^{j'}(l) \neq 0$, there are $j_1 , j_2 \in S$ such that 
$ j_1 + j_2 = j - j' $, $j_1 + j' \in S$, $ l = \ell(j_1) + \ell(j_2) $ and
\eqref{con lo zero} holds. Thus, if $\bar \omega \cdot l + j'^3 - j^3 = 0$,  
Lemma \ref{lemma:interi} implies 
$ (j_1 + j_2) (j_1 + j')(j_2 + j') = 0 $, 
and, since $j' \in S^c $ and $ S $ is symmetric, the only possibility is $j_1 + j_2 = 0$. Hence $j = j'$, $l = 0$.
In conclusion, if  $\bar \omega \cdot l + j'^3 - j^3 = 0$, the only nonzero matrix element 
$ ( B_1 )_j^{j'}(l) $   is 
\begin{equation} \label{diag R0}
(B_1)_j^j(0) = 6 \ii \sum_{ j_1 \in S, \, j_1 + j \in S } \xi_{j_1} j_1^{-1} \,.
\end{equation}
By the same arguments, if  $ (B_2)_j^{j'} (l) \neq 0 $ and $\bar \omega \cdot l + j'^3 - j^3 = 0$ we find $ (j_1 + j_2) (j_1 + j')(j_2 + j') = 0 $, 
which is impossible because also $ j_1 + j_2 \neq 0$.
Finally, arguing as for $ B_1 $, if  $\bar \omega \cdot l + j'^3 - j^3 = 0$, then the only nonzero matrix element $ ( B_3 )_j^{j'}(l) $ is
\begin{equation} \label{diag R1 R2}
(B_3)_j^j(0) = 6 \ii \sum_{j_1 \in S, \, j_1 + j \in S } \xi_{j_1}  j_1^{-1} \,.
\end{equation}
From \eqref{parte diagonale zero}, \eqref{diag R0}, \eqref{diag R1 R2} we deduce that, if 
$ \bar \omega \cdot l + j'^3 - j^3 = 0 $, then the only non zero elements  $ (\frac12 [\mB_1, \bar A_1] + B_1 + B_3)_j^{j'} (l) $ must be for $ (l, j, j') = (0, j, j) $. 
In this case, we get
\begin{equation}\label{Tjj0}
\frac12( [\mB_1, \bar A_1])_j^j(0) + (B_1)_j^j(0) + (B_3)_j^j(0) 
= 12 \ii \sum_{\begin{subarray}{c} j_1 \in S \\ j_1 + j \in S^c \end{subarray}} \frac{\xi_{j_1}}{ j_1}
+ 12 \ii \sum_{\begin{subarray}{c} j_1 \in S \\ j_1 + j \in S \end{subarray}} \frac{\xi_{j_1}}{ j_1}
= 12 \ii \sum_{j_1 \in S} \frac{\xi_{j_1}}{ j_1} 
= 0
\end{equation}
because 
the case $j_1 + j = 0$ is impossible ($j_1 \in S$, $j' \in S^c$ and $S$ is symmetric), 
and the function $S \ni j_1 \to \xi_{j_1} / j_1 \in \R$ is odd. 
The lemma follows by \eqref{T hamiltoniano}, \eqref{Tjj0}.
\end{proof}

The choice of $ A_2 $ in \eqref{A2} and Lemma \ref{pezzo epsilon 2 A} imply that 
\begin{equation}\label{pezzo zero}
 \Pi_S^\bot \big( {\cal D}_\om A_2  + m_3 [\partial_{xxx}, A_2] + T \big) \Pi_S^\bot = 0 \, .
\end{equation}

\begin{lemma}\label{A2 decay}
$ |\partial_x A_2|_s^\Lipg + |  A_2 \partial_x|_s^\Lipg \leq C(s) $. 
\end{lemma}

\begin{proof}
First we prove that the diagonal elements $ T_j^j (l) = 0 $ for all $ l \in \Z^\nu $. 
For $l = 0$, we have already proved that $T_j^j(0) = 0$ (apply Lemma \ref{pezzo epsilon 2 A} with $j = j'$, $l=0$).
Moreover, in each term $[\mB_1, \bar A_1]$, $B_1$, $B_2$, $B_3$ (see \eqref{forma funzionale B1 A1}, \eqref{B1B2 Fourier}, \eqref{B3 Fourier}) 
the sum is over $j_1 + j_2 = j - j'$, $l = \ell(j_1) + \ell(j_2)$. 
If $j = j'$, then $j_1 + j_2 = 0$, and $l = 0$. 
Thus $ T_j^j (l) =  T_j^j (0)  = 0 $. 
For the off-diagonal terms $ j \neq j' $ we argue as in Lemmata \ref{lem: A decay},  \ref{lemma:Dx A bounded},
using that all the denominators $ |\om \cdot l + m_3 ( j'^3 - j^3 ) | \geq c ( |j| + |j'|)^2 $.  
\end{proof}

For $ \e $ small, the map $ \Phi_2 $ in \eqref{def Phi2} is invertible and $ \Phi_2 = \exp(-\e^2 A_2 ) $. 
Therefore \eqref{L4 diff}, \eqref{pezzo zero} imply 
\begin{align} \label{L5 KdV}
{\cal L}_5 
& := \Phi_2^{-1} {\cal L}_4 \Phi_2 
= \Pi_S^\bot ( {\cal D}_\om + m_3 \partial_{xxx} +  \tilde{d}_1 \pa_x  + R_5 ) \Pi_S^\bot \,,
\\
\label{R5}
R_5 
& := ( \Phi_2^{-1} - I) \Pi_S^\bot \tilde{d}_1 \partial_x 
+ \Phi_2^{-1}  \Pi_S^\bot {\tilde R}_5 \,. 
\end{align}
Since $ A_2 $ is a Hamiltonian vector field, the map $\Phi_2$ is symplectic and so  ${\cal L}_5$ is Hamiltonian.
\begin{lemma} \label{lemma:R5}
$R_5$ satisfies the same estimates \eqref{stima Lip R4} as $R_4$ (with a possibly larger $\s$).
\end{lemma}

\begin{proof} 
Use \eqref{R5}, Lemma \ref{A2 decay}, \eqref{tilde d1 d0 KdV}, \eqref{def tilde R5}, 
\eqref{stima Lip R4} and the interpolation inequalities \eqref{multiplication Lip}, \eqref{interpm Lip}. 
\end{proof}

\subsection{Descent method}\label{step5}

The goal of this section is to transform $ {\cal L}_5 $ in \eqref{L5 KdV} so that the coefficient of $ \pa_x $ becomes constant.
We conjugate $ {\cal L}_5 $ via  a symplectic map of the form 
\begin{equation}\label{def descent}
{\cal S} := {\rm exp}(\Pi_S^\bot (w \partial_x^{-1}))\Pi_S^\bot = \Pi_S^\bot \big( I  + w   \partial_x^{-1} \big) \Pi_S^\bot + \widehat {\cal S}\,,\quad \widehat {\cal S} := {\mathop \sum}_{k \geq 2} 
\frac{1}{k!} [\Pi_S^\bot (w \partial_x^{-1})]^k \Pi_S^\bot\,,
\end{equation}
where $ w : \T^{\nu+1} \to \R $ is a function. 
Note that $\Pi_S^\bot (w \partial_x^{-1}) \Pi_S^\bot$ is the  Hamiltonian vector field  generated by 
 $ - \frac12 \int_\T w (\partial_x^{-1} h)^2\,dx$, $h \in H_S^\bot$.
Recalling \eqref{def pi 0}, we calculate
\begin{align} \label{L5 diff}
& {\cal L}_5 {\cal S} - {\cal S} \Pi_S^\bot
( {\cal D}_\om + m_3 \partial_{xxx}  + m_1 \partial_x ) \Pi_S^\bot 
= \Pi_S^\bot ( 3 m_3 w_x + \tilde{d}_1 - m_1 ) \partial_x \Pi_S^\bot + \tilde R_6 \,,
\\
& \tilde R_6
:= \Pi_S^\bot \{ ( 3 m_3 w_{xx} + \tilde d_1 \Pi_S^\bot w - m_1 w ) \pi_0 
+ ( ({\cal D}_\om w) + m_3 w_{xxx} + \tilde d_1 \Pi_S^\bot w_x ) \pa_x^{-1}
+ ({\cal D}_\omega \widehat{\cal S}) 
\notag \\ & \qquad \ \ 
+ m_3 [\partial_{xxx}, \widehat{\cal S}] 
+ \tilde{d}_1 \partial_x \widehat{\cal S} 
- m_1 \widehat{\cal S} \pa_x 
+ R_5 {\cal S} \} \Pi_S^\bot
\notag 
\end{align}
where $\tilde R_6$ collects all the terms of order at most $\pa_x^0$. 
By Remark \ref{d1 media}, we solve $ 3 m_3 w_x + \tilde{d}_1 - m_1 = 0 $  
by choosing $w := - (3 m_3)^{-1} \partial_x^{-1} ( \tilde{d}_1 - m_1 )$.
For $ \e $ small, the operator $ {\cal S} $ is invertible and, by \eqref{L5 diff}, 
\begin{equation}\label{def L6}
\mL_6 := \mS^{-1} \mL_5 \mS 
= \Pi_S^\bot ( {\cal D}_\om + m_3 \partial_{xxx} + m_1 \partial_x ) \Pi_S^\bot + R_6 \,, 
\qquad R_6 :=  {\cal S}^{-1} \tilde R_6 \, .
\end{equation}
Since $ {\cal S} $ is symplectic, ${\cal L}_6$ is Hamiltonian (recall Definition \ref{operatore Hamiltoniano}).

\begin{lemma}\label{lemma L6}
 There is $ \s = \s(\nu,\t) > 0 $ (possibly larger than in Lemma \ref{lemma:R5}) such that 
$$
|{\cal S}^{\pm 1} - I|_s^\Lipg 
\leq_s \e^{5} \gamma^{-1} + \e \| {\mathfrak I}_\delta\|_{s + \sigma}^\Lipg\,, \quad 
|\partial_i {\cal S}^{\pm 1} [\widehat \imath ]|_s  
\leq_s \e ( \| \widehat \imath\|_{s + \sigma} + \|{\mathfrak I}_\delta \|_{s + \sigma} \| \widehat \imath \|_{s_0 + \sigma} ).
$$
The remainder $R_6$ satisfies the same estimates \eqref{stima Lip R4} as $R_4$. 
\end{lemma}

\begin{proof}
By \eqref{tilde d1 d0 KdV},\eqref{stima m1},\eqref{stima m3}, $\| w \|_s^\Lipg \leq_s  \e^{5} \gamma^{-1} + \e \| {\mathfrak I}_\delta\|_{s + \sigma}^\Lipg $, and the lemma follows by \eqref{def descent}. Since $ \widehat { \cal S } = O(\pa_x^{-2})$ the commutator 
$ [ \pa_{xxx}, \widehat { \cal S }] = O(\pa_x^0 )$ and 
$ | [ \pa_{xxx}, \widehat { \cal S }] |_s^\Lipg \leq_s \| w \|_{s_0+3}^\Lipg \| w \|_{s+3}^\Lipg $. 
\end{proof}

\subsection{KAM reducibility and inversion of $ {\cal L}_{\om} $} \label{subsec:mL0 mL5}

The coefficients $ m_3, m_1 $  of the operator $ {\cal L}_6 $ in \eqref{def L6} are constants, and the remainder $ R_6 $ is a bounded operator of order $ \partial_x^0 $ with  small matrix decay 
norm, see \eqref{decay R6}.
Then we  can diagonalize $ {\cal L}_6 $ 
by applying the iterative KAM reducibility Theorem 4.2  in \cite{BBM}  along the 
sequence of scales  
\begin{equation}\label{defN}
N_n := N_{0}^{\chi^n}, \quad n = 0,1,2,\ldots,  
\quad  \chi := 3/2, \quad N_0  > 0 \, .
\end{equation}
In section \ref{sec:NM}, the initial $ N_0 $ will (slightly) increase to infinity 
as $ \e \to 0 $, see \eqref{nash moser smallness condition}. 
The required smallness condition (see (4.14) in \cite{BBM}) is (written in the present notations)
\be\label{R6resto}
N_0^{C_0} | R_6 |_{s_0 + \beta}^{\Lipg} \g^{-1} \leq 1 
\ee
where $ \b := 7 \tau + 6  $ (see (4.1) in \cite{BBM}), 
$ \tau $ is the diophantine exponent in \eqref{omdio} and \eqref{Omegainfty}, 
and the constant $ C_0 := C_0 (\t, \nu ) > 0 $  is fixed in
Theorem 4.2 in \cite{BBM}.
By Lemma \ref{lemma L6}, the remainder $ R_6 $ satisfies the bound \eqref{stima Lip R4}, 
and using \eqref{ansatz delta} we get (recall \eqref{link gamma b})
\be \label{decay R6}
| R_6|_{s_0 + \beta}^{\Lipg} \leq C \e^{7 - 2 b} \gamma^{-1} = C \e^{3 - 2 a}, \qquad
| R_6 |_{s_0 + \beta}^{\Lipg} \g^{-1} \leq C \e^{1 - 3 a}  \, .
\ee
We use that  $ \mu $ in \eqref{ansatz delta}  is assumed to satisfy $ \mu  \geq \s + \b $ where $ \s := \s (\tau, \nu ) $ 
is given in Lemma \ref{lemma L6}.  

\begin{theorem} \label{teoremadiriducibilita} 
{\bf (Reducibility)}
Assume that  $\omega \mapsto i_\d (\omega)
$ is a Lipschitz function defined on some subset $\Omega_o \subset \Omega_\e $ (recall  \eqref{Omega epsilon}), satisfying 
\eqref{ansatz delta} with  $ \mu \geq \s + \b $ where $ \s := \s (\tau, \nu) $ is given in Lemma \ref{lemma L6} and $ \b  := 7 \tau + 6 $. 
Then there exists $ \delta_{0} \in (0,1) $  such that, if
\begin{equation}\label{condizione-kam}
N_0^{C_0}  \e^{7 - 2 b} \gamma^{-2} = N_0^{C_0}  \e^{1 - 3 a}
\leq \delta_{0} \, , \quad \g := \e^{2 + a} \, , \quad a \in (0,1/6) \, , 
\end{equation}
then:

$(i)$ {\bf (Eigenvalues)}.
For all $ \omega \in \Omega_\e $ there exists a sequence 
\begin{equation} \label{espressione autovalori}
\mu_j^\infty(\omega) := \mu_j^\infty(\omega, i_\d (\om)) 
:=  \ii \big( - {\tilde m}_3 (\omega) j^3 +  {\tilde m}_1(\omega)  j \big) 
+ r_j^\infty(\omega), \quad j \in  S^c \, ,  
\end{equation}
where $ {\tilde m}_3, {\tilde m}_1$  coincide with the coefficients $m_3, m_1$ of $ {\cal L}_6 $ in \eqref{def L6} for all $ \omega \in \Omega_o $,
and 
\begin{align}
\label{autofinali}
| {\tilde m}_3 - 1 |^{{\rm Lip}(\gamma)} + | {\tilde m}_1 |^{{\rm Lip}(\gamma)} \leq C \e^4 \, , \quad 
| r^{\infty}_j |^{{\rm Lip}(\gamma)} & \leq C \e^{3 - 2 a} \, , 
\quad \ \forall j \in  S^c \,  , 
\end{align}
for some  $ C > 0 $. 
All the eigenvalues $\mu_j^{\infty}$ are purely imaginary. We define, for convenience,  $  \mu_0^\infty (\om)  := 0 $. 

\smallskip

$(ii)$ {\bf (Conjugacy)}.
For all $\omega$ in the set
\begin{equation}  \label{Omegainfty}
\Omega_\infty^{2\g} := \Omega_\infty^{2\g} (i_\d) 
:= \Big\{ \omega \in \Omega_o : \,
| \ii\omega \cdot l + \mu^{\infty}_j (\omega) - \mu^{\infty}_{k} (\omega) |
\geq \frac{2 \gamma | j^{3} - k^{3} |}{ \langle l \rangle^{\tau}}, 
\, \forall l \in \Z^{\nu}, \, j ,k \in S^c \cup \{0\}\Big\} 
\end{equation}
there is a real, bounded, invertible linear operator $\Phi_\infty(\omega) : H^s_{S^\bot} (\T^{\nu+1}) \to 
H^s_{S^\bot} (\T^{\nu+1}) $, with bounded inverse 
$\Phi_\infty^{-1}(\omega)$, that conjugates $\mL_6$ in \eqref{def L6} to constant coefficients, namely
\begin{equation}\label{Lfinale}
{\cal L}_{\infty}(\omega)
:= \Phi_{\infty}\inv(\omega) \circ \mL_6(\omega) \circ  \Phi_{\infty}(\omega)
=  \om \cdot \partial_{\vphi} + {\cal D}_{\infty}(\omega),  \quad
{\cal D}_{\infty}(\omega)
:= {\rm diag}_{j \in S^c} \{ \mu^{\infty}_{j}(\omega) \} \, .
\end{equation}
The transformations $\Phi_\infty, \Phi_\infty\inv$ are close to the identity in matrix decay norm,
with
\begin{equation} \label{stima Phi infty}
| \Phi_{\infty} - I |_{s,\Omega_\infty^{2\g}}^{{\rm Lip}(\gamma)}
+ | \Phi_{\infty}^{- 1} - I |_{s,\Omega_\infty^{2\g}}^\Lipg
\leq_s \e^{5} \gamma^{-2} + \e \gamma^{-1} \| {\mathfrak I}_\delta \|_{s + \sigma}^\Lipg .
\end{equation}
Moreover $\Phi_{\infty}, \Phi_{\infty}^{-1}$  are symplectic, and 
$\mL_\infty $ is a Hamiltonian operator.
\end{theorem}

\begin{proof}
The proof is the same as the one of Theorem 4.1 in \cite{BBM}, 
which is based on Theorem 4.2, Corollaries 4.1, 4.2 and Lemmata 4.1, 4.2 of \cite{BBM}.
A difference is that here $\om \in \R^\nu$, while in \cite{BBM} the parameter $\lm \in \R$ is one-dimensional. 
The proof is the same because Kirszbraun's Theorem on Lipschitz extension of functions also holds in $\R^\nu$ 
(see, e.g., Lemma A.2 in \cite{Po2}). 
The bound \eqref{stima Phi infty} follows by Corollary 4.1 of \cite{BBM} and the estimate of 
$ R_6 $ in Lemma \ref{lemma L6}.
We also use the estimates \eqref{stima m3}, \eqref{stima m1} for $ \pa_i m_3 $, 
$ \pa_i m_1 $ which correspond to (3.64) in \cite{BBM}.  
Another difference is that here the sites $ j \in S^c \subset \Z \setminus \{0\} $ unlike in \cite{BBM} where $ j \in \Z $. 
We have defined $ \mu_0^\infty := 0 $ 
so that also the first  Melnikov conditions \eqref{prime di melnikov} are included  in the definition of $ \Om^{2 \g}_{\infty} $.
\end{proof}

\begin{remark}
Theorem 4.2 in \cite{BBM} 
also provides the Lipschitz dependence of the (approximate) 
eigenvalues $ \mu_j^n $ with respect to the unknown $ i_0 (\vphi) $,  
which is used for the measure estimate Lemma \ref{matteo 10}.
 \end{remark}

All the parameters $ \omega \in \Omega_\infty^{2 \gamma} $ satisfy  (specialize \eqref{Omegainfty} for $ k = 0 $) 
\begin{equation}\label{prime di melnikov}
|\ii \omega \cdot l + \mu_j^\infty(\omega)| \geq 
2 \g | j |^3 \langle l \rangle^{-\tau} \, , \quad \forall l \in \Z^\nu , \ j \in S^c, 
\end{equation}
and the diagonal operator $ {\cal L}_\infty $ is invertible.  

In the following theorem we finally verify the inversion assumption \eqref{tame inverse} for ${\cal L}_\om $. 

\begin{theorem}\label{inversione linearized normale} {\bf (Inversion of $ {\cal L}_\om $)}
Assume the hypotheses of Theorem \ref{teoremadiriducibilita} and \eqref{condizione-kam}. 
Then there exists $ \s_1 := \s_1 ( \t, \nu ) >  0 $ such that, 
$ \forall \omega \in \Omega^{2 \gamma}_\infty(i_\d )$ (see \eqref{Omegainfty}),
for any function $ g \in H^{s+\s_1}_{S^\bot} (\T^{\nu+1})  $ 
the equation  ${\cal L}_\omega h = g$ 
has a solution $h = {\cal L}_\omega^{-1} g \in H^s_{S^\bot} (\T^{\nu+1})$, satisfying
\begin{align}\label{stima inverso linearizzato normale}
\| {\cal L}_\omega^{-1} g \|_s^{{\rm Lip}(\gamma)} 
& \leq_s \gamma^{-1} \big( \| g \|_{s +\sigma_1}^{{\rm Lip}(\gamma)} 
+ \e \gamma^{-1} \| {\mathfrak I}_\delta\|_{s + \sigma_1}^\Lipg
\| g \|_{s_0}^{{\rm Lip}(\gamma)} \big) 
\\
& \leq_s \gamma^{-1} \big( \| g \|_{s +\sigma_1}^{{\rm Lip}(\gamma)} 
+ \e \gamma^{-1} \big\{  \| \fracchi_0 \|_{s + \s_1 + \s}^\Lipg + \g^{-1} 
 \| \fracchi_0 \|_{s_0 +  \s }^\Lipg
\| Z \|_{s + \s_1 + \s}^\Lipg \big\}
\| g \|_{s_0}^{{\rm Lip}(\gamma)} \big)\, . \nonumber 
\end{align}
\end{theorem}

\begin{proof}
Collecting Theorem \ref{teoremadiriducibilita} with the results of 
sections \ref{step1}-\ref{step5}, we have obtained the (semi)-conjugation of the operator 
$ \mL_\om $ (defined in \eqref{Lom KdVnew}) 
to $\mL_\infty $ (defined in \eqref{Lfinale}), namely  
\begin{equation} \label{def cal M 12}
\mL_\om = {\cal M}_1 \mL_\infty {\cal M}_2\inv, \qquad 
{\cal M}_1 := \Phi B \rho {\cal T}  \Phi_1 \Phi_2  {\cal S} \Phi_{\infty}, \quad 
{\cal M}_2 := \Phi B {\cal T}  \Phi_1 \Phi_2 {\cal S} \Phi_{\infty} \,,
\end{equation}
where $ \rho $ means the multiplication operator by the function $ \rho $ defined in \eqref{anche def rho}. 
By \eqref{prime di melnikov} and Lemma 4.2 of \cite{BBM} we deduce that 
$ \| {\cal L}_\infty^{-1} g \|_s^\Lipg \leq_s \g^{-1} \| g \|_{s+ 2 \tau + 1}^\Lipg $.
In order to estimate $\mM_2, \mM_1^{-1}$, we recall that the composition of tame maps is tame, see Lemma 6.5 in \cite{BBM}. 
Now,  
$ \Phi ,  \Phi^{-1}$ are estimated in Lemma \ref{lemma:stime coeff mL1}, 
$ B, B^{-1} $ and $ \rho $ in Lemma \ref{lemma:stime coeff mL2}, 
$ {\cal T}, {\cal T}^{-1} $ in Lemma \ref{lemma:stime coeff mL3}. 
The decay norms $ |\Phi_1|_s^\Lipg$, $ |\Phi_1^{-1}|_s^\Lipg $, $ |\Phi_2|_s^\Lipg $, $ |\Phi_2^{-1}|_s^\Lipg \leq C(s) $ by  Lemmata \ref{lemma:Dx A bounded}, \ref{A2 decay}. 
The decay norm of $ {\cal S}, {\cal S}^{-1} $ is estimated in Lemma \ref{lemma L6}, 
and $ \Phi_\infty, \Phi_\infty^{-1} $ in \eqref{stima Phi infty}. 
The decay norm controls the Sobolev norm by \eqref{interpolazione norme miste}. 
Thus, by \eqref{def cal M 12}, 
\[
\| \mM_2 h \|_s^\Lipg + \| \mM_1^{-1} h \|_s^\Lipg 
\leq_s \| h \|_{s+3}^\Lipg + \e \g^{-1} \| \fracchi_\d \|_{s+\s+3}^\Lipg \| h \|_{s_0}^\Lipg \,,
\]
and  \eqref{stima inverso linearizzato normale}  follows. 
The last inequality in \eqref{stima inverso linearizzato normale} follows by
\eqref{stima y - y delta} and \eqref{ansatz 0}.
\end{proof}

\section{The Nash-Moser nonlinear iteration}\label{sec:NM}
 
In this section we prove Theorem \ref{main theorem}. It will be a consequence of the Nash-Moser Theorem \ref{iterazione-non-lineare} below.

Consider the finite-dimensional subspaces
\[
E_n := \big\{ \fracchi (\vphi) = ( \Theta, y, z )(\vphi)  : \, \Theta = \Pi_n \Theta, \ y = \Pi_n y, \ z = \Pi_n z \big\}
\]
where $ N_n := N_0^{\chi^n} $ are introduced  in \eqref{defN}, and $ \Pi_n $ are the projectors 
(which, with a small abuse of notation, we denote with the same symbol)
\begin{align}
\Pi_n \Theta (\ph) :=  \sum_{|l| < N_n} \Theta_l e^{\ii l \cdot \ph}, \quad 
\Pi_n y (\ph) :=  \sum_{|l| < N_n} y_l e^{\ii l \cdot \ph}, \quad 
& \text{where} \ \Theta (\ph) = \sum_{l \in \Z^\nu} \Theta_l e^{\ii l \cdot \ph}, \quad 
y(\ph) = \sum_{l \in \Z^\nu} y_l e^{\ii l \cdot \ph}, \nonumber
\\
\Pi_n z(\ph,x) := \sum_{|(l,j)| < N_n} z_{lj} e^{\ii (l \cdot \ph + jx)}, \quad  
& \text{where} \ z(\ph,x) = \sum_{l \in \Z^\nu,  j \in S^c} z_{lj} e^{\ii (l \cdot \ph + jx)}. \label{Pin def}
\end{align}
We define $ \Pi_n^\bot := I - \Pi_n $.  
The classical smoothing properties hold: for all $\alpha , s \geq 0$, 
\begin{equation}\label{smoothing-u1}
\|\Pi_{n} \fracchi \|_{s + \alpha}^\Lipg 
\leq N_{n}^{\alpha} \| \fracchi \|_{s}^\Lipg \, , 
\  \forall \fracchi (\om) \in H^{s} \,,  
\quad
\|\Pi_{n}^\bot \fracchi \|_{s}^\Lipg 
\leq N_{n}^{-\alpha} \| \fracchi \|_{s + \alpha}^\Lipg \, , 
\  \forall \fracchi (\om) \in H^{s + \alpha} \, .
\end{equation}
We define the constants
\begin{alignat}{3} \label{costanti nash moser}
& \mu_1 := 3 \mu + 9\,,\quad &
& \alpha := 3 \mu_1 + 1\,,\quad &
& \alpha_1 := (\alpha - 3 \mu)/2 \,, 
 \\
& \kappa := 3 \big(\mu_1 + \rho^{-1} \big)+ 1\,,\qquad &
& \beta_1 := 6 \mu_1+  3 \rho^{-1} + 3 \, ,  \qquad &
& 0 < \rho < \frac{1 - 3 a}{C_1(1 + a)}\,, \label{def rho}
\end{alignat}
where $ \mu := \mu (\tau, \nu) $ is the ``loss of regularity" defined in Theorem \ref{thm:stima inverso approssimato} 
(see \eqref{stima inverso approssimato 1}) and $ C_1 $ is fixed below. 

\begin{theorem}\label{iterazione-non-lineare} 
{\bf (Nash-Moser)} Assume that $ f \in C^q $ with $ q > S := s_0 + \b_1 + \mu + 3 $. 
Let $ \tau \geq \nu + 2 $. Then there exist $ C_1 > \max \{ \mu_1 + \a, C_0 \} $
(where $ C_0 := C_0 (\tau, \nu) $  is the one in Theorem \ref{teoremadiriducibilita}),  
$ \delta_0 := \d_0 (\tau, \nu) > 0 $ such that, if
\begin{equation}\label{nash moser smallness condition}  
N_0^{C_1}  \e^{b_* + 1} \gamma^{-2}< \d_0\,,\quad \gamma:= \e^{2 + a} = \e^{2b} \,,\quad 
N_0 := (\e \gamma^{-1})^\rho\,,\quad b_* := 6 - 2 b \, , 
\end{equation}
then, for all $ n \geq 0 $: 

\begin{itemize}
\item[$({\cal P}1)_{n}$] 
there exists a function 
$(\fracchi_n, \zeta_n) : {\cal G}_n \subseteq \Omega_\e \to E_{n-1} \times \R^\nu$, 
$\omega \mapsto (\fracchi_n(\omega), \zeta_n(\omega))$,  
$ (\fracchi_0, \zeta_0) := 0 $, $ E_{-1} := \{ 0 \} $,   
satisfying $ | \zeta_n |^\Lipg \leq C \|{\cal F}(U_n) \|_{s_0}^\Lipg $, 
\begin{equation}\label{ansatz induttivi nell'iterazione}
\| \fracchi_n \|_{s_0 + \mu }^{{\rm Lip}(\gamma)} 
\leq C_* \e^{b_*} \gamma^{-1}\,, \quad 
\| {\cal F}(U_n)\|_{s_0 + \mu + 3}^{{\rm Lip}(\gamma)} \leq C_*\e^{b_*} \,, 
\end{equation}
where $U_n := (i_n, \zeta_n)$ with $i_n(\ph) = (\ph,0,0) + \fracchi_n(\ph)$.
The sets ${\cal G}_{n} $ are defined inductively by: 
$$
{\cal G}_{0} := \big\{\omega \in \Omega_\e \, :\, |\omega \cdot l| \geq  2 \gamma \langle l \rangle^{-\tau}, \,  \forall l \in \Z^\nu \setminus \{0\} \big\}\,, 
$$
\begin{equation}\label{def:Gn+1}
{\cal G}_{n+1} :=  
\Big\{ \omega  \in {\cal G}_{n} \, : \, |\ii \omega \cdot l + \mu_j^\infty ( i_n) -
\mu_k^\infty ( i_n )| \geq \frac{2\gamma_{n} |j^{3}-k^{3}|}{\left\langle l\right\rangle^{\tau}}, \, \forall j , k \in S^c \cup \{0\}, 
\, l \in \Z^{\nu} \Big\}\,,
\end{equation}
where $ \gamma_{n}:=\gamma (1 + 2^{-n}) $ and $\mu_j^\infty(\omega) := \mu_j^\infty(\omega, i_n(\omega)) $ 
are defined in \eqref{espressione autovalori} (and  $ \mu_0^\infty(\omega) = 0 $).

The differences $ 
\widehat {\mathfrak I}_n 
:=  {\mathfrak I}_n - {\mathfrak I}_{n - 1} $ 
(where we set $ \widehat \fracchi_0 := 0 $) is defined on $\mG_n$, and satisfy
\begin{equation}  \label{Hn}
\| \widehat {\mathfrak I}_1 \|_{ s_0 + \mu}^{\Lipg} \leq C_* \e^{b_*} \gamma^{-1} \, , \quad 
\| \widehat {\mathfrak I}_n \|_{ s_0 + \mu}^{\Lipg} \leq C_* \e^{b_*} \gamma^{-1} N_{n - 1}^{-\alpha_1} \, , \quad \forall n > 1 \, .
\end{equation}
\item[$({\cal P}2)_{n}$]   $ \| {\cal F}(U_n) \|_{ s_{0}}^{{\rm Lip}(\gamma)} \leq C_* \e^{b_*} N_{n - 1}^{- \alpha}$ 
where we set $N_{-1} := 1$.
\item[$({\cal P}3)_{n}$] \emph{(High norms).} 
\ $  \| \fracchi_n \|_{ s_{0}+ \beta_1}^{{\rm Lip}(\gamma)} \leq C_* \e^{b_*} \gamma^{-1}  N_{n - 1}^{\kappa} $ and  
$ \|{\cal F}(U_n ) \|_{ s_{0}+\beta_1}^{{\rm Lip}(\gamma)} \leq C_* \e^{b_*}   N_{n - 1}^{\kappa} $.

\item[$({\cal P}4)_{n}$] \emph{(Measure).} 
The measure of the ``Cantor-like" sets $ {\cal G}_n $ satisfies
\begin{equation}\label{Gmeasure}
 | \Omega_\e \setminus {\cal G}_0 | \leq C_* \e^{2(\nu - 1)} \g \, , \quad  
\big| {\cal G}_n \setminus {\cal G}_{n+1} \big|  \leq  C_* \e^{2(\nu - 1)} \g N_{n - 1}^{-1}  \, . 
\end{equation}
\end{itemize}
All the Lip norms are defined on $ {\cal G}_{n} $, namely  $\| \ \|_s^{{\rm Lip}(\gamma)} = \| \ \|_{s,\mG_n}^{{\rm Lip}(\gamma)} $.\end{theorem}

\begin{proof}
To simplify notations, in this proof we denote $\| \, \|^{{\rm Lip}(\g)}$ by $\| \, \|$. We first prove $({\cal P}1, 2, 3)_n$.

\smallskip

{\sc Step 1:} \emph{Proof of} $({\cal P}1, 2, 3)_0$. 
Recalling \eqref{operatorF} we have $ \| {\cal F}( U_0 ) \|_s = $ 
$ \| {\cal F}(\vphi, 0 , 0, 0 ) \|_s = \| X_P(\vphi, 0 , 0 ) \|_s \leq_s \e^{6 - 2b} $ by
\eqref{stima XP}. Hence (recall that $ b_* = 6 - 2 b $) the smallness conditions in
$({\cal P}1)_0$-$({\cal P}3)_0$ hold taking $ C_* := C_* (s_0 + \b_1) $ large enough.

\smallskip

{\sc Step 2:} \emph{Assume that $({\cal P}1,2,3)_n$ hold for some $n \geq 0$, and prove $({\cal P}1,2,3)_{n+1}$.}
By \eqref{nash moser smallness condition} and \eqref{def rho},
$$
N_0^{C_1} \e^{b_* + 1} \gamma^{-2} = N_0^{C_1} \e^{1-3a} = \e^{1 - 3 a - \rho C_1(1 + a)}  < \d_0 
$$
for $\e$ small enough, and the smallness condition \eqref{condizione-kam} holds. 
Moreover \eqref{ansatz induttivi nell'iterazione} imply  \eqref{ansatz 0}
(and so \eqref{ansatz delta}) and  Theorem \ref{inversione linearized normale} applies. 
Hence the operator $ {\cal L}_\omega := {\cal L}_\omega(\omega, i_n(\omega))$ 
defined in \eqref{cal L omega} is invertible for all $\omega \in {\cal G}_{n + 1}$ and 
the last estimate in \eqref{stima inverso linearizzato normale} holds. 
This means that the assumption \eqref{tame inverse} of Theorem \ref{thm:stima inverso approssimato} 
is verified with $\Omega_\infty = {\cal G}_{n + 1}$. 
By Theorem \ref{thm:stima inverso approssimato} there exists an approximate inverse 
${\bf T}_n(\omega) := {\bf T}_0 (\omega, i_n(\omega))$ of the linearized operator 
$L_n(\omega) := d_{i, \zeta} {\cal F}(\omega, i_n(\omega)) $, satisfying \eqref{stima inverso approssimato 1}. 
Thus, using also \eqref{nash moser smallness condition}, \eqref{smoothing-u1}, \eqref{ansatz induttivi nell'iterazione},
\begin{align}
\| {\bf T}_n g  \|_s 
& \leq_s \gamma^{-1} \big( \| g \|_{s + \mu} 
+ \e \gamma^{-1}\{\| \fracchi_n \|_{s + \mu} + 
\gamma^{-1} \|{\mathfrak I}_n \|_{s_0 + \mu}  \|{\cal F}(U_n) \|_{s + \mu} \} \| g \|_{s_0+ \mu}\big) \label{stima Tn} \\
\| {\bf T}_n  g \|_{s_0} 
& 
\leq_{s_0} \gamma^{-1} \| g \|_{s_0 + \mu}   \label{stima Tn norma bassa}
\end{align}
and, by \eqref{stima inverso approssimato 2}, 
using also \eqref{ansatz induttivi nell'iterazione}, 
\eqref{nash moser smallness condition}, \eqref{smoothing-u1}, 
\begin{align}
\| \big(L_n \circ  {\bf T}_n -I \big)  g \|_s 
& \leq_s \gamma^{-1}  \big( \| {\cal F}(U_n) \|_{s_0 + \mu} \| g\|_{s + \mu} + 
\| {\cal F}(U_n) \|_{s + \mu} \|  g \|_{s_0 + \mu}  \nonumber
\\
& \qquad + \e \gamma^{-1} \| \fracchi_n \|_{s + \mu} \| {\cal F}(U_n) \|_{s_0 + \mu} \| g \|_{s_0 + \mu} \big)\, , 
 \label{stima Tn inverso approssimato} \\
\label{stima Tn inverso approssimato norma bassa}
\| \big(L_n \circ  {\bf T}_n -I \big) g  \|_{s_0} 
& \leq_{s_0} \gamma^{-1} \| {\cal F}(U_n)\|_{s_0 + \mu} \| g \|_{s_0 + \mu}  \nonumber
\\
& \leq_{s_0}  \gamma^{-1} \big( \|\Pi_n  {\cal F}(U_n)\|_{s_0 + \mu}+  \|\Pi_n^\bot  {\cal F}(U_n)\|_{s_0 + \mu} \big) \|  g  \|_{s_0 + \mu}  \nonumber 
\\
&  \leq_{s_0} N_{n }^{\mu} \gamma^{-1} \big( \| {\cal F}(U_n)\|_{s_0} + N_{n}^{-\beta_1} \| {\cal F}(U_n)\|_{s_0 + \beta_1} \big) \|  g  \|_{s_0 + \mu}\, .
\end{align}
Then, for all $ \om \in {\cal G}_{n+1} $, $ n \geq 0 $, we define 
\begin{equation}\label{soluzioni approssimate}
U_{n + 1} := U_n + H_{n + 1}\,, \quad 
H_{n + 1} :=
( \widehat \fracchi_{n+1}, \widehat \zeta_{n+1}) :=  - {\wtilde \Pi}_{n } {\bf T}_n \Pi_{n } {\cal F}(U_n) 
\in E_n \times \R^\nu\, ,  
\end{equation}
where  $ {\wtilde \Pi}_n ( \fracchi , \zeta ) :=   ( \Pi_n \fracchi , \zeta ) $ 
with $ \Pi_n $ in \eqref{Pin def}.
Since $ L_n := d_{i,\zeta} {\cal F}(i_n) $,  we write  
$ {\cal F}(U_{n + 1}) =  {\cal F}(U_n) + L_n H_{n + 1} + Q_n $, where
\begin{equation}\label{def:Qn}
Q_n := Q(U_n, H_{n + 1}) \, , \quad 
Q (U_n, H)  :=  {\cal F}(U_n + H ) - {\cal F}(U_n) - L_n H \,, \quad 
H \in E_{n} \times \R^\nu . 
\end{equation}
Then, by the definition of $ H_{n+1} $ in \eqref{soluzioni approssimate}, and writing $ {\wtilde \Pi}_n^\bot (\fracchi, \zeta ) :=
(\Pi_n^\bot \fracchi, 0) $, we have 
\begin{align}
{\cal F}(U_{n + 1}) & = 
 {\cal F}(U_n) - L_n {\wtilde \Pi}_{n } {\bf T}_n \Pi_{n } {\cal F}(U_n) + Q_n = 
 {\cal F}(U_n) - L_n  {\bf T}_n \Pi_{n } {\cal F}(U_n) + L_n  {\wtilde \Pi}_n^\bot  {\bf T}_n \Pi_{n } {\cal F}(U_n)
 + Q_n \nonumber\\ 
& =  {\cal F}(U_n)  - \Pi_{n } L_n {\bf T}_n \Pi_{n }{\cal F}(U_n) 
+ ( L_n  {\wtilde \Pi}_n^\bot -  \Pi_n^\bot L_n ) {\bf T}_n \Pi_{n }{\cal F}(U_n) + Q_n \nonumber\\
 & = \Pi_{n }^\bot {\cal F}(U_n) + R_n + Q_n + Q_n'  
\label{relazione algebrica induttiva}
\end{align}
where 
\begin{equation}\label{Rn Q tilde n}
R_n := (L_n  {\wtilde \Pi}_n^\bot -  \Pi_n^\bot L_n) {\bf T}_n \Pi_{n }{\cal F}(U_n) \,,
\qquad 
Q_n' := - \Pi_{n } ( L_n {\bf T}_n - I) \Pi_{n } {\cal F}(U_n)\,.
\end{equation}

\begin{lemma}\label{lemma convergence}
Define 
\begin{equation}\label{riscalamenti nash moser}
w_n := \e \gamma^{-2}  \|{\cal F}(U_n) \|_{s_0}\,,\quad B_n := \e \gamma^{-1}\| \fracchi_n \|_{s_0 + \beta_1} + \e \gamma^{-2}  \|{\cal F}(U_n) \|_{s_0 + \beta_1} \,.
\end{equation}
Then there exists $ K :=  K( s_0, \b_1 ) > 0 $ such that,  for all $n \geq 0$, 
setting $ \mu_1 := 3 \mu + 9$ (see \eqref{costanti nash moser}), 
\begin{equation}\label{relazioni induttive}
w_{n + 1} \leq K N_{n }^{\mu_1 + \frac{1}{\rho} - \beta_1} B_n +  K  N_n^{\mu_1} w_n^2\,,\qquad 
B_{n + 1} \leq K N_{n }^{\mu_1 + \frac{1}{\rho}}  B_n\, . 
\end{equation}
\end{lemma}
\begin{proof}
We  estimate separately the terms $ Q_n $ in \eqref{def:Qn} and $ Q_n' , R_n $ in \eqref{Rn Q tilde n}. 
\\[1mm]
{\it Estimate of $ Q_n $.}
By  \eqref{def:Qn}, \eqref{operatorF}, \eqref{parte quadratica da P}
and \eqref{ansatz induttivi nell'iterazione}, \eqref{smoothing-u1}, 
we have the quadratic estimates
\begin{align}
\label{stima parte quadratica norma alta}
\| Q(U_n, H)\|_s & \leq_s \e \big(\| \widehat \fracchi \|_{s + 3} \|  \widehat \fracchi \|_{s_0 + 3} + \| \fracchi_n \|_{s + 3} 
\|  \widehat \fracchi \|_{s_0 + 3}^2 \big) \\
\label{stima parte quadratica norma bassa}
\| Q(U_n, H) \|_{s_0} & \leq_{s_0}\e N_n^6 \|  \widehat \fracchi \|_{s_0}^2\,, \quad \forall  \widehat \fracchi \in E_n   \, . 
\end{align}
Now by the definition of $H_{n + 1}$ in \eqref{soluzioni approssimate} and \eqref{smoothing-u1}, 
\eqref{stima Tn}, \eqref{stima Tn norma bassa}, \eqref{ansatz induttivi nell'iterazione},   we get 
\begin{align}
\|  \widehat \fracchi_{n + 1} \|_{s_0 + \beta_1} 
& \leq_{s_0 + \beta_1} N_n^{\mu} \big( \gamma^{-1} \|{\cal F}(U_n) \|_{s_0 + \beta_1} + \e \gamma^{-2} \| {\cal F}(U_n)\|_{s_0 + \mu} \{\| \fracchi_n\|_{s_0 + \beta_1} + \gamma^{-1} \| {\cal F}(U_n)\|_{s_0 + \beta_1} \} \big)
\nonumber \\
& \leq_{s_0 + \beta} N_n^{\mu} \big( \gamma^{-1}\| {\cal F}(U_n) \|_{s_0 + \beta_1}  
+ \| \fracchi_n \|_{s_0 + \beta_1}\big)\, ,  \label{H n+1 alta} 
\\
\label{H n+1 bassa}
\|  \widehat \fracchi_{n + 1}\|_{s_0} 
& \leq_{s_0} \gamma^{-1}N_{n}^\mu \| {\cal F}(U_n)\|_{s_0} \, .
\end{align}
Then the term $ Q_n $ in \eqref{def:Qn} satisfies, 
by \eqref{stima parte quadratica norma alta}, \eqref{stima parte quadratica norma bassa},  
\eqref{H n+1 alta}, \eqref{H n+1 bassa},  
\eqref{nash moser smallness condition},  \eqref{ansatz induttivi nell'iterazione}, 
$ ({\cal P}2)_n $, \eqref{costanti nash moser},   
\begin{align} 
\| Q_n \|_{s_0 + \beta_1} 
& \leq_{s_0 + \beta_1}
 N_n^{2 \mu + 9}  \g \big( \g^{-1} \| {\cal F}(U_n)\|_{s_0 + \beta_1}  + 
\| \fracchi_n \|_{s_0 + \beta_1} \big) \, ,  \label{Qn norma alta} \\
\| Q_n \|_{s_0} 
&  \leq_{s_0}  N_n^{2 \mu + 6 } \e \gamma^{-2} \| {\cal F}(U_n) \|_{s_0}^2\, . \label{Qn norma bassa}
\end{align}
{\it Estimate of $ Q_n' $.}
The bounds \eqref{stima Tn inverso approssimato}, \eqref{stima Tn inverso approssimato norma bassa},
\eqref{smoothing-u1}, \eqref{costanti nash moser}, \eqref{ansatz induttivi nell'iterazione} imply
\begin{align}
\label{Qn' norma alta}
\| Q_n' \|_{s_0 + \beta_1} &  \leq_{s_0 + \beta_1} N_n^{2 \mu} \big(
\| {\cal F}(U_n) \|_{s_0 + \beta_1} + \| \fracchi_n \|_{s_0 + \beta_1} \| {\cal F}(U_n)\|_{s_0} \big) \, , \\
\label{Qn' norma bassa}
\| Q_n'\|_{s_0} & \leq_{s_0} \gamma^{-1} N_{n }^{2\mu }\big(\| {\cal F}(U_n) \|_{s_0} + N_{n}^{- \beta_1} \|{\cal F}(U_n)\|_{s_0 + \beta_1} \big) \| {\cal F}(U_n) \|_{s_0}\,.
\end{align}
{\it Estimate of $ R_n $.} For  $ H := (\widehat \fracchi, \widehat \zeta ) $ we have 
$ (L_n  {\wtilde \Pi}_n^\bot -  \Pi_n^\bot L_n) H = $ $ [{\bar D}_n, \Pi_n^\bot ] \widehat \fracchi = $
$ [\Pi_n  , {\bar D}_n] \widehat \fracchi $ where $ {\bar D}_n := d_i X_{H_\e}(i_n) + (0,0, \partial_{xxx} ) $. 
Thus Lemma \ref{lemma quantitativo forma normale},
\eqref{ansatz induttivi nell'iterazione}, \eqref{smoothing-u1} and \eqref{tame commutatori}  imply 
\begin{align}\label{stima commutatore modi alti norma bassa}
\| (L_n  {\wtilde \Pi}_n^\bot -  \Pi_n^\bot L_n) H \|_{s_0} & \leq_{s_0+ \b_1} 
\e N_{n }^{- \beta_1 + \mu + 3} \big(\|  \widehat \fracchi \|_{s_0 + \beta_1 - \mu} + 
\| \fracchi_n \|_{s_0 + \beta_1 - \mu} \|  \widehat \fracchi \|_{s_0 + 3}\big)\,, \\
\label{stima commutatore modi alti norma alta}
 \| (L_n  {\wtilde \Pi}_n^\bot -  \Pi_n^\bot L_n) H \|_{s_0 + \beta_1} & 
\leq_s 
\e N_n^{\mu + 3} \big(\|  \widehat \fracchi \|_{s_0 + \beta_1 - \mu } + \| \fracchi_n \|_{s_0 + \beta_1 - \mu } 
\|  \widehat \fracchi \|_{s_0 + 3} \big)\,.
\end{align}
Hence, applying 
\eqref{stima Tn}, 
\eqref{stima commutatore modi alti norma bassa}, \eqref{stima commutatore modi alti norma alta},  
\eqref{nash moser smallness condition}, 
\eqref{ansatz induttivi nell'iterazione}, 
\eqref{smoothing-u1}, 
the term $R_n$ defined in \eqref{Rn Q tilde n} satisfies
\begin{align} \label{stima Rn norma bassa}
\| R_n\|_{s_0} 
& \leq_{s_0 + \beta_1}  N_n^{ \mu + 6 - \beta_1} ( \e \gamma^{-1} \| {\cal F}(U_n)\|_{s_0 + \beta_1} + \e \| \fracchi_n  \|_{s_0 + \beta_1} )\,, 
\\
\label{stima Rn norma alta}
\| R_n \|_{s_0 + \beta_1} 
& \leq_{s_0 + \beta_1} N_n^{ \mu + 6} ( \e \gamma^{-1} \| {\cal F}(U_n)\|_{s_0 + \beta_1} + \e \| \fracchi_n  \|_{s_0 + \beta_1} )\,.
\end{align}
{\it Estimate of $ {\cal F}(U_{n + 1}) $.} 
By \eqref{relazione algebrica induttiva} and \eqref{Qn norma alta}, 
 \eqref{Qn norma bassa}, \eqref{Qn' norma alta}, \eqref{Qn' norma bassa}, 
\eqref{stima Rn norma bassa}, \eqref{stima Rn norma alta},  \eqref{nash moser smallness condition},  
\eqref{ansatz induttivi nell'iterazione}, 
we get
\begin{align}\label{F(U n+1) norma bassa}
& \| {\cal F}(U_{n + 1})\|_{s_0} \leq_{s_0 + \beta_1}   N_{n }^{\mu_1 - \beta_1} (  \e \gamma^{-1}\| {\cal F}(U_n)\|_{s_0 + \beta_1} +  \e \| \fracchi_n \|_{s_0 + \beta_1} ) 
+ N_n^{\mu_1} \e \gamma^{-2} \| {\cal F}(U_n)\|_{s_0}^2\,, 
\\
\label{F(U n+1) norma alta}
& \| {\cal F}(U_{n + 1}) \|_{s_0 + \beta_1} 
\leq_{s_0 + \beta_1}  N_n^{\mu_1} ( \e \gamma^{-1}\| {\cal F}(U_n) \|_{s_0 + \beta_1} + \e \| \fracchi_n \|_{s_0 + \beta_1} ) \,,
\end{align}
where $\mu_1 := 3 \mu + 9$.

\noindent 
{\it Estimate of $ \fracchi_{n+1} $.} 
Using  \eqref{H n+1 alta}  the term $ \fracchi_{n+1} = \fracchi_n + \widehat {\mathfrak I}_{n+1} $ is bounded by
\begin{equation}\label{U n+1 alta}
\| \fracchi_{n + 1}\|_{s_0 + \beta_1} \leq_{s_0 + \beta_1} 
N_n^\mu ( \| \fracchi_n\|_{s_0 + \beta_1} +\gamma^{-1} \| {\cal F}(U_n)\|_{s_0 + \beta_1} )\, . 
\end{equation}
Finally, recalling  \eqref{riscalamenti nash moser}, the  inequalities \eqref{relazioni induttive} follow by 
\eqref{F(U n+1) norma bassa}-\eqref{U n+1 alta}, \eqref{ansatz induttivi nell'iterazione} and 
$ \e \gamma^{-1} = N_0^{1/\rho} \leq N_n^{1/\rho}$.
\end{proof} 

\emph{Proof of $({\cal P}3)_{n + 1}$}. 
By \eqref{relazioni induttive} and $ ({\cal P}3)_n $, 
\begin{equation}
B_{n + 1} 
\leq K N_{n}^{\mu_1 + \frac{1}{\rho}} B_n 
\leq 2 C_* K \e^{b_* + 1} \gamma^{-2} N_{n}^{\mu_1+ \frac{1}{\rho}} N_{n-1}^\kappa \leq C_* \e^{b_* + 1} \gamma^{-2} N_n^{\kappa} \,,  
\label{stima B n+1} 
\end{equation}
provided $ 2 K N_n^{\mu_1 + \frac{1}{\rho} - \kappa} N_{n-1}^\kappa \leq 1 $, $  \forall n \geq 0 $. 
This inequality  holds by \eqref{def rho}, taking $N_0$ large enough (i.e $\e$ small enough). By \eqref{riscalamenti nash moser}, the bound $B_{n + 1} \leq C_* \e^{b_* + 1} \gamma^{-2} N_n^{\kappa}$ implies $({\cal P}3)_{n + 1}$.

\smallskip

\emph{Proof of $({\cal P}2)_{n + 1}$}. 
Using \eqref{relazioni induttive}, \eqref{riscalamenti nash moser} and $  ({\cal P}2)_n, ({\cal P}3)_n $, we get
\begin{align*}
w_{n + 1}  
& \leq K N_n^{\mu_1+ \frac{1}{\rho} - \beta_1} B_n +  K N_{n }^{\mu_1} w_n^2   
\leq K N_n^{\mu_1 + \frac{1}{\rho} - \beta_1} 2 C_* \e^{b_* + 1} \gamma^{-2} N_{n - 1}^{\kappa} 
+ K N_{n}^{\mu_1} ( C_* \e^{b_* + 1} \gamma^{-2} N_{n - 1}^{- \alpha} )^2 
\end{align*}
which is $ \leq C_* \e^{b_* + 1} \gamma^{-2} N_n^{- \alpha} $ provided that 
\begin{equation} \label{provided 2}
4 K N_n^{\mu_1 + \frac{1}{\rho} - \b_1 + \a} N_{n-1}^\kappa \leq 1, 
\quad
2 K C_* \e^{b_*+1} \g^{-2} N_n^{\mu_1 + \a} N_{n-1}^{-2\a} \leq 1 \, , 
\quad \forall n \geq 0.
\end{equation}
The inequalities in \eqref{provided 2} hold by \eqref{costanti nash moser}-\eqref{def rho}, 
\eqref{nash moser smallness condition}, $C_1 > \mu_1 + \alpha$, taking $\d_0$ in \eqref{nash moser smallness condition} small enough. 
By \eqref{riscalamenti nash moser}, the inequality $w_{n + 1} \leq C_* \e^{b_* + 1} \gamma^{-2} N_n^{- \alpha}$ implies $({\cal P}2)_{n + 1}$.

\smallskip

\emph{Proof of $({\cal P}1)_{n + 1}$}.
The bound \eqref{Hn} for $  \widehat \fracchi_1 $ 
follows by \eqref{soluzioni approssimate}, \eqref{stima Tn} (for $ s = s_0 + \mu $) and 
$  \| {\cal F} ( U_0 ) \|_{s_0 + 2 \mu} = $ $  \| {\cal F}(\vphi, 0, 0, 0) \|_{s_0 + 2 \mu} \leq_{s_0+2\mu } \e^{b_*} $.
The bound \eqref{Hn} for $ \widehat \fracchi_{n + 1} $ follows by \eqref{smoothing-u1}, 
\eqref{H n+1 bassa}, $({\cal P}2)_n $, \eqref{costanti nash moser}. 
It remains to prove that \eqref{ansatz induttivi nell'iterazione} holds at the step $n + 1$.
We have
\begin{equation}\label{W n+1 norma bassa}
\| \fracchi_{n + 1} \|_{s_0 + \mu} 
\leq {\mathop \sum}_{k = 1}^{n + 1} \| \widehat {\mathfrak I}_k \|_{s_0 + \mu} 
\leq C_* \e^{b_*}\gamma^{-1} {\mathop \sum}_{k \geq 1} N_{k - 1}^{- \alpha_1} 
\leq C_*  \e^{b_*} \gamma^{-1}
\end{equation}
for $ N_0 $ large enough, i.e. $ \e $ small. 
Moreover, using  \eqref{smoothing-u1}, (${\cal P}2)_{n +1}$, (${\cal P}3)_{n + 1}$, \eqref{costanti nash moser}, we get 
\begin{align*}
\| {\cal F}(U_{n + 1}) \|_{s_0 + \mu+ 3} 
& \leq  N_n^{\mu + 3} \|{\cal F}(U_{n + 1}) \|_{s_0} + N_{n}^{\mu + 3 - \beta_1} \| {\cal F}(U_{n + 1}) \|_{s_0 + \beta_1} 
\\
& \leq C_* \e^{b_*} N_n^{\mu + 3- \alpha} + C_* \e^{b_*} N_n^{\mu + 3 - \beta_1 + \kappa}
\leq  C_* \e^{b_*}\,,
\end{align*}
which is the second inequality in \eqref{ansatz induttivi nell'iterazione} at the step $n + 1$.
The bound $ | \zeta_{n+1} |^\Lipg \leq C \|{\cal F}(U_{n+1}) \|_{s_0}^\Lipg $ is a consequence of Lemma 
\ref{zeta = 0} (it is not inductive).

\smallskip

{\sc Step 3:} \emph{Prove $({\cal P}4)_n$ for all $n \geq 0$.} 
For all $n \geq 0$, 
\begin{equation}\label{Gn inclusioni}
{\cal G}_n \setminus {\cal G}_{n + 1} = \!\! \!\!\!  \bigcup_{\begin{subarray}{c}
l \in \Z^\nu, \, j , k \in S^c \cup\{ 0 \}
\end{subarray}} \!\! \! \!\!  R_{ljk}(i_n)
\end{equation}
where 
\begin{equation}\label{resonant sets}
R_{ljk}(i_n) := \big\{ \omega \in {\cal G}_n \, : \, |\ii \omega \cdot l + \mu_j^\infty (i_{n}) -
\mu_k^\infty (i_{n})| < 2\gamma_{n} |j^{3}-k^{3}|\left\langle l\right\rangle^{- \tau}\big\}\,.
\end{equation}
Notice that $R_{ljk}(i_n) = \emptyset$ if $j = k$, so that we suppose in the sequel that $j \neq k$.

\begin{lemma} \label{matteo 10}
For all $n \geq 1$, $|l| \leq N_{n - 1}$, the set $R_{ljk}(i_n) \subseteq R_{ljk}(i_{n - 1}) $. 
\end{lemma}

\begin{proof}
Like Lemma 5.2 in \cite{BBM} 
(with $\omega$ in the role of $\lambda \bar\omega$,  and $N_{n-1}$ instead of $N_n$).
\end{proof}

By definition, $ R_{ljk} (i_n) \subseteq {\cal G}_n $ (see \eqref{resonant sets})
and Lemma \ref{matteo 10} implies that, for all $ n \geq 1$, $ |l| \leq N_{n-1} $, the set
$R_{ljk}(i_n) \subseteq R_{ljk}(i_{n - 1}) $. On the other hand 
$ R_{ljk}(i_{n - 1})  \cap {\cal G}_{n} = \emptyset $ (see \eqref{def:Gn+1}). As a consequence,
for all $ |l| \leq N_{n-1} $, $ R_{ljk} (i_n) = \emptyset $ and, by \eqref{Gn inclusioni}, 
\begin{equation}\label{Gn Gn+1}
{\cal G}_n \setminus {\cal G}_{n+1}  \subseteq \!\!  \bigcup_{|l| > N_{n-1}, \, j, k \in S^c \cup \{0\}} 
\!\!\! \!\!\! R_{ljk} ( i_n)  \qquad \forall n \geq 1. 
\end{equation}

\begin{lemma}\label{matteo 4}
Let  $n \geq 0$. If $R_{ljk}(i_n) \neq \emptyset$  
then $|l| \geq C |j^3 - k^3| \geq \frac12 C (j^2 + k^2) $ for some $C > 0$.
\end{lemma}

\begin{proof}
Like Lemma 5.3 in \cite{BBM}. The only difference is that  $\omega$ is not constrained to a fixed  direction. 
Note also that $ | j^3- k^3 | \geq (j^2 + k^2) / 2  $, $ \forall j \neq k $. 
\end{proof}

By usual arguments (e.g. see Lemma 5.4 in \cite{BBM}), using Lemma \ref{matteo 4} and \eqref{autofinali} we 
have:

\begin{lemma} \label{lemma:risonanti}
For all $n \geq 0$, the measure $ |R_{ljk}(i_n)| \leq C \e^{2(\nu - 1)}  \gamma \langle l \rangle^{- \tau} $.
\end{lemma}

By \eqref{Gn inclusioni} and Lemmata \ref{matteo 4}, \ref{lemma:risonanti} we get
$$
|{\cal G}_0 \setminus {\cal G}_1 | \leq \sum_{l \in \Z^\nu, |j|, |k| \leq C |l|^{1/2}} 
| R_{ljk} ( i_0)| \leq \sum_{l \in \Z^\nu} \frac{C  \e^{2(\nu-1)} \g }{\langle l \rangle^{\tau-1}} \leq C' \e^{2(\nu-1)} \g \,.
$$
For $ n \geq 1$, by \eqref{Gn Gn+1}, 
$$
|{\cal G}_n \setminus {\cal G}_{n+1} | \leq \sum_{|l| > N_{n-1}, |j|, |k| \leq C |l|^{1/2}} 
| R_{ljk} ( i_n)| \leq \sum_{|l| > N_{n-1}} \frac{C  \e^{2(\nu-1)} \g }{\langle l \rangle^{\tau-1}} \leq C' \e^{2(\nu-1)} \g N_{n-1}^{-1} 
$$
because $\tau \geq \nu + 2$.
The estimate $ | \Omega_\e \setminus {\cal G}_0| \leq C \e^{2(\nu-1)} \g $ is elementary. 
Thus \eqref{Gmeasure} is proved. 
\end{proof}

\noindent
{\bf Proof of Theorem \ref{main theorem} concluded.}  
Theorem \ref{iterazione-non-lineare} implies that 
the sequence $ (\fracchi_n, \zeta_n) $ 
is well defined for $ \omega \in {\cal G}_\infty := \cap_{n \geq 0} {\cal G}_n $,  that 
$ \fracchi_n $  
is a Cauchy sequence in  $\| \ \|_{s_0 + \mu, {\cal G}_\infty}^{\Lipg}$, see 
\eqref{Hn}, and  $ |\zeta_n |^\Lipg  \to 0 $. Therefore $ \fracchi_n $ converges to a limit $ \fracchi_\infty $ 
in norm $\| \ \|_{s_0 + \mu, {\cal G}_\infty}^{\Lipg}$ and, 
by $ ({\cal P}2)_n $,  for all $\omega \in {\cal G}_\infty$, 
$ i_\infty (\vphi) :=  (\vphi,0,0) + \fracchi_\infty(\vphi)$,  
is a solution of 
$$
{\cal F}(i_\infty, 0 )= 0\, \quad \text{with} \quad 
 \| \fracchi_\infty \|^\Lipg_{s_0 + \mu,  {\cal G}_\infty} 
 \leq C \e^{6 - 2b} \gamma^{-1} 
$$
by \eqref{ansatz induttivi nell'iterazione} (recall that $b_* := 6 - 2b$). 
Therefore $\vphi \mapsto i_\infty(\vphi)$ is an invariant torus for the 
Hamiltonian vector field $ X_{H_\e} $ (see \eqref{hamiltoniana modificata}).
By  \eqref{Gmeasure}, 
$$
|\Omega_\e \setminus {\cal G}_\infty| 
\leq |\Omega_\e \setminus {\cal G}_0| + \sum_{n \geq 0} |{\cal G}_n \setminus {\cal G}_{n + 1}| 
\leq 2 C_* \e^{2(\nu - 1)} \g + C_* \e^{2(\nu - 1)} \g \sum_{n \geq 1} N_{n - 1}^{-1} 
\leq C \e^{2(\nu - 1)} \g \,. 
$$
The set $\Om_\e$ in \eqref{Omega epsilon} has measure $|\Omega_\e | = O( \e^{2 \nu} ) $.
Hence $|\Om_\e \setminus \mG_\infty| / |\Om_\e| \to 0$ as $\e \to 0$ 
because $ \g = o(\e^2)$, and therefore the measure of  $\mC_\e := \mG_{\infty}$
satisfies \eqref{stima in misura main theorem}.

In order to complete the proof of Theorem \ref{main theorem} we show the linear 
stability of the solution $ i_\infty (\om t ) $. 
By section \ref{costruzione dell'inverso approssimato} 
the system obtained linearizing the Hamiltonian vector field $ X_{H_\e }  $ 
at a quasi-periodic solution $ i_\infty (\om t ) $ 
is conjugated to the linear Hamiltonian system
\begin{equation}\label{sistema lineare dopo}
\begin{cases}
\dot \psi & \hspace{-6pt} = K_{20}(\om t) \eta + 
K_{11}^T (\om t ) w  
\\
\dot \eta & \hspace{-6pt} = 0 
\\
\dot w - \pa_x K_{0 2}(\om t ) w & \hspace{-6pt} = \partial_x  K_{11}(\om t) \eta   
\end{cases} 
\end{equation}
(recall that the torus $ i_\infty $ is isotropic and the transformed nonlinear Hamiltonian system is \eqref{sistema dopo trasformazione inverso approssimato} where  $ K_{00}, K_{10}, K_{01} = 0 $, 
see Remark \ref{rem:KAM normal form}). 
In section \ref{operatore linearizzato sui siti normali} we have proved the reducibility of the linear system
$ \dot w - \pa_x K_{0 2}(\om t ) w $, conjugating the last equation in \eqref{sistema lineare dopo}
to a diagonal system
\be\label{vjmuj}
{\dot v}_j +  \mu_j^\infty v_j = f_j (\om t)  \, , \quad j \in S^c \, , \quad \mu_j^\infty \in \ii \R \, , 
\ee
see \eqref{Lfinale}, and $ f (\vphi, x) = \sum_{j \in S^c} f_j (\vphi) e^{\ii j x } \in H^s_{S^\bot}
(\T^{\nu+1}) $. 
Thus \eqref{sistema lineare dopo} is stable.
Indeed the actions $ \eta (t) = \eta_0 \in \R $, $ \forall t \in \R $. 
Moreover the solutions of the  non-homogeneous equation \eqref{vjmuj} are
$$
v_j (t) = c_j e^{ \mu_j^\infty t} + {\tilde v}_j (t) \,, \quad \text{where} \quad 
{\tilde v}_j (t)  := \sum_{l \in \Z^\nu} \frac{f_{jl} \, e^{\ii \om \cdot l t } }{ \ii \om \cdot l + \mu_j^\infty } 
$$ 
is a quasi-periodic solution (recall that the first Melnikov conditions \eqref{prime di melnikov} hold at a solution).
As a consequence  (recall also $ \mu_j^\infty \in \ii \R  $)
 the Sobolev norm of the solution of \eqref{vjmuj} with initial condition $ v(0) = \sum_{j \in S^c} v_j (0) e^{\ii j x } 
 \in H^{s_0} (\T_x) $, $ s_0 < s $, does not increase in time.  \qed
\\[1mm]
{\bf Construction of the set $S$ of tangential sites.} 
We finally prove that, for any $\nu \geq 1$, the set $S$ in \eqref{tang sites} satisfying $({\mathtt S}1)$-$({\mathtt S}2)$ can be constructed inductively with only a \emph{finite} number of restriction at any step of the induction.

First, fix any integer $\bar\jmath_1 \geq 1$. 
Then the set $J_1 := \{ \pm \bar\jmath_1 \}$ trivially satisfies 
$({\mathtt S}1) $-$({\mathtt S}2)$.
Then, assume that we have fixed $n$ distinct positive integers $\bar\jmath_1, \ldots, \bar\jmath_n$, $n \geq 1$, such that the set $J_n := \{ \pm \bar\jmath_1, \ldots, \pm \bar\jmath_n \}$ satisfies $({\mathtt S}1)$-$({\mathtt S}2)$.
We describe how to choose another positive integer $\bar\jmath_{n+1}$, which is different from all $j \in J_n$, such that $J_{n+1} := J_n \cup \{ \pm \bar\jmath_{n+1} \}$ also 
satisfies $({\mathtt S}1), ({\mathtt S}2)$.

Let us begin with analyzing $({\mathtt S}1)$. 
A set of 3 elements $j_1, j_2, j_3 \in J_{n+1}$ can be of these types: 
$(i)$ all ``old'' elements $j_1, j_2, j_3 \in J_n$; 
$(ii)$ two ``old'' elements $j_1, j_2 \in J_n$ and one ``new'' element $j_3 = \s_3 \bar\jmath_{n+1}$, $\s_3 = \pm 1$;
$(iii)$ one ``old'' element $j_1 \in J_n$ and two ``new'' elements $j_2 = \s_2 \bar\jmath_{n+1}$,  $j_3 = \s_3 \bar\jmath_{n+1}$, with $\s_2, \s_3 = \pm 1$;
$(iv)$ all ``new'' elements $j_i = \s_i \bar\jmath_{n+1}$, $\s_i = \pm 1$, $i = 1,2,3$.

In case $(i)$, the sum $j_1 + j_2 + j_3$ is nonzero by inductive assumption. 
In case $(ii)$, $j_1 + j_2 + j_3$ is nonzero provided 
$\bar\jmath_{n+1} \notin \{ j_1 + j_2 : j_1, j_2 \in J_n \}$, which is a finite set. 
In case $(iii)$, for $\s_2 + \s_3 = 0$ the sum $j_1 + j_2 + j_3 = j_1$ is trivially nonzero because $0 \notin J_n$, while, for $\s_2 + \s_3 \neq 0$, the sum $j_1 + j_2 + j_3 = j_1 + (\s_2 + \s_3) \bar\jmath_{n+1} \neq 0$ if $\bar\jmath_{n+1} \notin \{ \tfrac12 j : j \in J_n \}$,
which is a finite set. 
In case $(iv)$, the sum $j_1 + j_2 + j_3 = (\s_1 + \s_2 + \s_3) \bar\jmath_{n+1} \neq 0$ because 
$\bar\jmath_{n+1} \geq 1$ and $\s_1 + \s_2 + \s_3 \in \{ \pm 1, \pm 3\}$. 

\smallskip

Now we study $({\mathtt S}2)$ for the set $J_{n+1}$. 
Denote, in short, $b := j_1^3 + j_2^3 + j_3^3 + j_4^3 - (j_1 + j_2 + j_3 + j_4)^3$.  

A set of 4 elements $j_1, j_2, j_3, j_4 \in J_{n+1}$ can be of 5 types: 
$(i)$ all ``old'' elements $j_1, j_2, j_3, j_4 \in J_n$; 
$(ii)$ three ``old'' elements $j_1, j_2, j_3 \in J_n$ and one ``new'' element $j_4 = \s_4 \bar\jmath_{n+1}$, $\s_4 = \pm 1$;
$(iii)$ two ``old'' element $j_1, j_2 \in J_n$ and two ``new'' elements $j_3 = \s_3 \bar\jmath_{n+1}$,  $j_4 = \s_4 \bar\jmath_{n+1}$, with $\s_3, \s_4 = \pm 1$;
$(iv)$ one ``old'' element $j_1 \in J_n$ and three ``new'' elements $j_i = \s_i \bar\jmath_{n+1}$, 
$\s_i = \pm 1$, $i=2,3,4$;
$(v)$ all ``new'' elements $j_i = \s_i \bar\jmath_{n+1}$, $\s_i = \pm 1$, $i = 1,2,3,4$. 

In case $(i)$, $b \neq 0$ 
by inductive assumption. 

In case $(ii)$, assume that $j_1 + j_2 + j_3 + j_4 \neq 0$, 
and calculate
\begin{align*}
b & = - 3 (j_1 + j_2 + j_3) \bar\jmath_{n+1}^2 
- 3 (j_1 + j_2 + j_3)^2 \s_4 \bar\jmath_{n+1} 
+ [j_1^3 + j_2^3 + j_3^3 - (j_1 + j_2 + j_3)^3] \ 
=: p_{j_1,j_2,j_3, \s_4}(\bar\jmath_{n+1}). 
\end{align*}
This is nonzero provided $p_{j_1,j_2,j_3, \s_4}(\bar\jmath_{n+1}) \neq 0$ for all $j_1, j_2, j_3 \in J_n$, $\s_4 = \pm 1$. 
The polynomial $p_{j_1,j_2,j_3, \s_4}$ is never identically zero because either the leading coefficient $-3(j_1 + j_2 + j_3) \neq 0$ (and, if one uses $(\mathtt{S}_3)$, this is always the case), or, if $j_1 + j_2 + j_3 = 0$, then $j_1^3 + j_2^3 + j_3^3 \neq 0$ 
by \eqref{prodottino} (using also that $0 \notin J_n$).

In case $(iii)$, assume that $ j_1 + \ldots + j_4  = j_1 + j_2 + (\s_3 + \s_4) \bar\jmath_{n+1} \neq 0$, and calculate
\begin{align*}
b & = - 3 \a \bar\jmath_{n+1}^3 
- 3 \a^2 (j_1 + j_2)  \bar\jmath_{n+1}^2
- 3 (j_1 + j_2)^2 \a \bar\jmath_{n+1} 
- j_1 j_2 (j_1 + j_2) 
=: q_{j_1,j_2,\a}(\bar\jmath_{n+1}), 
\end{align*}
where $\a := \s_3 + \s_4$. 
We impose that $q_{j_1,j_2,\a}(\bar\jmath_{n+1}) \neq 0$ for all $j_1, j_2  \in J_n$, 
$\a \in\{ \pm 2, 0 \}$. 
The polynomial $q_{j_1,j_2,\a}$ is never identically zero because either the leading coefficient $-3\a \neq 0$, or, for $\a = 0$, the constant term $- j_1 j_2 (j_1 + j_2) \neq 0$ (recall that $0 \notin J_n$ and $j_1 + j_2 + \a \bar\jmath_{n+1} \neq 0$). 

In case $(iv)$, assume that $j_1 + \ldots + j_4 = j_1 + \a \bar\jmath_{n+1} \neq 0$, 
where $\a := \s_2 + \s_3 + \s_4 \in \{ \pm 1, \pm 3 \}$, and calculate
\[
b = \a \bar\jmath_{n+1} r_{j_1,\a}(\bar\jmath_{n+1}), 
\quad
r_{j_1, \a}(x) := (1-\a^2) x^2 - 3 \a j_1 x - 3 j_1^2.
\]
The polynomial $r_{j_1, \a}$ is never identically zero because $j_1 \neq 0$. 
We impose $r_{j_1, \a}(\bar\jmath_{n+1}) \neq 0$ for all $j_1 \in J_n$, $\a \in \{ \pm 1, \pm 3\}$.  

In case $(v)$, assume that $j_1 + \ldots + j_4 = \a \bar\jmath_{n+1} \neq 0$, with $\a := \s_1 + \ldots + \s_4 \neq 0$, and calculate 
$b = \a (1 - \a^2) \bar\jmath_{n+1}^3$.
This is nonzero because $\bar\jmath_{n+1} \geq 1$ and $\a \in \{\pm 2, \pm 4\}$. 

We have proved that, in choosing $\bar\jmath_{n+1}$, there are only  finitely many integers to avoid.   

\noindent
Pietro Baldi, Dipartimento di Matematica e Applicazioni ``R. Caccioppoli'',
Universit\`a di Napoli Federico II,  Via Cintia, Monte S. Angelo, 
80126, Napoli, Italy,  
{\tt pietro.baldi@unina.it}.

\smallskip

\noindent
Massimiliano Berti, Riccardo Montalto,  
SISSA,  Via Bonomea 265, 34136, Trieste, Italy, 
{\tt berti@sissa.it}, 
{\tt riccardo.montalto@sissa.it}.

%


\begin{thebibliography}{10}

\bibitem{Baldi Benj-Ono}
Baldi P., 
\emph{Periodic solutions of fully nonlinear autonomous equations of Benjamin-Ono type}, 
Ann. Inst. H. Poincar\'e (C) Anal. Non Lin\'eaire 30 (2013), 33-77.

\bibitem{BBM} Baldi P., Berti M., Montalto R., 
\emph{KAM for quasi-linear and fully nonlinear forced perturbations of Airy equation}, 
Math. Annalen, in print (published online 2014, DOI 10.1007/s00208-013-1001-7). 


\bibitem{BBiP1}
Berti M., Biasco P., Procesi M., {\it KAM theory for the Hamiltonian DNLW}, 
Ann. Sci. \'Ec. Norm. Sup\'er. (4),  
Vol. 46, fascicule 2 (2013), 301-373.

\bibitem{BBiP2}
Berti M., Biasco P., Procesi M., {\it KAM theory for the reversible derivative wave equation},
Arch. Rational Mech. Anal.,  212, (2014), 905-955. 

\bibitem{BB13JEMS} Berti M., Bolle P.,  {\it Quasi-periodic solutions 
with Sobolev regularity of NLS on $ \T^d $ with a multiplicative potential}, 
Eur. Jour. Math. 15 (2013), 229-286.

\bibitem{BB13} Berti M., Bolle P., {\it A Nash-Moser approach to KAM theory}, preprint 2014.

\bibitem{BB14} Berti M., Bolle P., {\it Quasi-periodic solutions 
for autonomous NLW on $ \T^d $ with a multiplicative potential}, in preparation. 




\bibitem{B96}  Bourgain J.,  {\it Gibbs measures and quasi-periodic solutions for nonlinear Hamiltonian partial differential equations}, 
23-43, Gelfand Math. Sem., Birkh\"auser Boston, Boston, MA, 1996.

\bibitem{B5}  Bourgain J.,  {\it Green's function estimates for lattice Schr\"odinger 
operators and applications}, Annals of Mathematics Studies 158, 
Princeton University Press, Princeton, 2005.

\bibitem{C} Craig W.,  {\it Probl\`emes de petits diviseurs dans
les \'equations aux d\'eriv\'ees partielles},
Panoramas et Synth\`eses, 9, 
Soci\'et\'e Math\'ematique de France, Paris, 2000. 

\bibitem{CW} Craig W., Wayne C.E., {\it Newton's method and periodic solutions
of nonlinear wave equation}, Comm. Pure  Appl. Math. 46, 1409-1498, 1993.

\bibitem{EK} Eliasson L.H., Kuksin S.,
{\it KAM for non-linear Schr\"odinger equation},  Annals of Math.,
172 (2010), 371-435.

\bibitem{GXY} Geng J., Xu X.,  You J., {\it An infinite dimensional KAM theorem and 
its application to the two dimensional cubic Schr\"odinger equation}, 
Adv.  Math. 226 (2011) 5361-5402.

\bibitem{K13} Huang G., Kuksin S.,   
{\it KdV equation under periodic boundary conditions and its perturbations}, preprint, 
http://arxiv.org/abs/1309.1597. 

\bibitem{IP09}  Iooss G.,   Plotnikov P.I.,  
{\it Small divisor problem in the theory of three-dimensional water gravity waves}, 
Mem. Amer. Math. Soc. 200, no. 940 (2009).

\bibitem{IP11}  Iooss G.,   Plotnikov P.I.,  
{\it Asymmetrical three-dimensional travelling gravity waves}, 
Arch. Rational Mech. Anal. 200  no. 3, (2011), 789-880. 

\bibitem{Ioo-Plo-Tol} Iooss G.,   Plotnikov P.I., Toland J.F., 
{\it Standing waves on an infinitely deep perfect fluid under gravity}, 
Arch. Rational Mech. Anal. 177  no. 3, (2005), 367-478. 

\bibitem{Lax} Lax P.,
{\it Development of singularities of solutions of nonlinear
hyperbolic partial differential equations}, J. Mathematical Phys. 5 (1964), 611-613.

\bibitem{LY}  Liu J., Yuan X.,
{\it A KAM Theorem for Hamiltonian Partial Differential
Equations with Unbounded Perturbations}, Comm. Math. Phys, 307 (3) (2011), 629-673.

\bibitem{KaP} Kappeler T., P\"{o}schel J., {\it KAM and KdV}, Springer, 2003.

\bibitem{KM} Klainerman S., Majda A.,
{\it Formation of singularities for wave equations including the
nonlinear vibrating string}, Comm. Pure Appl. Math., 33,  (1980), 241-263.

\bibitem{Ku}  Kuksin S.,
{\it Hamiltonian perturbations of infinite-dimensional linear
systems with imaginary spectrum},
Funktsional. Anal. i Prilozhen. 21, no. 3, 22--37, 95, 1987.

\bibitem{K2} Kuksin S., {\it A KAM theorem for equations of the Korteweg-de Vries type},
Rev. Math. Phys., 10, 3, (1998), 1-64.

\bibitem{k1}  Kuksin S.,
{\it Analysis of Hamiltonian PDEs},
Oxford Lecture Series in Mathematics and its Applications, 19.
Oxford University Press (2000).





\bibitem{Po2} P\"oschel J.,
{\it A KAM-Theorem for some nonlinear PDEs},  Ann. Sc. Norm. Pisa, 23, (1996) 119-148.

\bibitem{Po3} P\"oschel J., {\it Quasi-periodic solutions for
a nonlinear wave equation}, Comment. Math. Helv.,  71, no. 2, (1996)
269-296.

\bibitem{PP1}
Procesi M., Procesi C., 
{\it A normal form for the Schr\"odinger equation with analytic non-linearities}, Comm. Math. Phys. 312 (2012), 501-557. 

\bibitem{PP} Procesi C., Procesi M., {\it A KAM algorithm for the completely resonant nonlinear Schr\"odinger equation}, preprint (2013).

\bibitem{Taylor}  Taylor M. E., {\it Pseudodifferential Operators and Nonlinear PDEs}, 
Progress in Mathematics, Birkh\"auser, 1991.

\bibitem{Wang} 
Wang W. M., {\it Supercritical nonlinear Schr\"odinger equations I: quasi-periodic solutions},  preprint.

\bibitem{W1}  Wayne E., {\it Periodic and quasi-periodic solutions
of nonlinear wave equations via KAM theory}, Comm. Math. Phys. 127, 479-528, 1990.

\bibitem{ZGY} Zhang J.,  Gao M.,  Yuan X.
{\it KAM tori for reversible partial differential equations},
Nonlinearity 24 (2011), 1189-1228. 

\bibitem{Z1} Zehnder E., {\it Generalized implicit function theorems with applications to some small divisors problems
I-II}, Comm. Pure Appl. Math. 28 (1975), 91-140, and 29 (1976), 49-113.


\end{thebibliography}
\end{document}